\documentclass[11pt]{amsart}
\usepackage{graphicx}
\usepackage{amssymb,amsmath,color}
\usepackage{hyperref}

\newcommand{\e}{\varepsilon}

\newcommand{\curl}{\operatorname{curl}}
\newcommand{\dist}{\operatorname{dist}}

\newcommand{\p}{\partial}
\newcommand{\C}{\mathbb{C}}
\newcommand{\R}{\mathbb{R}}

\newcommand{\A}{\mathcal{A}}

\newcommand{\D}{\mathcal{D}}

\newcommand{\ot}{\overline{t}}

\newcommand{\logep}{\left| \log \e \right|}
\newcommand{\ilogep}{{\frac{1}{\logep}}}

\newcommand{\Wdot}{{{\dot W}^{-1,1}}}

\newcommand{\Omegab}{\overline{\Omega}}

\newcommand{\norm}[1]{\left\|#1\right\|}

\newcommand{\LN}{\left\|}
\newcommand{\RN}{\right\|}
\newcommand{\LV}{\left|}
\newcommand{\RV}{\right|}
\newcommand{\LC}{\left(}
\newcommand{\RC}{\right)}
\newcommand{\LB}{\left[}
\newcommand{\RB}{\right]}
\newcommand{\LCB}{\left\{}
\newcommand{\RCB}{\right\}}
\newcommand{\LA}{\left<}
\newcommand{\RA}{\right>}

\numberwithin{equation}{section}
\newtheorem{theorem}{Theorem}[section]
\newtheorem{proposition}[theorem]{Proposition}
\newtheorem{corollary}[theorem]{Corollary}
\newtheorem{lemma}[theorem]{Lemma}
\newtheorem{definition}[theorem]{Definition}

\theoremstyle{remark}

\begin{document}

\title[Vortex liquids]{Vortex liquids and the Ginzburg-Landau equation}
\author[M. Kurzke]{Matthias Kurzke}
\address{School of Mathematical Sciences\\
University of Nottingham\\
Nottingham NG7 2RD\\
UK
}
\email{matthias.kurzke@nottingham.ac.uk}
\author[D. Spirn]{Daniel Spirn}
\address{School of Mathematics\\
University of Minnesota\\
Minneapolis, MN 55455\\
USA}
\email{spirn@math.umn.edu}
\date{\today}

\thanks{M.K. was supported in part by DFG SFB 611.
 D.S. was supported in part by NSF grants DMS-0707714 and DMS-0955687.}

\setcounter{tocdepth}{1}

\begin{abstract}
We establish vortex dynamics for the time-dependent Ginzburg-Landau equation for asymptotically large numbers of vortices for the problem without a gauge field and either Dirichlet or Neumann boundary conditions.  As our main tool, we establish quantitative bounds on several fundamental quantities, including the kinetic energy, that lead to explicit convergence rates.  
For  dilute vortex liquids we prove that sequences of solutions converge to the hydrodynamic limit.
\end{abstract}

\maketitle

{
\tableofcontents
}

\section{Introduction}
\label{sec:intro}
Let $u: [0,\infty)\times\Omega \to \mathbb{C}$ satisfy the scaled Ginzburg-Landau equation
\begin{equation} \label{tdgl}
{1\over \logep} \p_t u = \Delta u + {1\over \e^2} u \LC 1 - |u|^2 \RC
\end{equation}
with  either  Dirichlet boundary conditions 
\begin{align}
\label{DirichletBC}  u & = e^{i n\theta + i \varphi_\star} \hbox { on } \p \Omega
\end{align}
with $\varphi_\star \in C^2$, $\int_{\p \Omega} \p_\tau \varphi_\star = 0$, so 
$\deg(u;\p\Omega)=n$,
 or Neumann boundary conditions
\begin{align}
\label{NeumannBC}  \p_\nu u & = 0  \hbox { on } \p \Omega.
\end{align}
We take $\Omega$ to be a smooth, simply connected domain {in $\R^2$} containing the origin.
Equation \eqref{tdgl} models the dynamic behavior of superconductors when the electromagnetic field potential is absent.   When a gauge field is present, the corresponding Gorkov-Eliashberg equations 
\begin{equation}\label{ge}
\begin{split}
\p_\Phi u & = \nabla_A^2 u + {1\over \e^2} u \LC 1 - |u|^2 \RC \\
E & = - \curl \curl A + j_A(u) ,
\end{split}\end{equation}
where $\p_\Phi = \p_t + i \Phi$, $E = \p_t A + \nabla \Phi$, and $j_A(u) = \LC i u, \nabla_A u \RC$, provide a more complete model of superconductivity.  

In order to describe the behavior of solutions of \eqref{tdgl} with small $\e$ we define 
 some fundamental quantities including:
\begin{align*}
\hbox{ energy density } \quad & e_\e(u) = {1\over 2 } \LV  \nabla u \RV^2 + {1\over 4\e^2} \LC 1 - |u|^2 \RC^2   \\   
\hbox{ supercurrent } \quad & j(u)  = \LC i u , \nabla u \RC   \\ 
\hbox{ vorticity/Jacobian } \quad & J(u)  = \det \nabla u  = {1\over 2 } \curl j(u) .  
\end{align*}
Here  $(\cdot,\cdot)$ denotes the real scalar product of two complex 
numbers, so $(a,b)=\frac12(\bar a b + a \bar b)$ for $a,b\in \C$.  
Solutions to equation \eqref{tdgl} diffuse the Ginzburg-Landau energy 
\begin{equation} \label{bbhenergy}
E_\e(u) = \int_\Omega e_\e(u) 
\end{equation}
via the identity
\begin{equation} \label{PDEenergyidentity}
 E_\e(u(t)) +  \int_0^t \int_\Omega {\LV \p_t u \RV^2 \over \logep} = E_\e(u(0)).
\end{equation}

\subsection{Vortex dynamics and vortex liquids}
\label{subsec:introliquids}
A prominent feature of type II superconductivity is the presence of localized regions, called \emph{vortices},  where superconductivity vanishes.  In particular there exist some points $\{a_j \}_{j=1}^n$ in $\Omega$ where $|u(a_j)| = 0$.  Furthermore, about each vortex the winding number of the phase is \emph{quantized}; in particular 
$${1\over 2 \pi} \int_{\p B_r (a_j)} \tau \cdot j(u) \approx d \in \mathbb{Z} \backslash \{0\}. $$  
In the vicinity of each vortex  the Ginzburg-Landau energy $E_\e(u)$ blows up at the rate $\pi \logep + O(1)$.  Bethuel-Brezis-H\'elein showed in \cite{BBH} that  minimizers of the  Ginzburg-Landau energy \eqref{bbhenergy}
can be expanded further up to second order 
\begin{equation} \label{energyexpansion2ndorder}  
E_\e(u) = n \LC \pi \logep + \gamma \RC + W(a) + o(1),
\end{equation}
where $\gamma$ is a universal constant and 
\begin{equation}\label{expansionW2ndorder}    
W(a) = - \pi \sum_{i \neq j} \log |a_i - a_j | + \hbox{ boundary effects}
\end{equation}
is a \emph{renormalized energy} and the winding number about each vortex is one.   
We use the shorthand
 $a = (a_1,...,a_n) \in \Omega^n$ for a collection of $n$ points in $\Omega$ here.
 
This renormalized energy is precisely the bounded domain version
of the \emph{Kirchhoff-Onsager} functional that arises in two dimensional incompressible Euler equations and other settings.  
The renormalized energy will be discussed in more detail in 
Sections~\ref{sec:renen} and~\ref{sec:dirichletrenen}.
From back-of-the-envelope calculations one finds that $J(u)$ is quantized and looks like a sum of integer-weighted delta functions; and so, for small $\e$ one finds that 
\[
J(u) \approx {e_\e(u) \over \logep} \approx \pi \sum_{j=1}^n \delta_{a_j}
\] 
in the case when the winding number about each vortex equals one, 
and as $\e \to 0$, $u$  limits to
\begin{equation*}  \label{CHMformallimit}
u_\star = \prod_{j=1}^n { x - a_j \over |x - a_j|} e^{i \psi_\star}
\end{equation*}
where $\psi_\star$ is $H^1(\Omega)$.  This $u_\star$ is referred to as the \emph{canonical harmonic map} when $\psi_\star$ is a harmonic function.  This limiting behavior was established in many situations, see for example \cite{BBH, LinLin, Sandier, JerrardSIMA, JerrardSoner}.

When dynamics  \eqref{tdgl} are turned on, these vortices move according to
the gradient flow of the Kirchhoff-Onsager energy:
\begin{equation}  \label{ODE}
\dot{a}_j = - {1\over \pi}  \nabla_{a_j}  W.
\end{equation}
The $\logep$ factor in front of \eqref{tdgl} is the critical time scale on which vortices will move and can be thought of as the length of time  it takes the unscaled time dependent Ginzburg-Landau equation to move an $O(\logep)$ amount of energy an $O(1)$ distance.  
That vortices satisfy \eqref{ODE} in the limit was the subject of a formal asymptotic study by E \cite{EVortex}.  Later,
 arguments of Lin \cite{LinHeat} and Jerrard-Soner \cite{JSparabolic} provided rigorous justification of the limit.  
Both \cite{LinHeat} and \cite{JSparabolic} assume that the number of 
vortices is uniformly bounded as $\e\to 0$.
The limit equation \eqref{ODE} is the gradient flow of $W$ just as 
\eqref{tdgl} is the (rescaled) gradient flow of the integrated 
energy density $\int_\Omega e_\e(u)$.  The similarity in structure can also be seen by the energy dissipation identity
\begin{equation}\label{ODEenergyidentity}
W(a(t)) + \pi\int_0^t \LV \dot{a}(s) \RV^2  = W(a(0)).
\end{equation}
 This structure was exploited to 
give a more abstract proof of the motion law
by Sandier-Serfaty \cite{SandSerf} in their $\Gamma$-convergence of gradient flows framework.  

In recent years there have been significant advances in understanding
 the dynamics of a  finite numbers of vortices   by Bethuel-Orlandi-Smets \cite{BOS} on $\R^2$ and  by Serfaty \cite{SerfatyCollision} on bounded domains.  These results allow for much weaker initial conditions, handle collisions of plus/minus vortices, and describe the dynamical behavior of higher degree vortices.  



On the other hand, the behavior of the time dependent Ginzburg-Landau equations with asymptotically large numbers of vortices
has  received mostly formal treatment.   The question of how large numbers of vortices behave in superconductors is important from both  experimental and  numerical perspectives.  In the former,  typical superconductors contain many millions of  vortices per sample \cite{Chapman, Essmann} so the large vortex problem is a fundamental feature of high $T_C$ superconducting devices.    In  the latter, point vortex methods provide a useful class of numerical algorithms for simulating challenging PDE's, like vortex sheets; hence, \eqref{ODE} is a reasonable numerical 
approximation of the limiting mean field equation with vortex sheet initial data.
 
In \cite{ELiquid} E looks at how the  analogue of \eqref{ODE} 
on $\R^2$ behaves in  a mean field sense as $n \to \infty$.  Defining the vortex density function $\omega_n = {1\over n} \sum_{j=1}^n \delta_{a_j(t)}$, the author shows that the limiting density, $\omega = \lim_{n\to\infty} \omega_n$, formally satisfies a weak PDE  of the form
\begin{align*}
\p_{{t}} \omega + \operatorname{div} ( \omega v) & = 0 \\
v & = \nabla \LC\Delta^{-1} \RC \omega
\end{align*}
after rescaling time $t$.
Subsequently, this ODE limit on $\R^2$ was   rigorously established by Lin-Zhang
\cite{LinZhang}.  

There are many similarities between this ODE limit problem and ODE limit problem arising from the point vortex method for the Euler equations.  In the latter case it was
 shown by Schochet \cite{Schochet}, and later by Liu-Xin \cite{LiuXin}, that the vortex density function for Euler point vortices on $\R^2$, which follow the Kirchhoff law 
\[
\dot{a}_j = - {1\over \pi} \nabla^\perp_{a_j} W(a),
\]
{limits} to a weak Delort solution to the incompressible Euler equations on $\R^2$.  Due to the similarities of the two problems,
 Lin-Zhang \cite{LinZhang} used the approach of \cite{LiuXin} to prove the associated hydrodynamic limit of the ODE \eqref{ODE} on $\R^2$.

The present work is the first to directly couple the Ginzburg-Landau equation
to a mean field PDE. All previous works either prove a PDE to ODE limit
for a finite number of vortices or pass from the ODE to the mean field PDE limit. 
Our quantitative results enable us to take the diagonal limit in a rigorous way.
\begin{figure}[htbp] 
   \centering
   \includegraphics[width=4in]{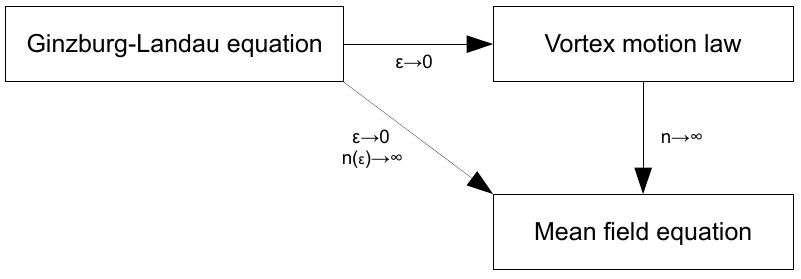} 
   \caption{Limiting from the Ginzburg-Landau equation directly to the mean field equation}
   \label{fig:example}
\end{figure}

In order to make the direct connection between the Ginzburg-Landau equation and the limiting mean field equation, it is necessary to establish two steps.  The first of which entails a proof that \eqref{tdgl} can accept asymptotically large numbers of vortices for long-enough times.  The second step involves coupling these Ginzburg-Landau solutions to an appropriate hydrodynamic limit of \eqref{ODE} on bounded domains.

\subsection{Results}
 In the following we let 
\[
A \lesssim B \quad \hbox{ if } A \leq C B
\]
for some $C$ that depends only on $\Omega$ and {$\varphi_\star$}.   
We tacitly assume that $\e$ is small enough that we 
can use estimates of the type $|\log\logep| \lesssim \logep$.

We define the \emph{excess energy} 
 $$
 D(a(t)) = E_\e(u(t) - \LB n \LC \pi \logep + \gamma \RC + W(a(t)) \RB,
 $$
{where  $\gamma$ and $W(a(t))$ are defined in \eqref{energyexpansion2ndorder} and \eqref{expansionW2ndorder}; the excess energy}
 will be used to control the deviation of the vortex path from the path defined by  the ODE \eqref{ODE}.  
 We also define 
\[
\rho_a = {1\over 4} \{ \min_{j \neq k } \LV a_j  - a_k \RV , \min_j \operatorname{dist} ( a_j, \p \Omega) \}.
\]
as a measure of how close vortices are to each other or the boundary.
We choose a number $\rho_\star$ with $0< \rho_\star< \rho_{a(0)}$. This 
defines a time scale 
\[
\tau_0 = \inf \left\{t >0 \hbox { such that }
{\rho_{a(t)}}\le \rho_\star \right\}>0
\]
on which vortices will stay well-separated.
For $\rho\le\rho_\star$ we set
 \[\Omega_\rho(a) = \Omega \backslash \cup_{j=1}^n B_\rho(a_j).\]
By $\Omega^{n*}$, we will denote the set of $a\in \Omega^n$ such 
that $a_i\neq a_j$ for $i\neq j$.
 
Finally, we introduce a weak topology related to the length of a minimal connection, see \cite{BCL}, 
$$
\| f \|_{\dot{W}^{-1,1}(\Omega)} = \sup_{\substack{\LN \nabla \phi \RN_{L^\infty(\Omega)} \leq 1\\ \phi \in W_0^{1,\infty}(\Omega)}}
\LV \int_\Omega \phi f \RV.
$$
 This norm  provides a good scale-invariant measure of the distance of $J(u)$ and ${e_\e(u) \over \logep}$ to a sum of delta functions.   In particular if $|a_j - b_j| \leq \rho_\star$ for $j=1, \ldots, n$ then  $$\| \sum_j \delta_{a_j} - \sum_j \delta_{b_j} \|_{\dot{W}^{-1,1}(\Omega)} = \sum_j \LV a_j - b_j \RV.$$

We can now state our first theorem which supplies a long time existence result of the vortex motion law for asymptotically large numbers of vortices in the dilute regime.  
\begin{theorem} \label{thm:quantdynamics}
Suppose $u$ solves \eqref{tdgl} with either \eqref{DirichletBC} or \eqref{NeumannBC}.  Furthermore, let 
$n \leq {\logep^{1\over 200}}$ and $\rho_\star \geq \logep^{-{1\over 100}}$ and 
 {suppose $\| u_{\e_n}(0) \|_{L^\infty(\Omega)} \leq 1$,}
\begin{align} \label{wellpreparedness}
D(a(0)) & \lesssim \logep^{-{2\over 5}}, \\
\label{wellprep2}
\| J(u(0)) -  \pi \sum_{j=1}^n \delta_{a_j(0)} \|_{\dot{W}^{-1,1}(\Omega)} 
& \lesssim \logep^{-{1\over 3}} 
\end{align}
then for all $0 \leq t \leq \tau_{max}$ we have
\begin{align}
\| {e_\e(u)(t) \over \logep} - \sum_{j=1}^n \pi \delta_{a_j(t)} \|_{\dot{W}^{-1,1}(\Omega)} & \lesssim \logep^{-{1\over 4}}  ,\label{energydensityconvergencethm1} \\
\| J(u)(t) - \sum_{j=1}^n \pi \delta_{a_j(t)} \|_{\dot{W}^{-1,1}(\Omega)} & \lesssim \logep^{-{1\over 4}}, \label{Jacobianconvergencethm1} \\
\int_{\Omega_{\rho_\star}(a(t))} e_\e(|u(t)|) + {1\over 4} \LV { j(u(t)) \over |u(t)| } - j(u_\star(t)) \RV^2 & \lesssim \logep^{- {1\over 5}} ,\label{honeconvergence} \\
\| { j(u(t)) \over |u(t)| } - j(u_\star(t)) \|_{L^{4 \over 3}(\Omega)} & \lesssim \logep^{- {1\over 10}} \label{lpconvergence}
\end{align}
where $j(u_\star(t)) = j(u_\star(a(t)))$, 
\begin{equation*} \label{ehrenfesttime}
\tau_{max} = \min\LCB \tau_0, C \sqrt{|\log\logep| {\rho_\star^4 \over n^{3}}} \RCB,
\end{equation*}
$a_j(t)$ solve \eqref{ODE},
and {$C = C(\Omega)$}. 
\end{theorem}
Theorem~\ref{thm:quantdynamics} can be {extended to initial data having vortex degrees $d_j = \pm 1$ following the the approach in \cite{JSp2}.}  
We also note from Lemma 14 of \cite{JSp2} one can easily construct maps $u_\star^{\e}(x; a)$ that satisfy the well-preparedness assumptions
 \eqref{wellpreparedness}--\eqref{wellprep2}. 
Finally, in the case of a bounded number of vortices it is well known that the well-preparedness hypothesis is not very important, since 
one can show that data will become well-prepared almost 
instantaneously due to strong convergence estimates,
see \cite{LinHeat, JSparabolic, BOS, SerfatyCollision}, and we have no reason to
expect a different behavior here.

Given the result above, we can prove that the sequence of solutions converge in a prescribed sense
 to the expected hydrodynamic limit.   In this theorem we study only Dirichlet boundary conditions \eqref{DirichletBC} since we need to have the vortex motion law hold for times of order $O(n^{-1})$, and in the Neumann case \eqref{NeumannBC}  vortices will migrate  to the boundary too quickly. 

What type of equations do we expect the vortex density to satisfy in the limit? Following E's formal calculations \cite{ELiquid}
and adapting them to the bounded domain case with Dirichlet boundary conditions,
we rescale time $\overline{t} =  {n t}$ 
and consider the limiting  vortex density function $\omega = \lim_{n\to\infty} {1\over n}\sum_{j=1}^n \delta_{a_j(\overline{t})}$.
We obtain the system
\begin{equation} \label{MFE}
\begin{split}
\p_{\overline{t}} \omega + \operatorname{div} ( \omega v) & = 0 \\
v & = 4\pi\nabla \LC\Delta^{-1}_\mathcal{N} \RC \omega
\end{split} \end{equation}
where $\Delta^{-1}_\mathcal{N}: g \to w$ arises through the Poisson problem
\begin{equation} \label{NeumannFunctionIntro} \begin{split}
\Delta w & = g  \hbox{ in } \Omega \\
\p_\nu w & = \frac{1}{2\pi}\p_\tau \theta \hbox{ on } \p \Omega ,
\end{split} \end{equation}
and
 $\theta = \arg(x+iy)$.  Note consistency requires $\int_\Omega g = 1$ due to the Neumann boundary condition\footnote{
{
Our choice of boundary condition \eqref{DirichletBC} is not the most general possible and requires the domain be star-shaped and include the origin. Up to a correction by $\varphi_\star$ that is 
asymptotically small as $n$ becomes large, we have chosen 
$u\approx(\frac{z}{|z|})^n$, which makes sense because we assume $0\in\Omega$. Following \cite{SandierSoret}, it is possible to choose any degree-one map
$U_0:\partial\Omega\to S^1$ and to use the boundary condition $u=U_0^n$ on any simply-connected domain with $C^2$ boundary, instead of \eqref{DirichletBC}. While this does not substantially complicate the analysis,
we have chosen the simpler  \eqref{DirichletBC}, motivated by the case of $\Omega=B_1(0)$.}
}.   To motivate a notion of an interior weak solution  of \eqref{MFE} we follow Lin-Zhang \cite{LinZhang}.
If $\omega$ is a smooth solution to \eqref{MFE}, we multiply by $\chi \in C_0^\infty([0,T]\times \Omega)$ and integrate by parts.  Then, 
writing $t$ for $\overline{t}$ again,  
\[
- \int \int \p_t \chi \omega - \frac1{4\pi}\int \int \p_{x_k} \chi v_k \p_{x_j} v_j = 0,
\]
where we used $\omega = \frac1{4\pi}\operatorname{div} v$ in the interior of $\Omega$. 
Performing
integration by parts and using $\p_{x_1} v_2 -\p_{x_2} v_1 =0$,
we obtain
the identity
\begin{equation}\label{eq:ibpdiv}
\int \p_{x_k} \chi v_k \p_{x_j} v_j = 
- \int (\p_{x_1}^2-\p_{x_2}^2)\chi (v_1^2-v_2^2)
-4\int \p_{x_1 x_2} \chi v_1 v_2,
\end{equation}
and so we arrive at
 \eqref{generalizedweaksolution} below.  We note that this definition is  similar to the one introduced in \cite{Delort, Majda} for weak solutions to the 2D incompressible Euler equations except that the associated test functions are exchanged.  

\begin{definition} \label{generalizedweaksolutiondef}
We say $\omega$ is a \emph{generalized interior weak solution} to \eqref{MFE} if for all $\chi \in C_0^\infty([0,T]\times\Omega)$
\begin{equation}  \label{generalizedweaksolution}
- \int_0^t \int_\Omega \omega \p_t \chi + 
\frac{1}{4\pi} \int_0^t  
\int_\Omega \LC {\p^2_{x_1} \chi - \p^2_{x_2} \chi} \RC \LC v_1^2 - v_2^2 \RC 
+ \frac1\pi\int_0^t \int_\Omega  \p_{x_1} \p_{x_2} \chi  \ v_1 v_2  = 0
\end{equation}
where \[
v_j (x) = 4\pi\p_{x_j} \Delta^{-1}_{\mathcal{N}} \omega = 
2\int \p_{x_j} N(x,y) 
\omega(y) dy.\]
Here
 $N(x,y)$ is the Neumann function, which satisfies
 \begin{align*}
 \Delta N(\cdot,y) &= 2\pi \delta_y \quad\text{in $\Omega$}
 \\
 \p_\nu N(\cdot,y) &= \p_\tau \theta \quad\text{on $\p\Omega$}.
 \end{align*}

\end{definition}
We can now state our main result which shows that we can solve \eqref{MFE}-\eqref{NeumannFunctionIntro} with vortex sheet initial data via a subsequence of 
either solutions of \eqref{tdgl} or \eqref{ODE} with appropriate data.  
\begin{theorem} \label{thm:hydrodynamic}
Assume that $\omega_0 \in \mathcal{M} \cap \dot{H}^{-1}(\Omega)$ satisfies $\omega_0 \geq 0$, $\int_\Omega \omega_0 = 1$, and $\operatorname{supp} (\omega_0) \subset \{ \operatorname{dist}(x, \p \Omega) \geq C_0 > 0 \}$ for some constant $C_0$. Then there exists a sequence of initial data $u_{\e_n}(0)$ with $n = |\log |\log |\log \e_n| | |^{1\over4} $ number of vortices that satisfies the hypotheses for Theorem~\ref{thm:quantdynamics}
such that ${1\over n} {e_{\e_n}(u_{\e_n}(0)) \over \pi |\log \e_n| } \to \omega_0$ in $\mathcal{M}$ as $\e_n \to 0$.
Such initial data generates a sequence of solutions $u_{\e_n}(t)$ of \eqref{tdgl} with boundary condition \eqref{DirichletBC} for times up to $T =  |\log |\log \e_n| |^{1\over 7}$.

Setting $\ot = n t$ and letting $\omega_{\e_n} (\ot) = { 1 \over n} { e_{\e_n}(u_{\e_n}(\ot)) \over \pi |\log\e_n|}$ then 
for a subsequence
\begin{align*}
\omega_{\e_n}  \to \omega \hbox{ in } \mathcal{M}(\Omega
\times
[0,\infty))
\end{align*}
where $\omega$ is a \emph{generalized interior weak solution}, defined above, to
\begin{equation} \label{meanfieldequation}
\begin{split}
\p_{\ot} \omega + \operatorname{div} \LC v \omega \RC & = 0 \\
v & = 4\pi\nabla \LC \Delta^{-1}_\mathcal{N} \RC \omega
\end{split}
\end{equation}
 Finally, $v(\overline t) \in L^2_{loc}(\Omega)$.  Here $\Delta^{-1}_\mathcal{N} f = w $ if 
\begin{equation} \label{definitionNeumannLaplaceOp}
\begin{split}
\Delta w & = f \hbox{ in } \Omega \\
\p_\nu w & = \frac{1}{2\pi} \p_\tau \theta \hbox { on } \p \Omega
\end{split}
\end{equation}
where $\theta = \arg(x+iy)$ and $\int_\Omega f = 1$.  
\end{theorem}

The convergence to the hydrodynamic limit holds true for a more general
 class of initial data similar to those we construct in Theorem~\ref{thm:hydrodynamic}.
 This yields the following result on the  limit from the parabolic Ginzburg-Landau equation to the mean field equation:

\begin{theorem}
 \label{lowerboundrhostardirichletconverse}
Let $u_{\e_n}(0)$ be a sequence of initial data to the Ginzburg-Landau equation \eqref{tdgl} with Dirichlet boundary conditions \eqref{DirichletBC} with 
{$|| u_{\e_n}(0) ||_{L^\infty(\Omega)} \leq 1$ and} satisfies the following
hypotheses:
\begin{align}
\LN J(u_{\e_n}(0)) - \sum_{j=1}^n \pi \delta_{a_j(0)} \RN_{\dot{W}^{-1,1}(\Omega)} & \lesssim \LV \log \e_n \RV^{-{1\over 3}} \label{hyp1}  \\
D(a(0)) & \lesssim \LV \log \e_n \RV^{-{2\over 5}}  \label{hyp2} \\
n &  \leq \LV \log | \log \LV \log \e_n \RV |\RV^{1 \over 4}   \label{hyp3} 
\end{align}
with $a(0) \in \Omega^n$ satisfying the following:
\begin{align}
\rho_{a(0)}&  \geq \LV \log | \log \LV \log \e_n \RV | \RV^{-{ 1 \over 3}}    \label{hyp4}  \\ 
-{1\over n^2} \sum_{j\neq k} N_n(a_j(0),a_k(0)) & \lesssim 1,  \label{hyp5} 
\end{align}
where $N_n(\cdot,\cdot)$ is defined in Section~\ref{sec:dirichletrenen} and is closely related to $N(\cdot,\cdot)$.
Setting $\ot = n t$ and letting $\omega_{\e_n} (\ot) = { 1 \over n} { e_{\e_n}(u_{\e_n}(\ot)) \over \pi |\log\e_n|}$ then 
for a subsequence
\begin{align*}
\omega_{\e_n}  \to \omega \hbox{ in } \mathcal{M}(\Omega
\times
[0,\infty)).
\end{align*}
where
 $\omega(\overline{t})$ is a {generalized interior weak solution} to
\eqref{meanfieldequation} for $\overline{t} \in (0,\infty)$.  If  $v = 4\pi\nabla \LC \Delta^{-1}_\mathcal{N} \RC \omega$, defined by $\eqref{definitionNeumannLaplaceOp}$, then  $v(\overline t) \in L^2_{loc}(\Omega)$.  
\end{theorem}
The assumptions of Theorem~\ref{lowerboundrhostardirichletconverse} 
may look rather demanding; nevertheless, such data exist per the construction in the proof of Theorem~\ref{thm:hydrodynamic}.
Furthermore, we expect that fairly generic data corresponding to a collection of degree $1$ vortices will satisfy such assumptions after a short time since the parabolic Ginzburg-Landau equation quickly
dissipates not only the Ginzburg-Landau energy and the renormalized energy, but also the 
excess energy, compare \cite{BOS,SerfatyCollision}.

\subsection{Discussion}

The issue of whether the weak solution satisfies the correct boundary condition is a deep and difficult question.  Since vorticity can (and should) concentrate on the boundary, it is difficult to acquire the necessary regularity to ensure the boundary conditions are achieved in the classical weak sense.  Some recent progress has been made in \cite{LNX} by establishing \emph{boundary-coupled weak solutions}  of the two dimensional incompressible Euler equations in exterior domains.

To make a fully consistent limit it would be interesting to study the question of uniqueness of the limiting mean field equation \eqref{meanfieldequation}.  In \cite{LinZhang} the authors establish uniqueness for initial data in $L^\infty$ with compact support for the problem
in $\R^2$.  A similar study of regular solutions would be natural for  \eqref{meanfieldequation}-\eqref{definitionNeumannLaplaceOp} too.

From formal considerations of \eqref{meanfieldequation} the vortex density function satisfies $\p_t \omega + v\cdot \nabla \omega = - 4\pi\omega^2$,
 so along the trajectory of the induced velocity one sees that the density function should decay like $t^{-1}$.  For smooth initial data on $\R^2$ Lin-Zhang \cite{LinZhang} proved this fact, which implies that the vorticity spreads out quickly from a compact set.   This behavior implies that we expect most vortices to be  pushed out to the boundary in a similar fashion.  This conforms to the picture presented in Sandier-Soret \cite{SandierSoret} for global minimizers of the functional $E_\e(u)$ on bounded domains, constrained to the boundary condition of the type $u=e^{i n \theta}$ and $n \to \infty$.  Sandier-Soret show that vortices accumulate close to the boundary of the domain as $n$ grows asymptotically large.  
Taken together, we should view Theorem~\ref{thm:hydrodynamic} as a mean field description of the vortex density for  times in the  \emph{mesoscale} in the interior of the domain.

 The dilute density of the vortex liquid results from two issues.  The first is that we use energy comparison and a Gronwall inequality to pin the vortex positions to the ODE \eqref{ODE}.  This results in an upper bound 
 $\tau_{max} \lesssim \sqrt{{\rho_\star^4\over n^3} | \log \logep|} $ in Theorem~\ref{thm:quantdynamics}.  
 Integrating methods of \cite{SerfatyCollision} and/or \cite{BOS} should improve some of these bounds.
   The second issue arises from the poor bounds on the intervortex distance for the ODE \eqref{ODE}.  Better knowledge of how the ODE behaves should improve the vortex density allowed here.

%

Although \eqref{tdgl} provides a fertile ground to test the mathematics of the Gorkov-Eliashberg equations, the more physical problem entails looking at the hydrodynamic limit of \eqref{ge}.  
For the Gorkov-Eliashberg equations \eqref{ge}, corresponding
proofs of the vortex motion law are due to the second author  \cite{SpirnGE} for $O(1)$ fields and 
Sandier-Serfaty \cite{SandSerf} for larger fields, following the formal asymptotic work of \cite{PR}.   
Formally, it was shown by Chapman-Rubinstein-Schatzman \cite{CRS}  that the hydrodynamic limit of the associated ODE arising from the vortex motion law of
\eqref{ge} converges to a weak solution of 
\begin{equation} \label{completemeanfieldequations}
\begin{split}
\p_{{t}} \omega + \operatorname{div} ( \omega v) & = 0 \\
v & = \nabla  \LC \Delta- \mathbb{I} \RC^{-1}  \omega .
\end{split}
\end{equation}

There has been a lot of recent progress on the limiting
equations for the vortex densities \eqref{completemeanfieldequations}. Ambrosio-Serfaty \cite{AS} and
Ambrosio-Mainini-Serfaty \cite{AMS} study them as a metric gradient flow
in the space of measures, with the Wasserstein distance as the natural metric.
However, they do not obtain the convergence.
Even when it becomes possible 
to carry out the program outlined in the survey of Serfaty 
\cite{SerfGCsurv} and
to directly obtain the Wasserstein gradient flow studied in \cite{AS,AMS}
from the Gorkov-Eliashberg equation by a $\Gamma$-convergence
of gradient flows type result, we believe that our approach will still 
be useful. For one, it provides quantitative bounds that are useful 
in type II superconductors without going to the $\e\to0$ limit of ``extreme''
type II superconductivity.
More importantly, as our approach does not rely on the gradient flow structure,
 it can be adapted to yield results for more general situations,
such as the mixed flows studied in \cite{KMMS} and \cite{Miot}
for the ungauged problem and in \cite{KS_CPDE} and \cite{SerfatyTice}
for the gauged problem. Such motion laws have physical importance,
as they can be used to explain the sign-change in the Hall effect of
type II superconductors, see \cite{Dorsey}, \cite{Kopnin}. 
Similarly, we expect that our approach can be adapted also
to the Hamiltonian Ginzburg-Landau wave system, where
results for the PDE to ODE limit for finitely many vortices
have been found in \cite{Jwave} and \cite{LinWave},
and the ODE to mean field PDE limit has been
studied in \cite{LinZhang:Wave}.

\subsection{Method}
We finish the introduction with an outline of the arguments in the paper.  The general scheme of the paper is to  deduce the vortex motion law for the time dependent Ginzburg-Landau equations by 
carefully considering certain differential identities,
in particular the time evolution of the energy density.  

Our proof is based on the following differential identities,
which hold for smooth solutions of \eqref{tdgl}:
\begin{align}  \label{conservationmass}
\hbox{ mass identity } \quad & \LB {1\over \logep} \p_t - \Delta - {2 \over \e^2} |u|^2 \RB \LC |u|^2 - 1\RC =  2 \LV \nabla u \RV^2  \\
\label{eqn:supercurrent}
\hbox{ supercurrent identity } \quad & {1\over \logep} \LC i u , \p_t u \RC = \text{div}  j(u) \\
 \label{eqn:conservationenergy}
\hbox{ energy identity } \quad & \p_t e_\e(u) = \text{div}  \LC \p_t u , \nabla u \RC - { \LV \p_t u \RV^2 \over \logep}.
\end{align}
For fixed $\e$ regularity follows from standard parabolic theory.  We remark that \eqref{conservationmass} can be used to show that $0 \leq |u| \leq 1$; \eqref{eqn:supercurrent} will be used to show that $j(u)$ is nearly divergence-free in a time-averaged sense.

The identity \eqref{eqn:conservationenergy} is crucial in 
obtaining a lower bound for the kinetic energy.
Using \eqref{tdgl} once more,
we can also deduce from \eqref{eqn:conservationenergy}
\begin{equation}\label{eq:consenwithdivT}
\frac{1}{\logep} \p_t e_\e(u) 
 = \mathrm{div} \ \mathrm{div}\left(\nabla u \otimes \nabla u - \frac12 e_\e(u) \mathrm{id} 
 \right)-\frac{|\p_t u|^2}{\logep^2},
\end{equation}
which
 is the primary tool to establish the vortex motion law. 
 
Passing to the limit $\e\to 0$ in \eqref{eq:consenwithdivT}
and controlling the growth of the energy excess
would yield a proof of the motion law for bounded numbers of 
vortices if the initial energy excess is $o(1)$ as $\e\to 0$.
This method is not as powerful as the elliptic PDE approach of Serfaty 
\cite{SerfatyCollision0, SerfatyCollision} or the parabolic PDE approach of 
Bethuel-Orlandi-Smets \cite{BOS}, but it
provides a way to avoid using convergence properties in the proof,
and we can use quantitative estimates in every step. Passing
to the limit $\e\to 0$ for bounded $n$, our results are weaker than those
in the literature, but our explicit bounds provide rates of 
convergence. 

Our approach of using differential identities and 
explicit estimates
follows the program of the second author and R. Jerrard \cite{JSp2} for the Gross-Pitaevsky equation $i \p_t u = \Delta u + {1\over \e^2} u \LC 1 - |u|^2 \RC$.  Surprisingly, implementing this approach for \eqref{tdgl} is more challenging and requires several new estimates.

One such additional difficulty is that 
 the arguments of all previous vortex motion law 
proofs for \eqref{tdgl} use a limiting kinetic energy lower bound,
which has so far only been available for a bounded number  of vortices.
In Theorem~\ref{thm:kineticenergybound},
 one of our central results,
 we provide such a bound for a large number of vortices.
This type of estimate is not needed for the 
Gross-Pitaevsky equation since one has conservation of energy
for both the PDE and ODE in that case.

We give an overview of the contents of the rest of the paper.

In Section~\ref{sec:renen}, we recall some known results
on the renormalized energy. Lemma~\ref{L.gradW} connects
the gradient of the renormalized energy to the
canonical harmonic map $u_\star$, and Proposition~\ref{gammastabilityprop}
quantifies how close $u$ and $u_\star$ are based on the excess energy.

In Section~\ref{sec:dirichletrenen}, we give some detailed
results for the renormalized energy in the Dirichlet case, following
Sandier-Soret
\cite{SandierSoret}. These estimates are used to show that vortices
stay away both from each other and the boundary for sufficiently 
long times to pass to the hydrodynamic limit under certain conditions.  

In Section~\ref{sec:local}, we discuss localization 
estimates for the Jacobian and energy density. For the
Jacobian, results of \cite{JSp2} yield points $\xi_j$ such that
$\|J(u)-\pi\sum \delta_{\xi_j}\|_{\Wdot}$ is small in a 
precisely quantified way.
We provide a
 new estimate on the localization of the Ginzburg-Landau energy density 
 to the same set of delta functions of the type
 $$ \LN {e_\e(u) \over \logep} - \sum_{j=1}^n \pi \delta_{\xi_j} \RN_{\dot{W}^{-1,1}(\Omega)} \lesssim {1\over \logep} \LB W(a)+  n D(a) \RB .$$  The estimate presented here is a refined (i.e. $\e$-rate dependent) version 
of an estimate found in Colliander-Jerrard \cite{CJ}.  Therefore, in order to localize the vortices to a high resolution, we need good estimates on the excess energy, $D(a(t))$.  

Since the localization and gamma stability error estimates depend explicitly on the excess energy, it is necessary to understand how the excess energy evolves in time.  By the energy dissipation identities  \eqref{PDEenergyidentity} and \eqref{ODEenergyidentity} we see that 
\begin{equation}  \label{excessenergyidentity}
D(a(t)) = D(a(0))   + \pi\int_0^t { \LV \dot{a} \RV^2 }  -   \int_0^t \int_\Omega {\LV \p_t u \RV^2 \over \logep}  ;
\end{equation}
consequently, $D(a(t))$ can be controlled by well-preparedness of the initial data and a lower bound on the kinetic energy.  
This lower bound is presented in Section~\ref{sec:quantkinetic}
as Theorem~\ref{thm:kineticenergybound}: 
\[
\int_0^t \int_\Omega {\LV \p_t u \RV^2 \over  \logep}  -\pi \int_0^t { \LV \dot{a} \RV^2 }  \geq
- \mathbf{error} , 
\]
where $\mathbf{error}$ 
depends explicitly on $\e$, the number of vortices, 
the minimal vortex distance, the time scale 
and the localization error.
 This result provides 
a purely quantitative approach to the kinetic energy lower bounds that are found in \cite{LinWave, Jwave, SandSerf}, each of which rely on compactness properties to get a lower bound.  To establish this result we make quantitative the kinetic energy estimate of \cite{LinWave}, who used the differential identity for the energy density, along with a limiting result on the equipartitioning of potential energy.  Here we make use of an optimal quantitative equipartitioning result in \cite{KurzkeSpirnMRI} that identifies how close in $L^1$ the tensor $\nabla u \otimes \nabla u$ is to the diagonal matrix $\frac12|\nabla u|^2 \mathrm{id}$.  Placing 
this equipartitioning result into the differential identity for the Ginzburg-Landau energy $e_\e(u)$, applying a test function 
$$
\sum \chi ( x - a_j(t)) \dot{a}_j(t) \cdot (x - a_j(t)), \ \hbox{  $\chi$ a smooth cutoff function},
$$
and integrating over $\Omega$ yields the lower bound.  

After these preparations, we prove Theorem~\ref{thm:quantdynamics}
in Section~\ref{sec:proofmainth}. The main task is to 
understand how close the points $\xi_j(t)$, found by the localization estimates, 
are to
the points $a_j(t)$ given by the ODE. To this end, we introduce
a quantity $\eta(t)$ which serves as a differentiable
 replacement 
for $\pi\sum_j|\xi_j(t)-a_j(t)|$.

In Subsection~\ref{subsec:assu}, we define various small quantities
that serve as error bounds in our estimates, and several
time intervals on which good estimates hold; 
in particular, we show that $\eta$ really controls 
everything we need.

It therefore suffices to control the growth of $\eta$ via a Gronwall argument.
We  estimate $\dot\eta$ in Subsection~\ref{subsec:doteta},
relying on the energy evolution \eqref{eq:consenwithdivT}.
The resulting simple bound of the type
$
|\dot \eta| \lesssim A_\e \sqrt{\eta} + o(1)_\e$
 is not sufficient to apply the Gronwall
inequality globally, but yields a reasonable short time result.
The culprit for the $\sqrt{\eta}$ is a certain supercurrent
estimate 
that is difficult to improve at a fixed time.

Subsection~\ref{subsec:timeaverage} provides the necessary improvements
by averaging over a short timescale, $\dot {\LA \eta \RA} \lesssim \widetilde A_\e \LA \eta \RA + o(1)_\e$. This technique, taken
from \cite{JSp2}, makes use of \eqref{eqn:supercurrent}
to obtain a
 quantitative bound 
on how far $j(u)$ is from being divergence free.
Using a Hodge decomposition of $\frac{j(u)}{|u|}-j(u_\star)$ 
and the fact that $j(u_\star)$ is divergence free 
while $\curl (j(u)-j(u_\star)$ is controlled, 
we can bound time averages of terms of the type
$\int_{\Omega_\rho} \zeta \LC {j(u) \over |u|} - j(u_\star)\RC$ 
for some prescribed function $\zeta$. As in 
\cite{JSp2}, this part is fairly technical, 
but the differences from the Gross-Pitaevsky case are
significant enough that we feel it is
necessary to include these  details.

The proof of Theorem~\ref{thm:quantdynamics} is finished in
Subsection~\ref{subsec:continuity}, where we show via 
a continuity argument that  
$J(u)$ is localized near the $a(t)$ for long times. In particular,
we obtain the vortex motion law.

In the final Section~\ref{sec:hydrodynamic}, we consider
the hydrodynamic limit and prove Theorems~\ref{thm:hydrodynamic} and~\ref{lowerboundrhostardirichletconverse}.
 In the first part of the section, we prove a hydrodynamic limit of the vortex ODE's for bounded domains
 which is analogous to 
 the results of \cite{Schochet, LiuXin} for the Euler point vortex method on $\R^2$ and the gradient flow version of \cite{LinZhang} on $\R^2$ for bounded domains.  
The proof
 requires a careful expansion of the time dependent behavior of ${1\over n} \sum_{j=1}^n \delta_{a_j(t)}$ integrated against a test function with compact support.  Implementing the strategy of \cite{LiuXin, LinZhang}
 and using estimates on the Neumann function,
  we prove the convergence and the local velocity bound.  
 
 To complete the proof of Theorem~\ref{thm:hydrodynamic}, we show 
 that nonnegative vortex sheet initial data
  with compact support can be 
  approximated by a sequence of a 
  sum of degree-one vortices that satisfy the conditions of our 
  class of initial data. 
  This improves on the construction in \cite{LinZhang},
  which uses vortex blobs with arbitrary vorticities.
    Finally, due to Theorem~\ref{thm:quantdynamics},
 the quantity
  $\omega_n(t) = {1\over n }{e_\e(u(t) ) \over \pi \logep}$
converges to the same limit
 as the vortex density function ${1\over n} \sum_{j=1}^n
 \delta_{a_j(t)}$.
 
\section{The renormalized energy}
\label{sec:renen}
In this section we recall some results on the renormalized
energy and the canonical harmonic map. 

Recall from \cite{BBH} the canonical harmonic map $u_\star = {u_\star(\alpha)}$ which satisfies the following Hodge system
\[
{\operatorname{div} j(u_\star)=0 \qquad \curl j(u_\star) = 2 \pi \sum_{j=1}^n \delta_{\alpha_j} }
\]
with either
\[
j(u_\star) \cdot \tau = n \p_\tau \theta + {\p_\tau \varphi_\star }
\]
on $\p \Omega$  or 
\[
j(u_\star) \cdot \nu = 0 
\]
on $\p \Omega$.  
There exists a $G$ with $\nabla^\perp G = j(u_\star)$, where
$G(x) = G(x; \alpha)$ is defined by the following Poisson equation
\begin{equation} \label{G.def}
\Delta G = 2 \pi  \sum_{j=1}^n \delta_{\alpha_j}\hbox{ in } \Omega 
\end{equation}
with either
\[
\p_\nu G = n \p_\tau \theta + {\p_\tau \varphi_\star } \hbox{ on }\p \Omega
\]
or 
\[
G = 0  \hbox{ on }\p \Omega.
\]

The \emph{renormalized energy}  is then  defined, {recalling $\Omega_\rho(\alpha)=\Omega\setminus\cup_{j=1}^n B_\rho(\alpha_j)$,} as
\[
{W(\alpha)} = \lim_{\rho \to 0} \left(\int_{\Omega_\rho(\alpha)} {1\over 2} \LV \nabla u_\star \RV^2  - \pi n \log {1\over \rho}\right).
\]
We also define the \emph{approximate energy} as
\[
{W_\e(\alpha) = W(\alpha)} + n \LC \pi \logep + \gamma \RC
\]
where 
$$\gamma=\lim_{\e\to 0}\left(\inf_{u\in H^1(B_1;\C),u(z)=z\text{ on $\p B_1$}}
E_\e(u)-\pi\logep
\right)
$$
is the constant from \cite[Lemma IX.1]{BBH}.  
Finally, the \emph{excess energy} is defined as 
\begin{equation} \label{definitionexcessenergy}
{ D(u;\alpha)= E_\e(u) - W_\e(\alpha).}
\end{equation}
{We usually simplify this to $D(\alpha)$} when the 
context is unambiguous.

 We also define the excess energy in the ball $B_\sigma(\alpha_j)$ to be
\begin{align*}
D_{B_\sigma(\alpha_j)} = \int_{B_\sigma(\alpha_j)} e_\e(u) - \LC \pi \log {\sigma \over \e} + \gamma \RC.
\end{align*}

We will use the following
 characterization of the gradient of the renormalized energy. 

\begin{lemma}  [Lin \cite{LinHeat}, Jerrard-Soner \cite{JSparabolic},
Jerrard \cite{Jwave}]
Let $\alpha \in \Omega^{n*}$  then the canonical harmonic
map $u_\star = u_\star(\cdot; \alpha)$ and the renormalized energy $W(\alpha)$ satisfy
\begin{equation*}  
\int \p_{x_k x_m} \zeta \LB  \LC j(u_\star)\RC_m \LC j(u_\star) \RC_k 
- {1\over2} \delta_{k m } \LV j(u_\star) \RV^2 \RB  
= -  \sum_{j=1}^n  \partial_k \zeta (\alpha_j)  \LC  \nabla_{\alpha_j}  W (\alpha) \RC_k 
\end{equation*}
where $\zeta \in C^2(\Omega)$ and $\nabla^2 \zeta$ has support in an annular  neighborhood of the $\alpha_j$'s. 
\label{L.gradW}\end{lemma}

Next we list estimates on the canonical harmonic map and renormalized energy and their derivatives.

\begin{lemma}
There exists a constant $C$ depending only on $\Omega$ such that
for every bounded, open  $\Omega\subset \R^2$,
$\alpha \in \Omega^{n*}$, the renormalized energy $W(\alpha)$,
canonical harmonic map $u_\star(\cdot,;\alpha)$
and its potential  $G(\cdot; \alpha)$ as defined in \eqref{G.def}
satisfy
\begin{equation}
\| j(u_\star) \|_{L^\infty(\Omega_r(\alpha))}\ = \ 
\|\nabla G \|_{L^\infty(\Omega_r(\alpha))} \ \le \ 
\frac {2n} r
\label{justar.Linf}
\end{equation}
for all $r \le \rho_\alpha$, and
\begin{equation}
|\nabla_i W(\alpha)|\le \frac {Cn}{\rho_\alpha},\quad\quad
|\nabla_i\nabla_j W(\alpha)|\le \frac {Cn}{\rho_\alpha^2}
\label{Wderivatives}\end{equation} 
for every $i,j\in \{1,\ldots, n\}$.

We also have the upper bound
\begin{equation}
W(\alpha) \le C( n^3 + \frac {n^2}{\rho^2_\alpha}). 
\label{W.scale}\end{equation}


Finally, let $\alpha = ( \alpha_1 ,\ldots, \alpha_n)$ and $\alpha' = (\alpha_1', \ldots, \alpha_n')$ with
$\alpha, \alpha' \in \Omega^{n*}$.  Let 
$\Omega_{r}(\alpha,\alpha') = \Omega \backslash \LC \cup_{j=1}^n B_r (\alpha_j) \cup  B_{r}(\alpha_j') \RC$,
then 
\begin{equation}  \label{chmdifferencebound1}
\LN j(u_\star)(\alpha) - j(u_\star)(\alpha') \RN_{L^\infty (\Omega_r(\alpha,\alpha'))} 
\leq {1 \over r^2} \sum_{j=1}^n \LV \alpha_j - \alpha_j' \RV
\end{equation} 
for all $r \leq \min\{\rho_{\alpha} , \rho_{\alpha'} \}$, and additionally, for $1< p<2$,
\begin{equation}  \label{chmdifferencebound2}
\| j(u_\star(\alpha)) - j(u_\star(\alpha'))\|_{L^p}
\le
( \pi \sum |\alpha_i - \alpha_i'|)^{\frac 2p-1}(2n \pi)^{2-\frac 2p}.
\end{equation}
\end{lemma}  
\begin{proof} 
The Neumann boundary condition results are proved in {Lemma~10, Lemma~11, and Lemma~13} of \cite{JSp2}; we note a typo in the statement of estimate \eqref{W.scale} in \cite{JSp2}.
  Corresponding results for the Dirichlet boundary condition can be established by using similar arguments.
 Further estimates in the Dirichlet case are discussed in Section~\ref{sec:dirichletrenen}.
\end{proof}

Finally, we will need the following quantitative 
coercivity or $\Gamma$-stability result for the renormalized energy:
\begin{proposition}[Jerrard-Spirn, {Theorem 2 of }\cite{JSp2}] \label{gammastabilityprop}
Let $\Omega$ be a bounded, open simply connected subset of $\mathbb{R}^2$ with $C^1$ boundary. Then there exist constants $C$, $K_\star$ depending only on $\Omega$  such that for any $u \in H^1(\Omega; \mathbb{C})$, if there exist $n \geq 0$, finite, with
$\alpha = (\alpha_1, \ldots, \alpha_n) \in \Omega^{n*}$ such that 
\begin{equation*}
\left\| J(u) - \sum_{j=1}^n \pi  \delta_{\alpha_j} \right\|_{\dot{W}^{-1,1}}
\leq s_\e \qquad \hbox{ for some } s_\e \in [ \e \sqrt{\log(\rho_\alpha/\e)}, {\rho_\alpha/ K_\star} ],
\end{equation*}
and if $4 s_\e \leq \sigma^* =\sqrt{ {\rho_\alpha \over n^3} \LC s_\e + \e  E_\e(u) \RC } \leq {\rho_\alpha \over K_\star}$ then
\begin{equation}   \label{liminfenergyexcess}
\int_{\Omega_{\sigma^*}(\alpha) } e_\e(|u|) + {1\over 4} \left| {j(u) \over |u|} - j(u_\star(\alpha)) \right|^2 
\leq D(\alpha) + C  \sqrt{ {n^5 \over \rho_\alpha} \LC {s_\e} + {\e} E_\e(u) \RC }.
\end{equation}
 Finally,
\begin{equation*}
\left\| j(u) - j(u_\star(\alpha)) \right\|_{L^{4\over3}} \leq C \sqrt{D(\alpha)} + {\mathbf {error}}
\end{equation*}
and 
\begin{align*}
\mathbf{error} & \leq C \e^{1\over2} E_\e(u)^{3\over4} \\
& \quad + C n \LC s_\e + \e E_\e(u) \RC^{1\over4} \LB \LC { n \over \rho_\alpha} \RC^{1\over4} 
+ \rho_\alpha^{1\over4} \LC 1  +\sqrt{E_\e(u) \over n^3} \RC \right] .
\end{align*}
\end{proposition}

\section{Estimates for the Dirichlet case}
\label{sec:dirichletrenen}
In this section we provide estimates on the Neumann functions that comprise the renormalized energy in the Dirichlet case.  These estimates will be used both to generate long-lived solutions of \eqref{tdgl} with asymptotically many vortices and to provide kernel estimates for the hydrodynamic limit theorem.  

We follow the approach of Sandier-Soret \cite{SandierSoret} to define the renormalized energy in terms of Neumann functions.  In particular 
let $N_n(x,y)$ denote the Neumann function which satisfies the following equation
\begin{equation*} \label{neumannfunctionwithn}
\begin{split}
\Delta N_n(\cdot,y) & = 2\pi\delta_{y} \hbox{ in } \Omega \\
\p_\nu N_n(\cdot, y) & = \p_\tau \theta + {1\over n} \p_\tau \varphi_\star \hbox{ on } \p \Omega\\
\int_{\p\Omega} N_n(\cdot,y) (\p_\tau \theta + {1\over n} \p_\tau \varphi_\star)  &= 0
\end{split}
\end{equation*}
and the limiting Neumann function $N(x,y) = N_\infty(x,y)$ which satisfies the following equation
\begin{equation*} \label{neumannfunction}
\begin{split}
\Delta N(\cdot,y) & = 2\pi\delta_{y} \hbox{ in } \Omega \\
\p_\nu N(\cdot, y) & = \p_\tau \theta  \hbox{ on } \p \Omega\\
\int_{\p\Omega} N(\cdot,y) \p_\tau \theta &=0.
\end{split}
\end{equation*}

We also define $H_n(x, y) = N_n(x, y) - \log | x - y |$ and $H(x, y) = N(x, y) - \log | x - y |$ to be the harmonic pieces of the Neumann functions $N_n(x,y)$ and $N(x,y)$, respectively.  
Then 
\begin{equation}  \label{alternativedefrenormalizedenergy}
W(a_1, \ldots, a_n) = -\pi  \sum_{j\neq k} N_n(a_j, a_k)  - \pi \sum_{j=1}^n H_n(a_j,a_j),
\end{equation}
see \cite{BBH,SandierSoret}.

We state the following useful set of estimates:
\begin{lemma}[Sandier-Soret \cite{SandierSoret}]  \label{SandierSoretEstimates}
The Neumann function $N_n(x,y)$ satisfies for $1\le n\le \infty$
the estimates
\begin{enumerate}
\item $N_n(x,y) = N_n(y,x)$
\item $N_n(x,y) = \log |x - y| + H_n(x,y)$ where $H_n(x,y)$ is continuous on $\Omega \times \overline{\Omega} \cup \overline{\Omega} \times \Omega$.  
\item $N_n(x,y) = 2 \log |x - y | + \widetilde{H}_n(x,y)$ where $\widetilde{H}_n$ is continuous on $\p \Omega \times \overline{\Omega} \cup \overline{\Omega}  \times \p \Omega$.
\end{enumerate}
\end{lemma}

In the proof of Lemma~\ref{SandierSoretEstimates} the authors generate $H_n(x,y)$ in steps.  When $\Omega = B_1$ and $\p_\nu N(\cdot,y) = 1$ then $H(x,y) = \widehat{H}(x,y)$, where  $\widehat{H}$ can be explicitly
computed:
\begin{equation} \label{formTrivialNeumannFunction}
\widehat{H} (x,y) = \log | 1 - x \overline{y} |.
\end{equation}
For nontrivial $\p_\nu N_n(\cdot,y) = f _n = \p_\tau \theta + {1\over n} \p_\tau \varphi_\star$ one finds $\widehat{H}_{n}$ satisfies $\widehat{H}_{f_n}(x,y)
= \widehat{H} (x,y) + P(x) + Q(y)$ where $P(x)$ and $Q(y)$ are harmonic in $B_1$ and bounded and continuous up to the boundary.   
Finally, for simply-connected domain $\Omega$ let $w(z)$ denote the conformal mapping of $\Omega$ into $B_1$.  Then one finds
\begin{equation}  \label{formNeumannFunction}
H_n(x,y) = \widehat{H}_{\widehat{f_n}} (w(x), w(y))
\end{equation} 
where $\widehat{f_n}(x)$ is defined as $\widehat{f_n}(w(z)) = f_n(z)/|w'(z)|$.  We note again that in  our case $\widehat{f_n}(x) = \p_\tau \theta + {1\over n} \p_\tau \varphi_\star$.

%

Using \eqref{alternativedefrenormalizedenergy} and Lemma~\ref{SandierSoretEstimates} we prove the following lemma that provides a lower bound on the intervortex distance as a function of the renormalized energy.

\begin{lemma} \label{lowerboundrhofromW} Let $W(a)$ be the renormalized energy for the Dirichlet case, then 
\begin{align*}
\rho_{a} & \geq e^{- {\frac1\pi W(a)} - C n^2}
\end{align*}
for some constant $C$ that depends only on $\Omega$ and $\varphi_\star$.  
\end{lemma}
\begin{proof}
Since the domain is bounded we have from (2) and (3) of the Sandier-Soret lemma
\begin{align*}
{\frac1\pi W(a)} & = - \sum_{j\neq k} N_n(a_j,a_k) - \sum_{j} H_n(a_j,a_j) \\
& \geq \log [ \min\{ \min_{j\neq k}\{ |a_j - a_k| \}, \min_j\{ \operatorname{dist}(a_j, \p \Omega) \} \}]^{-1} - C n^2\\
& \geq  \log \rho^{-1}_a - C n^2,
\end{align*}
where $C$ depends on $\varphi_\star$ and $\Omega$.
Therefore,
$$
\rho_a = e^{\log \rho_a} \geq  e ^{-{\frac1\pi W(a)} - C n^2} .
$$


\end{proof}

The following proposition gives a class of data 
where a good bound on the minimal intervortex distance holds for all time. 

\begin{proposition} \label{lowerboundrhostardirichlet}
Assume that $a_j(t)$ are solutions to 
\[
\dot{a}_j = -{1\over \pi} \nabla_{a_j} W(a) 
\]
with the renormalized energy 
$W$ arising from the Dirichlet boundary condition
 \eqref{DirichletBC}.
If 
$\rho_{a(0)} \geq \LV \log | \log \logep | \RV^{-{ 1 \over 3}}$ and $ n \leq \LV \log | \log \logep |\RV^{1 \over 4} $
then the $a_j(t)$ satisfy 
\[
{\rho_{a(t)} \geq \frac1C | \log \logep|^{-{1\over 10}} }
\]
for all $t$. 
\end{proposition}

%
%
%
\begin{proof}
From our assumptions on $\rho_{a(0)}$ and $n$ we have $$W(a(0)) \leq C   {n^2 \over \rho_{a(0)}} \leq C |\log | \log \logep | |^{5 \over 6}.$$  
 Since $W(a(t)) \leq  W(a(0))$ then 
from Lemma~\ref{lowerboundrhofromW} we see that
\begin{align*}
\rho^{-1}_{a(t)} & \leq C e^{ {\frac1\pi}  W(a(t))+ C n^2} \leq C e^{ {\frac1\pi} W(a(0))+Cn^2}\\
& \leq  C e^{C |\log | \log\logep| |^{{5\over 6}}}\leq  C e^{ {1\over 10} |\log | \log\logep| |} 
\end{align*}
for all time. 
\end{proof}
We note that for this class of initial data {and setting $\rho_\star =  | \log | \logep | |^{-{1\over 6}}$
 so $\rho_{a(t)}\ge \rho_\star$,} 
the quantity $T=C\sqrt{ | \log \logep | { \rho_\star^4 \over n^3}}$ that appears 
in the assumptions of Theorem~\ref{thm:quantdynamics}
satisfies $T\gtrsim \frac1n$, since
$$ C\sqrt{ | \log \logep | { \rho_\star^4 \over n^3}} 
\geq C {| \log \logep |^{ 1 \over6} \over |\log | \log  \logep | |^{3 \over 4} } \geq  | \log \logep|^{1 \over 7}\gtrsim \frac1n,$$
so the time rescaling in Theorem~\ref{thm:hydrodynamic} makes sense.


\section{Localization results}
\label{sec:local}
In this section we 
discuss quantitative estimates
that show how well the fundamental quantities $J(u)$ and
$\frac{e_\e(u)}{\logep}$ are approximated by 
sums of point masses. For $J(u)$, these results
were shown in  \cite{JSp,JSp2}; for the energy, they are new.

\begin{proposition}[Jerrard-Spirn, {Theorem 3 of} \cite{JSp2}] \label{gammalocalizationprop}
Let  $\Omega$ be a
bounded, open, simply connected subset of $\R^2$ with $C^1$ boundary. Then there
exist constants $C>0$ and $C_\star\ge 2$, with $C_\star = \max\{K_\star, {1\over4} \operatorname{diam}(\Omega)\}$ and $K_\star$ is the constant from Proposition~\ref{gammastabilityprop},
such that the following property holds.

For any $u\in H^1(\Omega; \mathbb{C})$, if there exist $n\ge 0$,
 $\alpha = \LC \alpha_1, \ldots, \alpha_n \RC \in \Omega^{n*}$ 
such that 
\begin{equation*}  
\LN J(u) - \sum_{j=1}^n \pi  \delta_{\alpha_j} \RN_{\dot{W}^{-1,1}} 
\le \frac{\rho_\alpha}{8 C_\star n^5},
\label{p2.h1}\end{equation*}
and if in addition $E_\e(u) \ge 1$ and
\begin{equation}
\left[ \frac{n^5}{\rho_\alpha}
E_\e(u) + \frac {n^{10}}{\rho_\alpha^2}\sqrt{E_\e(u)}
\right] \le \frac 1 \e,
\label{Ebds}\end{equation}
then there exist $(\xi_1,\ldots,\xi_n)\in \Omega^{n*}$ such that
$|\xi_i - \alpha_i|\le  \frac{\rho_\alpha}{2 C_\star n^4} $ for all $i$, and 
\begin{align*}
&\left\| J(u) -  \sum _{j=1}^n \pi  \delta_{\xi_j} \right\|_{\dot W^{-1,1} }  
\\
&\hspace{5em}\le \
C\ \e\left[n (C + D(\alpha))^2
e^{\frac 1\pi D(\alpha) } + 
(C + D(\alpha)) \frac{n^5}{\rho_\alpha}+ {E_\e(u)}
\right].\nonumber
\end{align*} 
\label{CHMapprox2}\end{proposition}

We now state a result that clarifies the convergence of ${e_\e(u) \over \logep}$ to a set of delta functions.

\begin{theorem}\label{thm:energyisdelta}
Let $u$ satisfy the same hypotheses as Proposition~\ref{CHMapprox2}, then 
the $\{ \xi_j\}_{j=1}^n$ found in Proposition~\ref{CHMapprox2} satisfy 
\begin{align}
&\left\| {e_\e (u)\over \logep} - \pi \sum _{i=1}^n \delta_{\xi_i} \right\|_{\dot W^{-1,1} }  
\label{approxenergymeasuredelta}
\\
&\hspace{5em}\le \
{C \over \logep} \LB n (D(\alpha)  + \log{ n^4 \over \rho_\alpha }+ C ) + {W(\alpha)}  \RB  .  \nonumber
\end{align} 
\label{CHMapprox3}\end{theorem}

To prove Theorem~\ref{thm:energyisdelta}, we
 make precise (and quantitative) an argument found in \cite{CJ}.    The first step is a moment estimate on the Ginzburg-Landau energy about the vortex core.  

\begin{lemma}
If $\| J(u) - \pi \delta_0 \|_{\dot{W}^{-1,1}(B_r)} \leq {r\over 4}$ and $\LV \int_{B_r} e_\e(u) - \pi \ln { r \over \e} \RV  = K_0$ then
there exists $\xi \in B_{r/2}(0)$ and a constant $C$, independent of $K_0$ and $u$, such that 
\begin{equation} \label{momentCalcBall}
\int_{B_r} |x - \xi| {e_\e(u)}   \leq {r C ( K_0 + 1 )} .
\end{equation}
\label{singlevortexfirstmomentlemma}
\end{lemma}

\begin{proof}

By Theorem 1.2' of \cite{JSp} then there exists $\xi \in B_{r/2} (0)$ such that for any $\tau <  r - |\xi|$ and $\e \leq \sigma < \tau$,
\begin{equation}  \label{vortexenergyJSP}
\int_{B_{\tau}(\xi) \backslash B_\sigma(\xi)} e_\e(u)  \geq \pi \ln { \tau \over \sigma} - K_0 - C.
\end{equation}

From \eqref{vortexenergyJSP} we see that 
\begin{equation} \label{vortexenergylowerboundannulus}
\int_{B_\tau(\xi) \backslash B_\e(\xi) } e_\e(u)  \geq \pi \ln { \tau \over \e} - K_0  - C,
\end{equation}
so from \eqref{vortexenergylowerboundannulus} and the assumption on the energy
\begin{equation*} \label{vortexenergyupperboundannulus}
\begin{split}
\int_{B_\e(\xi)} e_\e(u)  + \int_{B_r (0) \backslash B_\tau(\xi)} e_\e(u) 
&  \leq \pi \ln {r \over \e} + K_0 - \pi \ln {\tau \over \e} + K_0 + C  \\
& \leq \pi \ln { r \over \tau} + C K_0 + C.
\end{split}
\end{equation*}

Now we look at the energy in the annular set $B_{2^{-j}} \backslash B_{2^{-(j+1)}}(\xi) \subset B_r \backslash B_\e(\xi)$.  In particular
\begin{align*}
\int_{B_{2^{-j}} \backslash B_{2^{-(j+1)}}(\xi)} e_\e(u)  
& = \int_{B_r} e_\e(u)  - \int_{B_r \backslash B_{2^{-j}}(\xi)} e_\e(u)  - \int_{B_{2^{-(j+1)}}(\xi)} e_\e(u)  \\
& \leq \pi \ln {r \over \e} + K_0 - \pi \ln {r \over 2^{-j}} + C(K_0 + 1) - \pi \ln {2^{-(j+1)}\over \e} + C(K_0 + 1) \\
& = \pi \ln {2 } + 2 C (K_0 + 1)  = C(K_0 +1 ).
\end{align*}

Next we prove the claim.  If we let $2^{-M_\e} = {r \over 2}$ and $2^{- N_\e} = \e$ then
\begin{align*}
\int_{B_r} |x - \xi| e_\e(u)  
& = \int_{B_\e} |x - \xi| e_\e(u)  + \int_{B_r \backslash B_{r/2}(\xi)} |x - \xi| e_\e(u)   \\
& \quad + \sum_{j=M_\e}^{N_\e} \int_{B_{2^{-j}}\backslash B_{2^{-(j+1)}}(\xi)} |x - \xi| e_\e(u)   \\
& \leq r \pi \ln {r \over r/2} C (K_0 + 1) 
+ \sum_{j=M_\e}^{N_\e} 2^{-j} C (K_0)  \\
& \leq r C ( K_0  + 1)
\end{align*}
since $\sum_{j = M_\e}^{N_\e} 2^{-j} \leq r$.

\end{proof}


In order to establish the proof of the theorem, we use the following 
local
energy lower bound:

\begin{lemma}[Jerrard-Spirn, {Theorem 1.3 of} \cite{JSp}]  \label{JSpSingleVortexLower}
There exists an absolute constant $C > 0$ such that if $u \in H^1(B_\sigma)$ satisfies 
\[
\LN J(u) - \pi\delta_0 \RN_{\dot{W}^{-1,1}(B_\sigma)} \leq {\sigma\over 4}
\]
then
\begin{equation*}
0 \leq D_{B_\sigma} + C {\e \over \sigma} \sqrt{ \log{ \sigma \over \e}} + { C \over \sigma} \LN J(u) - \pi \delta_0 \RN_{\dot{W}^{-1,1}(B_\sigma)} .
\end{equation*}
\end{lemma}

We now present the

\begin{proof}[Proof of Theorem~\ref{CHMapprox3}]

From the proof of Proposition~\ref{gammalocalizationprop} in \cite{JSp2},  $C_\star$ satisfies ${\rho_\alpha \over C_\star } \leq {1\over 2}$.  Then choosing 
$\sigma = {\rho_\alpha \over 2  C_\star n^4 }$ we find that 
\begin{equation} \label{chaininequality}
4 \|J(u) - \sum \pi \delta_{\alpha_j} \|_{\dot W^{-1,1}} \leq {\rho_\alpha \over 2 C_\star n^5}  = {\sigma \over n} \leq \sigma \leq {\rho_\alpha \over n K_\star },
\end{equation}
{where $K_\star$ is the constant from Proposition~\ref{gammastabilityprop}.}
Therefore  \eqref{chaininequality} implies that
\begin{equation}  \label{localjacobianboundProofEnergyDelta}
\| J(u) - \pi \delta_{\alpha_j} \|_{\dot{W}^{-1,1}(B_\sigma(\alpha_j))} \leq {\sigma \over 4}.
\end{equation} 

1.  Given the choice of $\sigma$ we claim that the following bounds hold
\begin{align}\label{energyexcessballFirstest} 
- {C\over n}&  \leq D_{B_\sigma(\alpha_j)} \leq D(\alpha)  + C \\
\label{energyoutsidevortexballs} 
\int_{\Omega_\sigma} e_\e(u)  &  \leq \pi n \log {1\over \sigma} + D(\alpha) + {W (\alpha) } + C  . 
\end{align}

In order to prove \eqref{energyexcessballFirstest} we note that \eqref{Ebds}, \eqref{chaininequality}, \eqref{localjacobianboundProofEnergyDelta} and Lemma~\ref{JSpSingleVortexLower} imply 
\begin{align*}
0 & \leq { D_{B_\sigma (\alpha_j)} + C {\e \over \sigma} \sqrt{\log{\sigma \over \e}} 
+ {C \over \sigma} \LN J(u) - \pi  \delta_{\alpha_j} \RN_{\dot{W}^{-1,1}(B_\sigma(\alpha_j))} } \\
&  \leq {D_{B_\sigma(\alpha_j)}  + C \e {n^5 \over \rho_\alpha} \logep 
+ { C \over n}  \leq D_{B_\sigma(\alpha_j)} +  {C\over n}. }
  \end{align*}    

To prove the upper bound we use the following inequality
that can be found in the proof of Theorem~3 in \cite{JSp2}.
\begin{equation*} \label{excessenergyfixedtimevortexballs} \begin{split}
D(\alpha)
& = \int_{\Omega_\sigma(\alpha)} [e_\e(u) - e_\e(u_\star) ] + \sum_{j=1}^n D_{B_\sigma(\alpha_i)} + O ( 
( { n \sigma \over \rho_\alpha} )^2 ) \\
& \geq \int_{\Omega_\sigma(\alpha)} e_\e(|u|) + {1\over 4} \LV { j(u) \over |u|} - j(u_\star) \RV^2 
+ \sum_{j=1}^n D_{B_\sigma(\alpha_i)} - C
\end{split}
\end{equation*}
since $\sigma = {\rho_\alpha \over 2 C_\star n^4} \leq {\rho_\alpha \over n}$.  
Thus $\sum_{i=1}^n D_{B_\sigma(\alpha_i)} \leq D(\alpha) + C$, and hence 
\begin{equation} \label{excessballestimateviaglobal}
D_{B_\sigma(\alpha_i)} \leq D(\alpha) + C,
\end{equation}
since $D_{B_\sigma(\alpha_j)}  \geq - {C \over n}$ for each $j\in \{1,\ldots,n\}$.  This 
finishes the proof of \eqref{energyexcessballFirstest}.

For \eqref{energyoutsidevortexballs} we again write $D(\alpha) = \int_{\Omega_\sigma} e_\e(u) - e_\e(u_\star)  
+ \sum D_{B_\sigma(\alpha_i)} + O ( 
( { n \sigma \over \rho_\alpha} )^2 )$ and use Lemma 12 of \cite{JSp2} to estimate 
\begin{align*}
\int_{\Omega_\sigma} e_\e(u)  & = D(\alpha) - \sum D_{B_\sigma(\alpha_i)}  + \int_{\Omega_\sigma} e_\e(u_\star)  
+  O (  ({ n \sigma \over \rho_\alpha} )^2 ) \\
& = D(\alpha) - \sum D_{B_\sigma(\alpha_i)}  + O ( ( { n \sigma \over \rho_\alpha} )^2 ) + \pi n \log {1\over \sigma} +  
{W(\alpha)}  + O({n^3 \sigma^2 \over \rho_\alpha^2}) \\
& \leq \pi n \log {1\over \sigma} + D(\alpha) + {W(\alpha)} + C  
\end{align*}
where we again use that $\sigma= {\rho_\alpha \over 2 {C_\star} n^4 }$.   

2.  From Lemma~\ref{singlevortexfirstmomentlemma}, \eqref{energyexcessballFirstest}, and \eqref{localjacobianboundProofEnergyDelta} there exists a $\xi_j$ in $B_{\sigma/2}(\alpha_j)$ for each 
$j = \{1, \ldots, n\}$ such that 
\begin{equation}  \label{proofenergyDeltaSingleMomentSigma}
\int_{B_\sigma(\alpha_j)} |x - \xi_j| { e_\e (u)}  \leq C \sigma \LC D(\alpha) + C \RC.
\end{equation}
Next we choose an arbitrary $\phi \in W_0^{1, \infty}(\Omega)$, then
\begin{align*}
 \LV \int_\Omega \phi \LC { e_\e(u) \over \logep} - \pi \sum \delta_{\xi_j} \RC  \RV  
 &\leq \LV \int_{\Omega_\sigma} \phi { e_\e(u) \over \logep}  \RV 
+ \sum_{j=1}^n \LV \int_{B_\sigma(\alpha_j)} \phi \LC {e_\e(u) \over \logep} - \pi \delta_{\xi_j} \RC  \RV \\
& = A + B .
\end{align*}

We first handle $A$. {Since $\LN \phi \RN_{L^\infty(\Omega)} \leq \operatorname{diam}(\Omega)$, then} from \eqref{energyoutsidevortexballs}
\begin{align}
\LV \int_{\Omega_\sigma} \phi(x) {e_\e(u) \over \logep} \RV
& \leq {\operatorname{diam}(\Omega) \over \logep} \LB \pi n \log{1\over \sigma} +  D(\alpha) + {W(\alpha)} + C \RB.
\label{boundoutsidevortexenergydensity}
\end{align}

Next we estimate $B$.  Without loss of generality assume $\alpha_j =0$ and $\xi_j = \xi$ and again choosing 
$x_0 \in \p \Omega$, then
\begin{align*}
\LV \int_{B_\sigma} \phi \LC  {e_\e (u) \over \logep} - \pi \delta_\xi \RC  \RV
& \leq \LV \int_{B_ \sigma} \LC \phi(x) - \phi(\xi) \RC {e_\e(u) \over \logep}  \RV
+ \LV \phi(\xi)\RV \LV \int_{B_ \sigma} {e_\e (u) \over \logep  } - \pi \delta_\xi  \RV \\
& \leq \LV \int_{B_ \sigma} |x - \xi |{e_\e(u) \over \logep}  \RV 
+{\operatorname{diam}(\Omega)} \LV \int_{B_ \sigma} {e_\e (u) \over \logep  } - \pi \delta_\xi  \RV \\
& = B_1 + B_2.
\end{align*}
From \eqref{proofenergyDeltaSingleMomentSigma} we have 
\[
B_1 \leq {C \over \logep} \sigma \LC D(\alpha) + C \RC.
\]
Whereas, $\LV \int_{B_\sigma} {e_\e(u)} - \pi \logep \RV \leq \pi \log {1\over \sigma} + D_{B_\sigma} + \gamma$ implies
\[
B_2 \leq {\operatorname{diam}(\Omega) \over \logep} \LB \log{1\over \sigma} + D(\alpha) + C \RB.
\]
Since $\sigma = {\rho_\alpha \over K_2 n^4} \leq 1$ then
\[
B_1 + B_2 \leq {C \over \logep} \LB \log{K_2 n^4 \over \rho_\alpha}  + D(\alpha) + C \RB.
\]

Combining with the bound \eqref{boundoutsidevortexenergydensity} on $A$ with the bounds on $B$ yields
\begin{align*}
\LV \int_\Omega \phi \LC {e_\e(u) \over \logep} - \sum \pi \delta_{\xi_j} \RC  \RV
& \leq {C \over \logep} \LB n (D(\alpha) +  \log{ n^4 \over \rho_\alpha } + C) +{W(\alpha)} \RB,
\end{align*}
and  \eqref{approxenergymeasuredelta} follows.

\end{proof}

\section{Quantitative bounds on the kinetic energy}
\label{sec:quantkinetic}
We now present a kinetic energy bound for fixed $\e$.  Similar bounds with errors of the form $o_\e(1)$ can be found 
in \cite{Jwave, LinWave, SSProd}.
Our method of proof is inspired by the choice of test function found in the proof of Theorem 3 in \cite{KMM}.

\begin{theorem} \label{thm:kineticenergybound}
Let $u(t)$ be a smooth solution to \eqref{tdgl} and $a(t)$ a solution to \eqref{ODE} on $[0,T]$
for some $T\gtrsim 1$ with $\rho_{a(t)} \geq \rho_\star$ for all $0 \leq t \leq T$ {and assume}
\begin{equation} \begin{split}  \label{kineticenergyassumptions}
& \LN J(u(t)) - \sum_{j=1}^n \pi \delta_{a_j(t)} \RN_{\dot{W}^{-1,1}}
\leq \frac{\rho_\star}{8 C_\star n^5}, \\
& D(a(t)) \leq 1, \ 
 \ \hbox{ and } \  {n^{14} \over \rho_\star^4} \leq \logep
\end{split}
\end{equation}
with $C_\star$ the constant from Proposition~\ref{gammalocalizationprop}.
Then there exists $\xi(t) = (\xi_1(t),\ldots, \xi_n(t))$ such that
\begin{equation} \label{goodpositionskineticenergyproof}
\begin{split}
  \LN {e_\e(u(t))\over \logep} - \sum_{j=1}^n \pi  \delta_{\xi_j(t)} \RN_{\dot W^{-1,1}} & \lesssim {n^2  \over \rho_\star^2 \logep},
  \end{split}
\end{equation}
and {for any $0 \leq t_1 < t_2 \leq T$}
\begin{equation}
  \label{eq:mainkinbd}
\pi\int_{t_1}^{t_2}\sum_{j=1}^n |\dot a_j|^2  
 \le 
\int_{t_1}^{t_2} \int_\Omega  { |\p_t u|^2 \over \logep}    
 \quad  + C \A_\e \LB  \sup_{t \in[t_1,t_2]} \sum_j \LV \xi_j(t) - a_j(t) \RV + {1\over \logep^{1\over2}} \RB,
\end{equation}
where
\begin{equation*} \label{Aepsilondefinition}
\begin{split}  
\A_\e & := { n^3  T \over \rho_\star^3}   
\end{split}
\end{equation*}
and 
 $C$ depends only on $\Omega$ {and $\varphi_\star$}.

Furthermore, if $D(a(0)) \lesssim { n^3 T \over \rho_\star^3 \logep^{1\over2}}$, 
then
\begin{equation}  \label{excessenergyestimateviaetacontrol}
D(\xi(t)) \lesssim  \A_\e \LB \sup_{s\in[0,t]} \sum_{j=1}^n |\xi_j(s) - a_j(s)| +  \logep^{-{1\over2}} \RB ,
\end{equation}
and
\begin{align}
\label{gstab.ref1}
\int_{\Omega_{\rho_\star}(\xi)} e_\e(|u|) + \LV { j(u) \over |u|} - j(u_\star)(\xi(t)) \RV^2  & \lesssim
\A_\e  \LB \sup_{s\in[0,t]} \sum_{j=1}^n |\xi_j(s) - a_j(s)| +  \logep^{-{1\over2}} \RB,   \\
\label{gstab.ref2}
\LN j(u) - j(u_\star(\xi,t)) \RN_{L^{4\over3}} & \lesssim \sqrt{ \A_\e \LB \sup_{s\in[0,t]} \sum_{j=1}^n |\xi_j(s) - a_j(s)| +  \logep^{-{1\over2}} \RB } .
\end{align}

\end{theorem}

We prove a slightly stronger fact that the kinetic energy, localized at the vortex balls, is bounded below by the ODE kinetic energy, see \eqref{strongerkineticenergylowerbound} below.  

A similar theorem was proved in \cite{KurzkeSpirnMRI} for a single vortex that stays an $O(1)$ distance from the boundary for an $O(1)$ time.  Here we prove a much more explicit estimate.  The major tool to establishing a finite-$\e$ bound on the kinetic energy is the following optimal result on the equipartitioning of Ginzburg-Landau energy, which improves related results in \cite{LinWave}
and \cite{SSProd}.

\begin{proposition} [Kurzke-Spirn \cite{KurzkeSpirnMRI}] \label{stressenergy}
Suppose $\LN J(u) - \pi \delta_0 \RN_{\dot{W}^{-1,1}(B_\sigma)} \leq {\sigma \over4} $ and $\int_{B_\sigma} e_\e (u)  \leq \pi \log {\sigma \over \e} + K_0$ then
\begin{equation} \label{stresstensorest}
\LV {1 \over 2} \int_{B_\sigma} \begin{pmatrix} \LV \p_{x_1} u \RV^2 &  \LC \p_{x_1} u , \p_{x_2} u \RC  \\  \LC \p_{x_1} u , \p_{x_2} u \RC  & \LV \p_{x_2} u \RV^2 \end{pmatrix} 
 - \begin{pmatrix} {\pi \over 2} \log {\sigma \over \e} & 0 \\ 0 & {\pi \over 2} \log {\sigma \over \e}  \end{pmatrix} \RV
\leq \sqrt{K_1 \log {\sigma \over \e} } 
\end{equation}
where $K_1 = C (C + K_0) e^{K_0/\pi}$ and $C$ is a universal constant.
\end{proposition}

We apply this equipartitioning result to the evolution identity for the
 energy
and  deduce a rate of convergence for the kinetic energy.

%
\begin{proof}[Proof of Theorem~\ref{thm:kineticenergybound}]
To prove this estimate we first use the hypotheses to extract better vortex positions.  We then use the differential identity \eqref{eqn:conservationenergy} along with a special test function to prove the kinetic energy bounds.  

1. 
We first prove a pair of crude bounds that enable us to use Theorem~\ref{thm:energyisdelta} in the previous section.  From \eqref{W.scale} we find that $ |W(a(t))| \leq C \LC n^3 + {n^2 \over \rho^2_\star} \RC$.  Therefore, for any $0\leq t \leq T$ we have 
$E_\e(u(t)) = D_\e(t) + W_\e(a(t))  \lesssim 1 + n^3+ {n^2 \over \rho_\star^2} + n  \logep $.  From    \eqref{kineticenergyassumptions} we have the very crude bounds $n, \rho_\star^{-1} \lesssim \logep$; and hence, $E_\e(u(t)) \lesssim \logep^4$.  
As $n\ge 1$, we may additionally assume $E_\e(u(t))\ge 1$ for all times. 
We easily see that 
\begin{equation}  \label{largeenergyprevention}
{n^5 \over \rho_\alpha} E_\e(u(t)) + {n^{10} \over \rho_\alpha^2} \sqrt{E_\e(u(t))} \lesssim \logep^{13} \leq {1\over\e}.
\end{equation}

Set \begin{equation*} \label{sigmadefinition}
\sigma = {\rho_\star \over 2 C_\star n^4}, 
\end{equation*}
then by \eqref{kineticenergyassumptions} and \eqref{largeenergyprevention} and for each $t_1 \leq t \leq t_2$ we can use  Proposition~\ref{gammalocalizationprop} and Theorem~\ref{thm:energyisdelta}.  In particular, 
 for each $t_1 \leq t \leq t_2$ there exists a $\xi(t) = (\xi_1(t), \ldots, \xi_n(t))$ such that 
 \eqref{goodpositionskineticenergyproof} holds with $|\xi_j -a_j| \leq \sigma$ for each $j = 1,\ldots, n$.  
 By \eqref{chaininequality}, $B_{2\sigma}(a_j(s))\cap B_{2\sigma}(a_k(s))=\emptyset$
for all $s\in[0,t]$
unless $j=k$, and $B_{2\sigma}(a_j(s))\cap\p\Omega=\emptyset$ for all $j$.

%
%
%

Next we prove an estimate on the kinetic energy.  Conservation of energy implies
\begin{align*}\label{eq:ptbound}
\int_0^t \int_\Omega { \LV  \p_t u \RV^2 \over \logep}  & 
=  E_\e(u_0) - E_\e(u(t))  = D (a(0)) + W_\e(a_0) - W_\e(a(t)) - D(a(t))\\
& = D (a(0)) + W(a_0) - W(a(t)) - D(a(t)).
\end{align*}
By a similar argument as above and using {$n \lesssim \rho_\star^{-2}$}, we find that 
\begin{equation}  \label{kineticenergyboundunboundedvortices}
\int_0^t \int_\Omega { \LV \p_t u \RV^2 \over {\pi\logep}}   \lesssim   {n^2 \over \rho^2_\star}  .
\end{equation}
%
%

2.  We now make the following claim.
Let $\chi\in C^\infty_c(\R^2)$ be a function such that 
$\chi\ge 0 $, $\chi\equiv 1$ on $B_\sigma(0)$, $\chi\equiv 0 $ in $\R^2\setminus B_{2\sigma}(0)$ 
{and $|D^k\chi|\le C \sigma^{-k}$ for $k=1,2$
for some constant $C$, then } 
\begin{equation}
  \label{eq:p1r}\begin{split}
& \LV {1 \over \logep} \int_{t_1}^{t_2} \int_\Omega \sum_{j=1}^n \chi(x-a_j(t)) \ \dot a_j(t)\cdot (\p_t u, \nabla u)   +  \pi \int_{t_1}^{t_2}\sum_{j=1}^n |\dot a_j(t)|^2  \RV 
\\
& \lesssim  { n^3 T \over \rho_\star^3}  \sup_{t\in[t_1,t_2]} \sum_{j=1}^n \LV \xi_j(t) - a_j(t) \RV 
 + {   n^{10} T \over  \rho_\star^5 \logep } + {  n^3  T \over \rho^3_\star  \logep^{1\over 2}}  
.\end{split}
\end{equation}

For any test function $\phi\in C^2([0,T]\times\Omegab)$ with compact support in $\Omega$ and any $0\le t_1\le t_2\le T$
we have 
\begin{equation} \begin{split} \label{eq:withphi}
&  \int_{t_1}^{t_2} \int_\Omega \p_t \LC \phi { e_\e(u) \over \logep } \RC  - 
  \int_{t_1}^{t_2}  \int_\Omega (\p_t \phi) { e_\e(u) \over \logep}   
  \\ 
 & = - \frac{1}{\logep}  \int_{t_1}^{t_2}  \int_\Omega \phi { |\p_t u|^2 \over \logep}  
  - {1\over \logep} \int_{t_1}^{t_2}  \int_\Omega \nabla \phi \cdot (\p_t u, \nabla u)  ,
\end{split} \end{equation}
as is easily seen by multiplying
 \eqref{eqn:conservationenergy} by $\ilogep$ and integrating by parts.

We now follow \cite{KMM} and set 
\[
\phi(t,x)= \sum_{j=1}^n \chi(x-a_j(t)) \ \dot a_j(t)\cdot(x-a_j(t))  .
\]
 Then we calculate, dropping
the $t$-dependence of $a$,
\begin{align*}
\nabla \phi(t,x) &= \sum_{j=1}^n (\dot a_j\cdot(x-a_j))\nabla \chi(x-a_j)+ \dot a_j \ \chi(x-a_j),
\\
\p_t \phi(t,x) &= \sum_{j=1}^n -\dot a_j \cdot \nabla\chi(x-a_j) \ \dot a_j\cdot(x-a_j)
\\ 
&\quad+ \chi(x-a_j) \left( \ddot a_j \cdot (x-a_j)-\dot a_j\cdot \dot a_j \right),
\\
\nabla \p_t \phi(t,x) &= \sum_{j=1}^n -\dot a_j \cdot \nabla^2\chi(x-a_j) \ \dot a_j\cdot(x-a_j)
- \dot a_j \cdot \nabla \chi(x-a_j) \ \dot a_j \\
&\quad + \nabla \chi(x-a_j) \left( \ddot a_j\cdot (x-a_j) - |\dot a_j|^2 
\right) +\chi(x-a_j) \ \ddot a_j .
\end{align*}
We first note that $\dot a_j = - {1\over \pi} \nabla_{a_j} W$ which implies $\ddot a_j = {1\over \pi^2} \nabla_{a_k} W \nabla_{a_j, a_k} W$.  Therefore,
\begin{align*}
\sup_j \LN  \dot a_j \RN_{L^\infty_T} & \lesssim {n \over \rho_\star} , \\
\sup_j \LN  \ddot a_j \RN_{L^\infty_T} & \lesssim {n^3 \over \rho^3_\star} ,
\end{align*}
and these estimates imply the following bounds 
\begin{align*}
\norm{ \phi}_{L^\infty_T L^\infty_\Omega}&\lesssim {n  \sigma} \sup_j \LN \dot a_j \RN_{L^\infty_T   } \lesssim {1\over n^2} , \\
\norm{\nabla \phi}_{L^\infty_T L^\infty_\Omega}&\lesssim n \sup_j \LN \dot a_j \RN_{L^\infty_T   } \lesssim {n^2 \over \rho_\star}  , \\
\norm{\nabla\p_t \phi}_{L^\infty_T L^\infty_\Omega}&\lesssim {n \over \sigma} \sup_j \LN \dot{a}_j \RN^2_{L^\infty_T   } + n \sup_j \LN \ddot{a}_j \RN_{L^\infty_T   } \lesssim {n^7 \over \rho_\star^3}.
\end{align*}
Now we analyze the terms in \eqref{eq:withphi} one by one.
We have by \eqref{goodpositionskineticenergyproof} 
\begin{align*}
& \LV \int_{t_1}^{t_2}  \int_\Omega \p_t \LC \phi{ e_\e(u) \over \logep} \RC   
 - \LB \int_\Omega \phi(t_2,\cdot) \LC \pi \sum \delta_{\xi_j(t_2)} \RC- \phi(t_1,\cdot) \LC \pi \sum \delta_{\xi_j(t_1)} \RC  \RB  \RV  \\
&\leq 2 \sup_{t\in[t_1,t_2]} \LV  \int_\Omega \phi(t,\cdot) \LC {e_\e(u(t,\cdot)) \over \logep} - \pi \sum \delta_{\xi_j(t)} \RC   \RV \\
& \lesssim \LN \nabla \phi \RN _{L^\infty_T L^\infty_\Omega} \sup_{t \in [t_1,t_2]} \LN {e_\e(u(t)) \over \logep} - \pi \sum_{j=1}^n \delta_{\xi_j(t)}  \RN _\Wdot 
 \lesssim  {n^4 \over \rho_\star^3 \logep} .
\end{align*}
On the other hand 
\begin{align*}
& \LV  \sum_{j=1}^n \dot a_j(t_2) \cdot \LC \xi_j(t_2) - a_j(t_2) \RC - \dot a_j(t_1) \cdot \LC \xi_j(t_1) - a_j(t_1) \RC \RV \\
& \lesssim  {n \over \rho_\star} \sup_{t\in [t_1,t_2]} \sum_{j=1}^n \LV \xi_j(t) - a_j(t) \RV .
\end{align*}
Therefore,
\begin{equation}  \label{keest1}
\LV \int_{t_1}^{t_2} \int_\Omega \p_t \LC \phi{ e_\e(u) \over \logep} \RC   \RV \lesssim   
{n \over \rho_\star} \sup_{t\in [t_1,t_2]}  \sum_{j=1}^n \LV \xi_j(t) - a_j(t) \RV + {n^4 \over \rho^3_\star \logep}  .
\end{equation}

For the second term on the left-hand side of \eqref{eq:withphi},
\begin{align*}
& \LV \int_{t_1}^{t_2} \int_\Omega \p_t \phi { e_\e(u) \over \logep}   - \int_{t_1}^{t_2} \int_\Omega \p_t \phi  \sum_{j=1}^n \pi  \delta_{\xi(t) }   \RV \\
& \leq T \LN \nabla \p_t \phi \RN_{L^\infty_T L^\infty_\Omega} \sup_{t\in[t_1,t_2]} \LN {e_\e(u(t,\cdot)) \over \logep} - \sum_{j=1}^n \pi \delta_{\xi(t)} \RN_\Wdot  \lesssim { n^{10} T \over \rho_\star^5 \logep}
\end{align*}
and
\begin{align*}
& \LV \int_{t_1}^{t_2} \int_\Omega \p_t \phi  \sum_{j=1}^n \pi  \delta_{\xi_j}   + \pi \sum_{j=1}^n \int_{t_1}^{t_2} |\dot a_j |^2  \RV   = \pi  \LV  \sum_{j=1}^n \int_{t_1}^{t_2} \ddot a_j \cdot (\xi_j -a_j)  \RV \\
& \lesssim { n^3 T \over \rho_\star^3}
\sup_{t\in[t_1,t_2]}\sum_{j=1}^n \LV \xi_j(t) - a_j(t) \RV .
\end{align*}
Thus
\begin{equation} \label{keest2}
\begin{split}
& \LV \int_{t_1}^{t_2} \int_\Omega \p_t \phi { e_\e(u) \over \logep}    + \pi \sum_{j=1}^n \int_{t_1}^{t_2} |\dot a_j|^2   \RV \\
& \lesssim { n^3 T \over \rho_\star^3}
\sup_{t\in[t_1,t_2]}\sum_{j=1}^n \LV \xi_j(t) - a_j(t) \RV + {n^{10} T \over \rho_\star^5 \logep}   .
\end{split}
\end{equation}
Note that the previous equality {contains the second term of the left-hand side of \eqref{eq:p1r}. }

For the first term on the right-hand side of \eqref{eq:withphi} we use \eqref{kineticenergyboundunboundedvortices} and get 
\begin{equation}  \label{keest3}
\frac{1}{\logep} \int_{t_1}^{t_2} \int_\Omega \phi {|\p_t u|^2\over \logep}  
\lesssim \LN \phi \RN_{L_T^\infty}  {n^2 \over \logep \rho^2_\star}
\lesssim {1\over \rho_\star^2 \logep} .
\end{equation}

Finally, for the second term on the right-hand side of \eqref{eq:withphi} we have
\begin{equation}\label{keest4} \begin{split}
&\ilogep \int_{t_1}^{t_2} \int_\Omega \nabla \phi \cdot (\p_t u, \nabla u) \\
& \qquad  =
\sum_{j=1}^n \ilogep \int_{t_1}^{t_2} \int_{B_{2\sigma(a_j(t))}}
\nabla \chi(x-a_j)\cdot (\p_t u,\nabla u) \dot a_j\cdot (x-a_j)  
\\
&\qquad \quad + \sum_{j=1}^n \ilogep \int_{t_1}^{t_2} \int_\Omega \chi(x-a_j) \dot a_j\cdot (\p_t u, \nabla u) .
\end{split} \end{equation}
We note that the second term on the right-hand side of \eqref{keest4} is precisely the first term  on the left-hand side of
\eqref{eq:p1r}. We estimate the other term. Using  the Cauchy-Schwarz inequality
\begin{align*}
& \sum_{j=1}^n \ilogep \LV \int_{t_1}^{t_2} {\int_{B_{2\sigma(a_j(t))}} }
\nabla \chi(x-a_j)\cdot (\p_t u,\nabla u) \ \dot a_j\cdot (x-a_j)  \RV \\
& \leq  {\sigma\over \logep^{1\over2}}\sup_j \LN  \dot a_j \RN_{L^\infty_T  } \LN \nabla \chi(x - a_j) \RN_{L^\infty_T L^\infty_\Omega} \\
& \qquad\sum_{j=1}^n  \left(\int_{t_1}^{t_2} \int_ {B_{2\sigma}(a_j(t))\setminus B_\sigma(a_j(t))} {|\p_t u|^2 \over \logep}
   \right) ^{1\over 2}
 \left(\int_{t_1}^{t_2} \int_{B_{2\sigma}(a_j(t))\setminus B_\sigma(a_j(t))} 
{|\nabla u|^2}  
 \right) ^{1\over2} \\
 & \lesssim {n \over \rho_\star \logep^{1\over2}}  \LB \int_{t_1}^{t_2} \int_\Omega{|\p_t u|^2 \over \logep} 
 + \int_{t_1}^{t_2} \int_{\Omega_{\sigma}(a_j(t))} {|\nabla u |^2} \RB  \\
 & \lesssim {n \over  \rho_\star \logep^{1\over2}} \LB {n^2 \over \rho^2_\star} + {n^2 \over \rho^2_\star} |t_2 - t_1|  \RB  \lesssim {n^3 T  \over \rho^3_\star  \logep^{1\over2}} .
\end{align*}

3.  We now study the momentum term on the left hand side of \eqref{eq:p1r}.  From Cauchy-Schwarz
\begin{multline*}
\ilogep\int_{t_1}^{t_2} \int_\Omega \sum_{j=1}^n \chi(x-a_j) \  \dot a_j \cdot (\p_t u,\nabla u) \\
\leq \left(\ilogep\int_{t_1}^{t_2} \int_\Omega \sum_{j=1}^n \chi(x-a_j) |\p_t u|^2    \right)^{1\over2}
\left(\ilogep\int_{t_1}^{t_2} \int_\Omega \sum_{j=1}^n \chi(x-a_j) (\dot a_j \otimes \dot a_j) : (\nabla u \otimes \nabla u)    \right)^{1\over2},
\end{multline*}
where $(b\otimes b)_{ij} = b_i b_j$ for $b\in \R^2$ and $\LC \nabla u \otimes \nabla u \RC_{ij} = \LC \p_i u , \p_j u \RC$ for $u \in \C$.
For any $\dot a_j \in \R^2$ and $\chi$ as above, we claim that
\begin{equation}
\label{eq:equipa}
\begin{split}
& \LV \ilogep \int_{t_1}^{t_2} \int_\Omega \sum_{j=1}^n \chi(x-a_j) (\dot a_j\otimes \dot a_j) : (\nabla u\otimes \nabla u)   
- \int_{t_1}^{t_2} \sum_{j=1}^n \pi |\dot a_j|^2  \RV \\
&  \lesssim {  n^3 T \over \rho_\star^2 \logep^{1\over2}} .
\end{split}
\end{equation}
Indeed for any time $t_1 \leq t \leq t_2$ we find:
\begin{align*}
& \LV {1\over \logep} \int_\Omega \sum_{j=1}^n \chi (x - a_j) \dot a_j \otimes \dot a_j : \nabla u \otimes \nabla u 
-\sum_{j=1}^n \pi \LV \dot a_j \RV^2\RV \\
& \leq \LV  \int_\Omega \sum_{j=1}^n \chi (x - a_j) \dot a_j \otimes \dot a_j : { \nabla u \otimes \nabla u \over \logep}   \right. \\
& \qquad \left. - \sum_{j=1}^n \int_{B_{\sigma}(\xi_j(t) )}  |\dot a_j^x|^2 { |\p_x u|^2 \over \log {\sigma\over \e} } +  |\dot a_j^y|^2 { |\p_y u|^2 \over \log {\sigma\over \e}}  \RV \\
& \quad + \sum_{j=1}^n \LV \LC \int_{B_{\sigma}(\xi_j(t) )}  |\dot a_j^x|^2 { |\p_x u|^2 \over \log {\sigma\over \e}} +  |\dot a_j^y|^2 { |\p_y u|^2 \over \log {\sigma\over \e}}  \RC 
-  \pi \LV \dot a_j \RV^2 \RV  \\
& = I_1 + I_2   .
\end{align*}
First we analyze $I_1$.  {From \eqref{excessballestimateviaglobal} and $D(a(t)) \leq 1$ then Proposition~\ref{stressenergy} is applicable with $K_0 \lesssim 1$ since $|\int_{B_\sigma(\alpha_j)} e_\e(u) - \pi \log{\sigma \over \e} |\leq D_{B_\sigma(\alpha_j)} + \gamma$.  Choosing   $\sigma = \rho_\star \gg \e$, then \eqref{stresstensorest} implies}  
\begin{align*}
I_1 & \leq {2\over \logep} \sum_{j=1}^n \int_{B_{2\sigma}(a_j)} \chi(x - a_j) \LV \dot a_j^x \dot a_j^y \LC \p_x u , \p_y u \RC \RV  \\
& \quad +  \sum_{j=1}^n \int _{B_{2\sigma}(a_j) \backslash B_\sigma(\xi_j)} \chi(x-a_j) \LB { |\dot a_j^x|^2 |\p_x u|^2 \over \logep}
+ { |\dot a_j^y|^2 |\p_y u|^2 \over \logep} \RB      \\
& \quad + \LV 1 -  { \log {1\over \e} \over \log {\sigma \over \e} } \RV  \sum_{j=1}^n \int _{ B_\sigma(\xi_j)} \LB { |\dot a_j^x|^2 |\p_x u|^2 \over \logep} + { |\dot a_j^y|^2 |\p_y u|^2 \over \logep} \RB      \\
& \lesssim {n \over \logep} \sup_j \LN \dot a_j \RN^2_{L^\infty_T  }  \sqrt{ \log {\sigma \over \e}} + {\sup_j \LN \dot a_j \RN^2_{L^\infty_T  } \over \logep} \int_{t_1}^{t_2} \int_{\Omega_\sigma(a(t))} {\LV \nabla u \RV^2 }  \\
& \quad + {\log {1\over \sigma} \over \logep} \sup_j \LN \dot a_j \RN^2_{L^\infty_T  }   \sum_{j=1}^n \int _{ B_\sigma(\xi_j)} {\LV \nabla u \RV^2 \over \log {\sigma \over \e}}  \\
& \lesssim    \LB { n^3  \over  \rho_\star^2 \logep^{1\over 2}} +  {n^4 \over \rho^4_\star \logep} + {n^3  \log{1\over \sigma} \over \rho_\star^2 \logep} \RB  .
\end{align*}
Next we look at $I_2$.  Again {from \eqref{excessballestimateviaglobal} and \eqref{stresstensorest}} and since $\sigma \gg \e$, 
\begin{align*}
I_2 & \lesssim \sum_{j=1}^n |\dot a_j|^2 \LC \log { \sigma \over  \e} \RC^{-{1\over2}}  \lesssim {n^3 \over \rho_\star^2  \logep^{1\over 2}}     .
\end{align*}
Comparing the terms from $I_1$ and $I_2$ results in estimate \eqref{eq:equipa}.
Finally, we combine \eqref{eq:withphi} with \eqref{keest1}-\eqref{eq:equipa} { which yields} \eqref{eq:p1r}.

4.  
Using \eqref{eq:p1r} and assumptions \eqref{kineticenergyassumptions} we have
\begin{align*}
\left|\pi \int_{t_1}^{t_2} \sum_{j=1}^n \LV \dot a_j \RV^2  - F_1\right|
\le \left( \int_{t_1}^{t_2}  \int_\Omega 
\sum_{j=1}^n\chi(x-a_j) { |\p_t u|^2 \over \logep}    \right)^{1\over 2}
\left(\pi \int_{t_1}^{t_2} \sum_{j=1}^n|\dot a_j|^2  + F_2
\right)^{1\over 2}
\end{align*}
where 
\begin{align*}
F_1 & =  C {n^3 T  \over \rho_\star^3 \logep^{1\over2}} + C {n^3 T \over \rho_\star^3}\sup_{s \in [t_1,t_2]} \sum_j^n\LV \xi_j(s) - a_j(s) \RV  \\
F_2 & = C {n^3 T \over \rho_\star^2 \logep^{1\over 2}}  .
\end{align*}
We square the previous inequality, obtaining by division
\[
\frac{\left(\pi \int_{t_1}^{t_2} \sum_{j=1}^n|\dot a_j|^2  - F_1\right)^2}
{\pi \int_{t_1}^{t_2} \sum_{j=1}^n|\dot a_j|^2 +F_2}
\le \ilogep\int_{t_1}^{t_2} \int_\Omega \sum_{j=1}^n
\chi(x-a_j) |\p_t u|^2   .
\]
Setting $K=\pi \int_{t_1}^{t_2}\sum_{j=1}^n |\dot a_j|^2 $, 
we have using
\[
\frac{(K-F_1)^2}{K+F_2} =K+F_2 -2(F_1+F_2)+\frac{(F_1+F_2)^2}{K+F_2}
\ge K-F_1-2F_2
\]
that  
\begin{equation}  \label{strongerkineticenergylowerbound}
\ilogep\int_{t_1}^{t_2} \int_\Omega \sum_{j=1}^n\chi(x-a_j) |\p_t u|^2  
\ge \pi \int_{t_1}^{t_2}\sum_{j=1}^n |\dot a_j|^2  - CF_1-CF_2,
\end{equation}
and so \eqref{eq:mainkinbd} follows, since $F_1 \gtrsim F_2$.

5.  We next will show that $u(t)$ is well-approximated in certain ways 
by the canonical harmonic map {$u_\star(t) := u_\star(\cdot; \xi(t))$}
for $t \leq t_2$. To do this, we need to estimate the surplus energy
$D( \xi(t))$ with respect to the points $\xi(t)$
found in Step 1 above. 
Assuming $D(a(0)) \leq {n^3 T \over\rho_\star^3 \logep^{1\over2}}$ then
{by \eqref{excessenergyidentity} and \eqref{eq:mainkinbd} we have}
\begin{align*}
D (\xi(t))
& = 
D(a(t)) + W(a(t),d) - W(\xi(t),d)
\nonumber \\
& \lesssim   \A_\e \LB { \sup_{s\in[0,t]} \sum_{j=1}^n \LV \xi_j(s) - a_j(s) \RV }  + { \logep^{-{1\over2}}} \RB   \\
& \quad +
\left( {  \sup_{s \in [0,t]} \sum_{j=1}^n |\xi_j(s) - a_j(s)| } \right)\  
(\sup_j \sup_{ |y-a(t)| \le |\xi(t)-a(t)| }  |\nabla_{y_j} W(y)| ) 
\label{Sigma.est0}\end{align*}
where $\A_\e = {n^3 T \over \rho_\star^3}$.
If  $y\in \Omega^n$ is such that
$|y - a(t)|\le |\xi(t)- a(t)|$, so
$\rho_y \ge \frac 12 {\rho_{a(t)}}$ and 
\begin{equation} \label{excessenergyDxibound}
D(\xi(t))
 \le  C\A_\e \LB  \sup_{s\in[0,t]} \sum_{j=1}^n |\xi_j(s) - a_j(s)|  +  \logep^{-{1\over2}} \RB.
\end{equation}
which implies \eqref{excessenergyestimateviaetacontrol}.
Furthermore, we have
\begin{align*}
& \int_{\Omega_{\rho_\star} (\xi(t))} e_\e(|u(t)|)  + 
\frac 14\LV { j (u(t)) \over |u(t)|} - j( u_\star(t) )\RV^2 \\
& \le 
 C \A_\e \LB  \sup_{s\in[0,t]} \sum_{j=1}^n |\xi_j(s) - a_j(s)| +  \logep^{-{1\over2}} \RB +
C 
\left(\frac{n^5}{\rho_\star}( \e {n^5 \over \rho_\star} + \e E_\e(u))  \right)^{1\over 2}\\
& \le 
C \A_\e \LB  \sup_{s\in[0,t]} \sum_{j=1}^n |\xi_j(s) - a_j(s)| +  \logep^{-{1\over2}} \RB
\end{align*}
for all $t\in [0,t_2]$, where we used {Proposition~\ref{gammastabilityprop} and \eqref{excessenergyDxibound} in the first inequality}. {Estimate \eqref{gstab.ref2} follows from \eqref{gstab.ref1} 
by directly following the argument in Step 3 of the proof of Theorem~2 in \cite{JSp2}.   }

\end{proof}

\section{Proof of Theorem~\ref{thm:quantdynamics}}
\label{sec:proofmainth}
To prove Theorem~\ref{thm:quantdynamics},
we will use the energy identity
\eqref{eq:consenwithdivT} to connect PDE and ODE dynamics.
To control the errors, we apply
 the Gronwall inequality 
and continuity arguments that show the theorem is true for 
 longer and longer times. In order to apply Gronwall's inequality,
 we use time averaging to obtain improved estimates.
%
 
\subsection{Assumptions and initial estimates}
\label{subsec:assu}
 We recall the following assumptions:
 \begin{align}
\hbox{ number of vortices } \qquad  & n  \leq \logep^{1\over 200}  \label{nassumption}\\
\hbox{ minimal intervortex distance } \qquad & \rho_\star  \geq \logep^{-{1\over 100}} \label{rhoassumption} \\
\hbox{ total time scale } \qquad & { T \leq  |\log \logep| } \label{Tassumption} \\ 
\hbox{ initial excess energy } \qquad & D(a(0)) \leq  \logep^{-{2\over 5}} . \label{initialexcessassumption}
\end{align}
{Note the time scale, $T$, serves as a coarse bound for the eventual time frame for which we have the vortex motion law and will be used to simplify calculations.}
Additionally, we need the following small quantities:
\begin{align}
\hbox{ time averaging scale } \qquad & \delta_\e  = \logep^{- {1\over 4}} \label{deltaassumption} \\
\hbox{ resolution of vortex location } \qquad & \D_\e  = \logep^{-{1\over 4}}  \label{scriptDassumption} 
\end{align}
and the following composites:
\begin{equation}
\begin{split}
\hbox{ Jacobian localization error } \qquad& s_\e := C \e \left[ \frac{n^5}{\rho_\star} +{E_\e(u_0)}
\right],\\
\hbox{ energy localization error } \qquad& t_\e := {C \over \logep} \LB  n \log{ n^4 \over \rho_\star} +  W(a(0)) \RB .
\end{split}
\label{s.def}
\end{equation}

Since the energy is concentrating at the points $\xi_j(t)$ and the ODE 
gives us vortex positions
 $a_j(t)$, our main objective 
 is to estimate and control $\sum_j|\xi_j(t) - a_j(t)|$.
This is a challenging quantity to work with directly, so following \cite{JSp2}, we define a 
similar quantity that is differentiable and 
has very similar properties.  We set
\begin{equation}
\eta (t) :=
\sum_{j=1}^n |\eta_j(t)|
:= 
\sum_{j=1}^n 
\LV \int_\Omega {e_\e(u) \over \logep} \, \Phi_j \RV 
\label{eta.def}\end{equation}
where
\[
\Phi_j(x,t)  = \varphi(x - a_j(t)),  \quad\quad
\varphi(x)  = x \chi_{\rho_\star}(x)
\]
and $\chi_{\rho_\star}(x) = \chi(\frac{x}{\rho_\star})$  for a fixed  $\chi\in C^\infty_0(\R^2)$
satisfying
$
\chi(x) = \left\{ \begin{array}{ll} 
1 & \hbox{for } |x| \leq 1 \\
0 & \hbox{for } |x|\geq 2 \end{array} \right..
$
The $\Phi_j$'s are supported on $B_{2\rho_\star}(a_j(t))$, so that
$\left\{ \operatorname{supp} \Phi_j(x,t) \right\}$ are pairwise 
disjoint when $\rho_{a(t)} \ge\rho_\star $
and in particular for all  $0\le t \le \tau_1$.
Note that in \cite{JSp2}, the definition is essentially the
same, but uses the Jacobian instead of the energy density.

We recall and 
define a series of time intervals on which our function $u$ is well-behaved in different senses. 
\begin{align} 
\tau_0 &  = \inf \left\{t >0 \hbox { such that }
{\rho_{a(t)}}\le \rho_\star \right\}   \nonumber 
\\
\tau_{max} & = \min\LCB \tau_0, C \sqrt{ \log \logep {\rho_\star^4 \over n^{3}} }\RCB  
\nonumber 
 \\
\tau_1 & = \sup_{t} \LCB 0 \leq t \leq \tau_{max} \hbox { such that }   \| J(u(s)) - \sum_{i=1}^n\pi   \delta_{a_i(s)}\|_{\dot W^{-1,1}} 
\le {    \D_\e } \right.  \label{Te.def} \\
&  \qquad \qquad \hbox{ and } D(a(s)) \leq 1  \hbox { for all } 0 \leq s \leq t \Big\}  \nonumber \\
\tau_2 & = \sup_{t} \LCB 0 \leq t \leq \tau_1 \hbox { such that }   \eta(s) \leq {1\over 2} \D_\e     \hbox { for all } 0 \leq s \leq t \RCB  \nonumber
 .
\end{align}
In Subsection~\ref{subsec:continuity}, we will show that 
$\tau_1=\tau_2=\tau_{max}$.

The definition of  $\tau_1$ implies that 
\begin{equation}\label{eq:rhobound}
\rho_{a(t)}\ge \rho_\star \ge \logep^{-{1\over 100}}
\hbox{ and } 
  \| J(u(t)) - \sum_{i=1}^n\pi d_i \delta_{a_i(t)}\|_{\dot W^{-1,1}}
\le { \D_\e} 
\end{equation}
for all $t\in [0,\tau_1]$.
From \eqref{nassumption} and \eqref{eq:rhobound} we have 
\begin{align*}
\| J(u(t)) - \sum \pi  \delta_{a_i(t)} \|_{\dot W^{-1,1}}
\leq \logep^{-{1\over 4}}
  \leq {1 \over 8 C_\star} \logep^{-{7 \over 200}} 
\leq  {\rho_{a(t)} \over 8 C_\star n^5},
\end{align*}
where $C_\star$ is the constant found in Proposition \ref{gammalocalizationprop}. 
Therefore,  Proposition \ref{gammalocalizationprop} 
and Theorem~\ref{thm:energyisdelta} hold,  so 
 there exist
$\xi(t) = (\xi_1(t),\ldots, \xi_n(t))\in \Omega^{n*}$ such  that 
$|\xi_i - a_i| \le \frac{\rho_{a(t)}}{4}$ for all $i$, and 
\begin{equation}
\begin{split}
& \| J(u) (s)- \sum_{i=1}^n\pi  \delta_{\xi_i(s)}\|_{\dot W^{-1,1}}
\le s_\e \\
& \| {e_\e(u) (s) \over \logep} - \sum_{i=1}^n\pi \delta_{\xi_i(s)}\|_{\dot W^{-1,1}}
\le t_\e.
\end{split}
\label{xis.def}\end{equation}

Given our assumptions and composite quantities, we collect a few useful estimates.
\begin{lemma}
Assuming \eqref{nassumption}-\eqref{initialexcessassumption} then for $0\leq t \leq \tau_1$ 
\begin{align}
W(a(t)) & \lesssim \logep^{3 \over 100},  \label{Wtotalbound} \\
E_\e(u(t)) & \lesssim \logep^{1+ { 1 \over 200}} , \label{Energytotalbound} \\
s_\e &  \lesssim \e^{9\over10},  \label{sepsilonbound} \\
t_\e &  \lesssim \logep^{-{97\over 100}} \label{tepsilonbound},
\end{align}
and for all $0 \leq s \leq t \leq \tau_1$
\begin{equation} \label{kineticenergyboundfiniteinterval}
\int_s^t \int_\Omega { \LV\p_t u \RV^2 \over \logep} \lesssim 1 + \logep^{7 \over 200} |t - s|.
\end{equation}
\end{lemma}
\begin{proof}
By \eqref{W.scale}, since $\rho_{a(t)} \geq \rho_\star$ then \eqref{Wtotalbound} follows from \eqref{nassumption} and \eqref{rhoassumption}. 
Next note that 
\begin{align*}
E(u(t)) & =  W_\e(a(t)) +D(a(t)) = n \LC \pi \logep + \gamma \RC + W(a(t)) + D(a(t)) \\ & \lesssim \logep^{1 + {1\over 200}} 
\end{align*}
from \eqref{nassumption}, \eqref{Wtotalbound}, and the fact that $D(a(t)) \leq 1$ for $0 \leq t \leq \tau_1$.   As a result \eqref{sepsilonbound} and \eqref{tepsilonbound} follow from \eqref{nassumption}, \eqref{rhoassumption}, \eqref{Wtotalbound}, and \eqref{Energytotalbound}.

Finally, \eqref{kineticenergyboundfiniteinterval} follows from \eqref{PDEenergyidentity} and \eqref{ODEenergyidentity}:
\[
\int_s^t \int_\Omega { \LV \p_t u \RV^2 \over \logep} = D(a(s)) - D(a(t)) +  \int_s^t |\dot a|^2
\lesssim 1 + |t -s | {n^3 \over \rho_\star^2}
\]
by the hypotheses on $W(a)$.  
 \end{proof}

We now show that $\eta$ is a good measure for $\sum_j |\xi_j -a_j|$ 
and similar quantities.
\begin{lemma} \label{lemmameasuringmu}
If $ 0 \leq t \leq \tau_1$ then
\begin{align}
\label{eta.rocks1} \LV \eta(t)
-  \sum_{i=1}^n  \pi  |\xi_i(t)-a_i(t)| \RV & \lesssim   t_\e ,  \\
\label{eta.rocks2}\LV \eta(t) -  \|  {e_\e(u(t))\over \logep}  - \sum_{i=1}^n\pi \delta_{a_i(t)}\|_{\dot W^{-1,1}}
 \RV & \lesssim   t_\e , \\
\label{eta.rocks3} \LV \eta(t) -  \|  J(u(t))  - \sum_{i=1}^n\pi  \delta_{a_i(t)}\|_{\dot W^{-1,1}}
 \RV & \lesssim   t_\e,
\end{align} 
and 
\begin{equation}
\eta(t) \le  2\D_\e .
\label{eta.bound}\end{equation}
\end{lemma}
\begin{proof}
First note that in view of the definition of $\tau_1$, and
\begin{equation}
\begin{split}
\pi \sum_j |\xi_j(t) - a_j(t)| 
& =  \| \sum_{i=1}^n\pi ( \delta_{\xi_i(t)} - \delta_{a_i(t)}) \|_{\dot W^{-1,1}} \\
& \lesssim  \logep^{-{1 \over 4}} + t_\e \ \le \  {\rho_\star \over 4}
\end{split}
\label{distances.e0}\end{equation}
when $\e$ is sufficiently small,
for all $t\in [0, \tau_1]$.   From the definition of $\Phi_j$ it follows that
$\xi_j(t) - a_j(t)
=
\Phi_j(\xi_j(t), t) $
for all such $t$.   Therefore, there exists a unit
vector $v_j(t)$ such that 
$|\xi_j(t) - a_j(t)|
=
v_j\cdot \LC 2   \Phi_j(\xi_j(t)) \RC$; hence, 
\begin{align*}
&\pi \sum_j |\xi_j(t) - a_j(t)|   \nonumber \\
&=
\int \left(\pi\sum  \delta_{\xi_i(t)}\right)\left(\sum v_j\cdot \Phi_j(t)\right )\,   \nonumber \\
&\le \eta(t) +  
\int \left(  \pi\sum \delta_{\xi_i(t)}- {e_\e(u(t))\over \logep} \right)\left(\sum v_j\cdot \Phi_j(t)\right )\,   \nonumber \\
&\le \eta(t)   + \LN {e_\e(u(t))\over \logep} -\pi \sum \delta_{\xi_i(t)} \RN_{\dot W^{-1,1}}
\ \|\sum_j v_j\cdot\Phi_j(t)\|_{W^{1,\infty}}  \nonumber \\
&\le \eta(t) + 
C t_\e
\mbox{ for all }t\in [0, \tau_1]. 
\end{align*}
A similar argument shows that for such $t$,
\begin{align*}
\eta(t) \le \pi \sum|\xi_i(t)-a_i(t)| + C t_\e.
\end{align*}
which proves \eqref{eta.rocks1}.

{Again following the argument in \cite{JSp2} we use the triangle inequality and the $\dot W^{-1,1}$ norm to get }
\begin{align*}
& \| {e_\e(u(t)) \over \logep} - \sum_{i=1}^n  \pi  \delta_{a_i(t)}\|_{\dot W^{-1,1}}
\nonumber\\
&\hspace{1em}
\le \ 
\| {e_\e(u(t)) \over \logep} - \sum_{i=1}^n \pi  \delta_{\xi_i(t)}\|_{\dot W^{-1,1}} 
\ + \ \| \sum_{i=1}^n  \pi ( \delta_{\xi_i(t)} - \delta_{a_i(t)}) \|_{\dot W^{-1,1}}
\nonumber\\
&\hspace{1em}
\le \ 
t_\e +  \pi \sum|\xi_i(t) - a_i(t)| \le \ 
C t_\e + 2 \eta(t) \nonumber
\end{align*}
 for all $t\in [0, \tau_1]$.  In the same way one finds that
\begin{align*}
\eta(t) \le  C t_\e + 
 \| {e_\e(u(t))\over \logep} - \sum_{i=1}^n\pi  \delta_{a_i(t)}\|_{\dot W^{-1,1}}
\end{align*}
for all $t\in [0, \tau_1]$, which proves \eqref{eta.rocks2}.  A similar argument establishes \eqref{eta.rocks3}.
Finally,  by the triangle inequality, \eqref{Te.def}, and \eqref{eta.rocks3}
we arrive at  \eqref{eta.bound} since $t_\e \ll \mathcal{D}_\e$. 
\end{proof}

\subsection{Growth of the position error}
\label{subsec:doteta}
In the following, we  show that
 $|\dot \eta| \lesssim A_\e \sqrt{\eta} + {B}_\e$ for $B_\e \ll 1$, which is not in itself
 sufficient to prove $\eta\ll1$ for long times.  
 
\begin{proposition}  
For $t \in [0, \tau_1]$ 
\begin{align}
\label{dotetainitialestimates0}
  |\dot \eta(t)| & \lesssim {\A_\e \over \rho_\star} \LC   \sup_{s \in [0,t]} \eta(s) + { \logep^{-{1\over2}}} \RC  
+ 2 \int \p^2_{x_k x_\ell} \Phi_j  \LC {j(u)\over |u|} - j(u_\star)  \RC_k  \LC  j(u_\star) \RC_\ell    \\
\nonumber
& \quad  -  \int \p^2_{x_k x_k} \Phi_j \LC   { j(u) \over |u| } - j(u_\star) \RC_\ell \LC j(u_\star) \RC_\ell
+ \int {\LV \p_t u (t)\RV^2 \over \logep^2} \\
\label{dotetainitialestimates}
& \lesssim {\A_\e \over \rho_\star} \LC   \sup_{s \in [0,t]} \eta(s) + \logep^{-{1\over2}} \RC  + 
{n^{3 \over 2} \A_\e^{1\over2} \over \rho_\star} \sqrt{ \sup_{s \in [0,t]} \eta(s) + \logep^{-{1\over2}}} \\
& \nonumber \qquad 
+ \int {\LV \p_t u (t)\RV^2 \over \logep^2}
\end{align}
where
\begin{align*}
\mathcal{A}_\e & := {n^3 T \over \rho_\star^3} \lesssim \logep^{{9 \over 200}} |\log\logep|   
\end{align*}
is the constant defined in Theorem~\ref{thm:kineticenergybound}.
\label{dot.eta.est}\end{proposition}

To prove Proposition~\ref{dot.eta.est} we first compute the time derivative of $\eta(t)$.  
\begin{lemma}  \label{detatdecomposition}
Let $u$ be a solution to \eqref{tdgl}.  Then for $0 \leq t \leq \tau_1$ and $j = 1, \ldots, n$
\begin{equation}
\dot \eta_j = T_{j,1} + T_{j,2} + T_{j,3} +T_{j,4} +T_{j,5} +T_{j,6} +T_{j,7} 
\end{equation}
where
\begin{align*}
T_{j,1} & = \nabla \varphi(\xi_j  - a_j ) \cdot  \LC \nabla_j W(\xi) - \nabla_j W(a) \RC \\
T_{j,2} & = -\int \LC { e_\e (u) \over \logep} - \sum_{i=1}^n \pi \delta_{\xi_i} \RC \LC -  \nabla_j W(a) \RC \cdot 
\nabla \varphi (x - a_j)  \\
T_{j,3} & = \int \p^2_{x_k x_\ell} \Phi_j \LC \p_{x_\ell} |u| , \p_{x_k} |u| \RC - \p^2_{x_\ell x_\ell} \Phi_j \ e_\e(|u|)  \\
T_{j,4} & = \int \p^2_{x_k x_\ell} \Phi_j \LB  \LC {j(u)\over |u|} - j(u_\star)  \RC_\ell  \LC {j(u)\over |u|} - j(u_\star) \RC_k  - {\delta_{k \ell} \over 2} \LV { j(u) \over |u| } - j(u_\star) \RV^2 \RB  \\
T_{j,5} & = 2 \int \p^2_{x_k x_\ell} \Phi_j  \LC {j(u)\over |u|} - j(u_\star)  \RC_k  \LC  j(u_\star) \RC_\ell   \\
T_{j,6} & = -  \int \p^2_{x_k x_k} \Phi_j \LC   { j(u) \over |u| } - j(u_\star) \RC_\ell \LC j(u_\star) \RC_\ell  \\
T_{j,7} & = -  \int  \Phi_j  { \LV \p_t u \RV^2 \over \logep^2}  
\end{align*}
and $j(u_\star) = j(u_\star(\xi,d))$.
\end{lemma}

\begin{proof}
Differentiating $\eta_j$, we obtain
\[
{d \over dt} \eta_j = \int { e_\e(u) \over \logep} { d \over d t} \Phi_j  + \int \Phi_j {d \over dt} {e_\e(u) \over \logep} 
\]
Since ${d \over dt } \Phi_j(x,t) = {d \over dt} \varphi(x - a_j) = ( - \dot{a}_j ) \cdot \nabla \varphi (x - a_j)$, we can use
the ODE and the fact that $\Phi_j(\xi_i(t)) = 0 $ for $i \neq j$ to write
\begin{align*}
\int {e_\e(u) \over \logep} {d \over dt} \Phi_j  
& = \int {e_\e (u) \over \logep} ( - \dot{a}_j) \cdot \nabla \varphi(x - a_j)  \\
& =   {1\over \pi}\nabla_j W(a) \cdot \nabla \varphi (\xi_j - a_j) \\ 
&\quad + {1\over \pi} \int \LC { e_\e (u) \over \logep} - \sum_{i=1}^n \pi \delta_{\xi_i} \RC \LC -  \nabla_j W(a) \RC \cdot 
\nabla \varphi (x - a_j)  
\end{align*}

Next from the evolution identity for the  energy 
\eqref{eq:consenwithdivT},
 the representation $\nabla u = \LC \nabla |u| + i { j(u) \over |u| } \RC {u \over |u|}$,  and $A^2- B^2 = |A - B|^2 + 2 (A - B) \cdot B$ we find
\begin{align*}
\int \Phi_j {d \over dt} {e_\e(u) \over \logep} &
= \int  \p^2_{x_k x_\ell} \Phi_j \LB \LC \p_{x_\ell} u , \p_{x_k} u \RC - {\delta_{k \ell} \over 2} \LV \nabla u \RV^2 \RB
- \p^2_{x_\ell x_\ell} \Phi_j { \LC 1 - |u|^2 \RC^2 \over 4 \e^2 }  \\
& \quad -  \int \Phi_j { \LV \p_t u \RV^2 \over \logep^2}  \\
& =  \int \p^2_{x_k x_\ell} \Phi_j \LC \p_{x_\ell} |u| , \p_{x_k} |u| \RC - \p^2_{x_\ell x_\ell} \Phi_j e_\e(|u|) 
-  \int \Phi_j { \LV \p_t u \RV^2 \over \logep^2}\\
& \quad 
+ \int \p^2_{x_k x_\ell} \Phi_j \LB  \LC {j(u)\over |u|} - j(u_\star)  \RC_\ell  \LC {j(u)\over |u|} - j(u_\star) \RC_k  - {\delta_{k \ell} \over 2} \LV { j(u) \over |u| } - j(u_\star) \RV^2 \RB \\
& \quad 
+  \int \p^2_{x_k x_\ell} \Phi_j  \LC {j(u)\over |u|} - j(u_\star)  \RC_\ell  \LC  j(u_\star) \RC_k  \\
& \quad 
+  \int \p^2_{x_k x_\ell} \Phi_j  \LC {j(u)\over |u|} - j(u_\star)  \RC_k  \LC  j(u_\star) \RC_\ell   \\
& \quad 
-  \int \p^2_{x_\ell x_\ell} \Phi_j \LC   { j(u) \over |u| } - j(u_\star) \RC_\ell \LC j(u_\star) \RC_\ell  \\
& - \quad  {\sum_{k = 1}^n \p_{ x_k } \Phi_j (\xi_j)  \LC \nabla_{\xi_j} W(\xi)  \RC_k },
\end{align*}
where we have used Lemma~\ref{L.gradW}
to write $\nabla W$ by means of $j(u_\star)$.

\end{proof}

We estimate $\dot \eta$ by separately considering the
contributions from the different terms isolated in Lemma~\ref{detatdecomposition}, leading to the 
\begin{proof}[Proof of Proposition~\ref{dot.eta.est}]
%
%
%
%
Note from Lemma  \ref{detatdecomposition} and the definition \eqref{eta.def} of $\eta$ that 
\begin{equation*}
\dot \eta =  T_1+ \cdots + T_7, \quad
\mbox{ where }
T_k = \sum_{j=1}^n \frac{\eta_j}{|\eta_j|} \cdot T_{j,k}.
\label{doteta.split2}\end{equation*}
We estimate these terms in turn. 

First, note that {$\p_k \varphi_\ell (\xi_j - a_j) = \delta_{k \ell}$} for $0\le t\le \tau_1$,
by the definition of $\phi$ and \eqref{distances.e0}.
Thus, in view of \eqref{eta.rocks1},
\[
|T_1| \le  \sum_j |T_{j,1}| 
\le
C(\eta +t_\e) \sum_j |\nabla_jW(\xi) - \nabla_jW (a)|.
\]
And arguing as
in the proof of \eqref{chmdifferencebound1} we see that 
\begin{align*}
|\nabla_jW(\xi) - \nabla_jW (a)|
&\le
\sum_{k=1}^n |\xi_k(t) - a_k(t)|
(\sup_k \sup_{ |y-a(t)| \le |\xi(t)-a(t)| }  |\nabla_k \nabla_{j  } W(y)|) \\
&
\le
(\eta(t) + C t_\e)
C \frac {n}{\rho_\star^2},
\end{align*}
using \eqref{eta.rocks1} again, as well as bounds on $\nabla^2W$
from \eqref{Wderivatives}.
Thus
\begin{equation}
|T_1| 
\le
 C \frac {n^2}{\rho_\star^2} { (\eta(t) + C t_\e). }
\label{T1.est}
\end{equation}
Next, 
\begin{align}
\left| T_2 \right| &
=
\left| \int \LC {e_\e (u)\over \logep} - \sum_{i=1}^n \pi  \delta_{\xi_i} \RC 
\LC
\sum_j  \nabla_j W(a) \cdot \nabla ( \Phi_j \cdot  \frac{\eta_j}{|\eta_j|})
\RC  \right|
\nonumber \\
& \leq 
\LN {e_\e (u)\over \logep} - \sum_{i} \pi  \delta_{\xi_i} \RN_{\dot{W}^{-1,1}} 
\LN \nabla \sum_j  \nabla_j W(a) \cdot \nabla
(\Phi_j\cdot \frac{\eta_j}{|\eta_j|} )
\RN_{L^\infty} \nonumber.
\end{align}
Since the $\Phi_j$'s have disjoint support
\[
\LN\nabla \sum_j  \nabla_j W(a) \cdot \nabla
(\Phi_j\cdot \frac{\eta_j}{|\eta_j|} )
\RN_{L^\infty}
\le \sup_j | \nabla_j W(a)| \ \| \nabla^2 \Phi_j\|_\infty \le
C \frac n{\rho_\star^2}.
\]
We conclude from \eqref{xis.def} and the above that
\begin{equation}
|T_2| \leq Ct_\e {n \over \rho_\star^2}  .
\label{T2.est}
\end{equation}
Continuing, 
we use the fact that $\nabla^2 \Phi_j$ vanishes in $B_{\rho_\star}(a_j)$,
together with Theorem~\ref{thm:kineticenergybound},
to find that 
\begin{equation}
\begin{split}
\LV  T_3 \RV & \leq \LN \sum_{j}\frac{\eta_j}{|\eta_j|} \cdot \nabla^2 \Phi_j \RN_{L^\infty} 
\int_{\Omega_{\rho_\star}(a)} \LV \nabla |u| \RV^2  \\
&  \lesssim
{ \A_\e \over \rho_\star} 
\LB   \sup_{s \in [0,t]} \eta(s) + \logep^{-{1\over2}}  \RB  .
\label{T3.est}
\end{split} 
\end{equation}
Exactly the same considerations show that
\begin{equation}
\LV   T_4 \RV 
\lesssim
{ \A_\e \over \rho_\star} 
\LB   \sup_{s \in [0,t]} \eta(s) + \logep^{-{1\over2}}  \RB.
\label{T4.est}\end{equation}
Next,
\[
\LV  T_5 \RV \le 
\LN \sum_j \frac{\eta_j}{|\eta_j|}\cdot\nabla^2 \Phi_j\RN_{L^\infty} \ 
\LN \frac {j(u)}{|u|} - j(u_\star) \RN_{L^2(\Omega_{\rho_\star})}
\|  j(u_\star)\|_{L^2(\cup_j \mbox{\scriptsize{supp}} \nabla^2\Phi_j)}.
\]
Using \eqref{justar.Linf},
one can easily check that 
$\|  j(u_\star)\|_{L^2(\cup_j \mbox{\scriptsize{supp}} \nabla^2\Phi_j)}
\le \frac {Cn}{\rho_\star}( C n \rho_\star^2)^{1\over 2}$,
and hence we conclude that 
\begin{equation}
\begin{split}
\LV  T_5 \RV &  \lesssim 
\frac {n^{3\over 2} \A_\e^{1\over2} }{\rho_\star} 
 \sqrt{\sup_{s \in [0,t]} \eta(s)  + \logep^{-{1\over2}}}.
\label{T5.est1}
\end{split}
\end{equation}
Exactly the same argument shows that 
$\LV  T_6 \RV  \lesssim
\frac {n^{3\over 2} \A_\e^{1\over2} }{\rho_\star} 
 \sqrt{\sup_{s \in [0,t]} \eta(s)  + \logep^{-{1\over2}}}$.
Finally, since $|\Phi_j(x)| = |x - a_j| | \chi ({x - a_j \over \rho_\star})| \leq 2 \rho_\star$, then
\begin{equation} \label{T7.est}
|T_7| \lesssim \int_\Omega {| \p_t u |^2 \over \logep^2} .
\end{equation}
Combining \eqref{T1.est}-\eqref{T7.est} yields \eqref{dotetainitialestimates0} and \eqref{dotetainitialestimates}.



\end{proof}

The result of Proposition \ref{dot.eta.est}
 is not good enough to get any very strong result from Gronwall's inequality, but it still implies useful bounds that
 allow us to compare $\eta$ to its time averages.
 
 We define the  \emph{time average} of a function $h$ as 
\[
\LA h \RA_{\delta_\e} (t) = {1 \over {\delta_\e}} \int_{t-{\delta_\e}}^t h(s) 
\]
for any $t \geq \delta_\e$.

\begin{corollary}  \label{lipschitzboundoneta}
We have for all $0 \leq s \leq t \leq \tau_2$
\begin{equation}
\LV \eta(t) - \eta(s) \RV  \lesssim |t - s| \logep^{-{17 \over 200}}  |\log \logep |^{1\over 2} +
\logep^{-{97 \over 100}} |\log \logep|.
\label{doteta.bound1}\end{equation}
Furthermore, if $0 \le t-\delta_\e \le s \le t \le \tau_2$ then 
\begin{equation}
|\eta(s) - \LA \eta(t)\RA_{\delta_\e}| 
\lesssim \logep^{-{67\over 200}} |\log \logep |^{1\over 2} .
\label{avg.est1}\end{equation}
\end{corollary}
\begin{proof}
%
From Proposition~\ref{dot.eta.est} we have that 
\begin{align*}
\LV \dot{\eta}(t) \RV & \lesssim 
{ {n^{3\over 2} \A_\e^{1\over2}  \over \rho_\star} \sqrt{  \D_\e + \logep^{-{1\over2}}} }  + \int {|\p_t u |^2 \over \logep^2} \\
& \lesssim  \logep^{-{17 \over 200}}  |\log \logep |^{1\over 2} + \int {|\p_t u |^2 \over \logep^2}.
\end{align*}
Therefore, for any $0 \leq s \leq t \leq \tau_2$ we have
\begin{align*}
\LV \eta(t) - \eta(s) \RV & \lesssim |t - s| \logep^{-{17 \over 200}}  |\log \logep |^{1\over 2} +
\int_s^t \int { | \p_t u |^2 \over \logep^2} ,
\end{align*}
and by \eqref{kineticenergyboundfiniteinterval} estimate
 \eqref{doteta.bound1} follows. Bound \eqref{avg.est1} follows from a similar argument.
\end{proof}

\subsection{Improved supercurrent bounds by time averaging}
\label{subsec:timeaverage}
In this subsection we prove estimates of $T_1-T_7$
after averaging in time.  As in \cite{JSp2},
a simple bound using Cauchy-Schwarz and the 
  Gamma convergence estimates 
only results in bounds on $T_5$ and $T_6$ that 
involve $\sqrt{|\eta|}$. To remedy this problem, 
we follow the idea of \cite{JSp2} and {directly establish bounds} on $j(u)- j(u_\star)$ via
  Hodge decomposition 
 and time-averaging.  
%
Our result is
\begin{proposition} \label{averagedGronwallProp}
Suppose $\tau_2 > \delta_\e$ then for all $t\in [\delta_\e, \tau_2]$ and $j \in \{1,\ldots,7\}$
\begin{equation}\label{avg.finalest}
|  \LA T_j \RA_{\delta_\e}(t)| \lesssim {n^3 T \over \rho^4_\star} \sup_{s \in [\delta_\e,t]} \LA \eta \RA_{\delta_\e} (s)  + \logep^{-{3\over 10}}.
\end{equation}
\end{proposition}

\begin{proof}

We first consider $T_1$--$T_4 $ and $ T_7 $ since we can directly use
\eqref{T1.est}--\eqref{T4.est}, \eqref{T7.est} of Proposition~\ref{dot.eta.est}.  

1.  Since $\eta$ is continuous, we have for some $c \in [\delta_\e,t]$,
\begin{align*}
{\A_\e\over \rho_\star}  \sup_{s\in[\delta_\e, t]}  \eta(s) 
& =  {n^3 T\over \rho^4_\star}  \eta(c)  \\
& \lesssim {n^3 T \over \rho^4_\star}  \LB \LA \eta \RA_{\delta_\e} (c) + C \logep^{-{67\over 200}} |\log\logep|^{1\over 2} \RB \\
& \lesssim {n^3 T \over \rho^4_\star}  \LB \sup_{s\in[\delta_\e, t]} \LA \eta \RA_{\delta_\e} (s) + C \logep^{-{67\over 200}} |\log\logep|^{1\over 2}  \RB
\end{align*}
and since ${n^3 T \over \rho^4_\star} \logep^{-{67\over 200}} |\log\logep|^{1\over 2} \leq \logep^{-{3 \over 10}}$, the result follows.  

Next, we can estimate $\LA T_7 \RA_{\delta_\e}$ by \eqref{kineticenergyboundunboundedvortices} since 
\begin{align*}
\LA \int {|\p_t u |^2 \over \logep^2} \RA_{\delta_\e} 
& {1 \over \delta_\e \logep} \int_0^t \int { \LV \p_t u \RV^2 \over \logep} \lesssim { n^2 \over \rho_\star^2 \delta_\e \logep}\\
& \lesssim \logep^{ - {3 \over10}}.
\end{align*}

2.  Now we turn to the challenging $T_5$ and $T_6$ terms.   For simplicity we write
\[
T_5 = \int_\Omega \zeta_k \LC { j(u) \over |u|} - j(u_\star) \RC_k  
\]
where
\begin{equation}
 \zeta_k := 
\sum_j \p_{x_k x_m}(\frac{\eta_j}{|\eta_j|}\cdot  \Phi_j)  \ 
j_m(u_\star ) ,\quad\quad k=1,2 ,
\label{zeta.def}\end{equation}
and $j_m$ denotes the $m$ component of $j(u_\star)$, $m=1,2$.  Here
 $u_\star(x,t)  = u_\star(x ; \xi(t))$ as usual.

From the definitions and \eqref{justar.Linf}
one finds  $|\zeta| \le C \frac n {\rho_\star^2}$, and $| \operatorname{supp}  \zeta  | \leq C n \rho_\star^2$.
It follows that
\begin{equation}
\| \zeta\|_{L^q(\Omega)} \lesssim  { n^{1+\frac 1q}} {\rho_\star^{\frac 2q - 2}}
\label{zetaLq}\end{equation}
for $1 \leq  q \leq \infty$.

The following 
proof is quite similar to the proof found in Proposition 1 in \cite{JSp2}; however, we include it since the bounds are different, due to  a different differential identity for $\operatorname{div}j(u(t))$.  
We perform a Hodge decomposition 
\begin{equation} \label{hodge.bc} 
j(u) - j(u_\star) = \nabla f_1 + \nabla^\perp f_2
\end{equation}
with boundary conditions either
\begin{equation} \label{f.eqn.DBC}
 f_1 = 0 \hbox{ and } \p_\nu f_2 = 0 \hbox { on } \p \Omega, 
\end{equation}
or
\begin{equation} \label{f.eqn.NBC}
 \p_\nu f_1 = 0 \hbox{ and }  f_2 = 0 \hbox { on } \p \Omega, 
\end{equation}
depending whether we are dealing with Dirichlet or Neumann boundary conditions. 
And so we examine 
\begin{align*}
\Delta f_1 & = \operatorname{div} j(u) \\
-\Delta f_2 & = 2 \LB J(u) - \sum \pi  \delta_{\xi_j} \RB
\end{align*}
with \eqref{f.eqn.DBC} or \eqref{f.eqn.NBC}.

Since $\nabla f_1$ is small only after time-averaging, we write our estimate as 
\begin{align*}
\LA T_5 \RA_{\delta_\e} &= 
\LA \int \zeta \cdot  \frac{j(u)}{|u|} ({1- |u|})  \RA_{\delta_\e}   +  
\LA \int \zeta \cdot 
\curl f_2 \  \RA_{\delta_\e} + 
\LA \int \zeta \cdot 
\nabla f_1  \  \RA_{\delta_\e} \\
& =  \LA \int \zeta \cdot  \frac{j(u)}{|u|} ({1- |u|})  \RA_{\delta_\e} +  
\LA \int \zeta \cdot 
\curl f_2 \  \RA_{\delta_\e}   \ + \ 
\LA \int  \LA \zeta  \RA_{\delta_\e} \cdot 
\LA \nabla f_1  \RA_{\delta_\e} \RA_{\delta_\e} \\
& \quad + \LA \int  \LC \zeta - \LA \zeta  \RA_{\delta_\e} \RC \cdot \LC
\nabla f_1 - \LA \nabla f_1  \RA_{\delta_\e} \RC  \RA_{\delta_\e} \\
 &=
A_1 +A_2 +A_3 + A_4.
\end{align*}
The first term is estimated by the Cauchy-Schwarz inequality,
\begin{equation}
|A_1 | \le \| \zeta\|_{L^\infty} \  \left\| \frac{j(u)}{|u|} \right\|_{L^2} \| ( 1-|u|^2)\|_{L^2}
\le C\frac n{\rho_\star^2} \e E_\e(u) \lesssim  \logep^{-{1\over 3}} .
\label{A1.est}\end{equation}

3.  Next we claim that 
\begin{equation}
|A_2| \ \le C {s_\e}^{3\over 5} \left[n^{6\over 5}\ \rho_\star^{-{8\over 5}} (E_\e(u)+n\pi)^{2\over 5}\right]
\ \lesssim \logep^{-{1\over 3}}.
\label{A2.est}\end{equation}
From the Hodge decomposition and standard elliptic estimates \cite{R} we have
\[
\| \curl f_2\|_{L^p(\Omega)} \le 
\| f_2\|_{W^{1,p}(\Omega)} \ \le C {s_\e}^{\frac 2p -1} (E_\e(u) + n)^{2-\frac 2p}
\]
for $1\le p < 2$, with a constant depending on $p$.
Taking $\frac 1 q  = 1- \frac 1p $ in \eqref{zetaLq}
for $p\in [1,2)$ to be selected, we conclude
that
\[
|A_2| \le \|\zeta\|_{L^q} \|\curl f_2\|_{L^p}
\le C 
{ n^{2 -  \frac 1p}} \ 
{\rho_\star^{-\frac 2p}}\ 
{s_\e}^{\frac 2p-1}\ 
( E_\e(u) + n\pi)^{ 2-\frac 2p}.
\]
Choosing $p= \frac 54$, we arrive at \eqref{A2.est}.

4.  Next, we estimate $A_3$, and here we fundamentally use the time-averaging to control $\nabla f_1$.  
\begin{align*}
\LN \Delta \LA f_1 \RA_{\delta_\e} \RN^2_{L^2} &  = \LN \operatorname{div} \LA j(u) - j(u_\star) \RA_{\delta_\e} \RN^2_{L^2} 
 = \LN \LA \operatorname{div} j(u) \RA_{\delta_\e} \RN^2_{L^2} \\
& = \LN {1\over \delta_\e} \int_{t-\delta_\e}^t { ( i u , \p_t u ) \over \logep}  \RN^2_{L^2}  \\
& \leq \int_\Omega {1\over \delta_\e \logep} \int_{t-\delta_\e}^t {\LV \p_t u \RV^2 \over \logep}   \lesssim    {1 \over \delta_\e \logep} {n^2 \over \rho^2_\star} .
\end{align*}
By standard elliptic estimates
\begin{equation*}
\LN \LA f_1 \RA_{\delta_\e} \RN_{H^2} \lesssim \LN \Delta \LA  f_1 \RA_{\delta_\e} \RN_{L^2} \lesssim { n \over \delta_\e^{1\over 2} \rho_\star \logep^{1\over 2}} .
\end{equation*}
Combining with \eqref{nassumption}, \eqref{rhoassumption}, \eqref{deltaassumption} yields
\begin{align*}
\LV \int \LA \zeta \RA_{\delta_\e} \cdot \LA \nabla f_1 \RA_{\delta_\e}   \RV 
& \leq  \LN \LA \zeta \RA_{\delta_\e} \RN_{L^{4\over3}} \LN \LA \nabla f_1 \RA_{\delta_\e} \RN_{L^4}  \\
& \leq C  \LA \LN \zeta \RN_{L^{4\over3} } \RA_{\delta_\e} \LN \LA \nabla f_1 \RA_{\delta_\e} \RN_{H^1} \\
& \lesssim {n^{11 \over 4} \over \delta_\e^{1\over 2} \rho_\star^{3 \over 2} \logep^{1\over 2}} \lesssim \logep^{-{277 \over 800}} ;
\end{align*}
hence,
\begin{equation} \label{A3.est}
\LV A_3 \RV \lesssim \logep^{-{1\over 3}}.
\end{equation}

5.  Finally, we consider the challenging term $A_4$, and we again following the strategy of \cite{JSp2}. 
The idea is to take 
advantage  of the fact that $\delta_\e$ is small to show that $\zeta$ is close to
$\LA \zeta\RA_{\delta_\e}$, and similarly $\nabla f_1$ and $\LA \nabla f_1\RA_{\delta_\e}$.
First we have  
\begin{align}
|A_4| 
&\le 
\sup_{s\in [t-\delta_\e, t]}\ \| \zeta(s) - \LA \zeta\RA_{\delta_\e}\|_{L^4} 
\sup_{s\in [t-\delta_\e, t]}\ \| \nabla( f_1(s) - \LA f_1\RA_{\delta_\e} )\|_{L^{4\over3}}
\nonumber \\
&\le
\sup_{s, s'\in [t-\delta_\e, t]}\ \| \zeta(s) -  \zeta(s')\|_{L^4} 
\sup_{s,s'\in [t-\delta_\e, t]}\ \| \nabla( f_1(s) - f_1(s') )\|_{L^{4\over 3}}.
\label{52b.split1}
\end{align}

In estimating the quantities in \eqref{52b.split1},
we will  use  that for $s, s' \in [t-\delta_\e, t]$ with $t\in [\delta_\e, \tau_1]$,
\begin{equation}
|a_j(s) - a_j(s')| \lesssim \frac n{\rho_\star}\delta_\e,
\label{a.change}\end{equation}
which follows from  \eqref{ODE}.  Note that \eqref{a.change}
and \eqref{Wderivatives} imply that $|\dot a_j|\lesssim \frac{n}{\rho_\star}$. 
From \eqref{a.change} and \eqref{eta.rocks1}, \eqref{eta.bound} it follows that for
$s,s'$ as above,
\begin{equation}
\sum_{j=1}^n \LV \xi_j (s) - \xi_j (s') \RV \ \le
C \frac{n^2}{\rho_\star} \delta_\e + \eta(s) + \eta(s') + C t_\e \lesssim {n^2 \over \rho_\star} \delta_\e + \D_\e.
\label{xi.change}  \end{equation}

5a. We estimate  $\| \nabla( f_1(s) - f_1(s') )\|_{L^{4\over3} }$.
Assume that $s, s' \in [t-\delta_\e, t]$ for $t\in [\delta_\e, \tau_1]$.
By elliptic regularity, 
\eqref{hodge.bc}, and either \eqref{f.eqn.DBC} or \eqref{f.eqn.NBC} we find that,
\begin{align}
\| \nabla( f_1(s) - f_1(s') )\|_{L^{4\over3} } 
&\le C 
\| \Delta( f_1(s) - f_1(s') )\|_{\dot W^{-1,{4\over3}} }  \nonumber \\ 
&= \ 
\| \nabla \cdot [j(u)(s) - j(u)(s')] \|_{\dot W^{-1,{4\over3}} }  \nonumber \\
&\le 
\|  j(u)(s) - j(u)(s') )\|_{L^{4\over3} }.  
\label{fdiff.est1}\end{align}
Using the triangle inequality and \eqref{gstab.ref2}, it follows that
\begin{align*}
\|  j(u)(s) - j(u)(s') )\|_{L^{4\over3} }
& \lesssim \sqrt{ \A_\e \LC \sup_{r \in [\delta, s]} \eta (r) + \logep^{-{1\over2}}\RC}   +
\|  j(u_\star)(s) - j(u_\star)(s') )\|_{L^{4\over3} } \\
& \lesssim \logep^{- {41 \over 400} } | \log \logep |^{1\over 2}   +
\|  j(u_\star)(s) - j(u_\star)(s') )\|_{L^{4\over3} }.
\end{align*}
The last term on the right-hand side can be estimated by
combining \eqref{chmdifferencebound2} and \eqref{xi.change}, and 
we get 
\begin{align*}
\|  j(u_\star)(s) - j(u_\star)(s') )\|_{L^{4\over3} }
&\lesssim n^{1\over 2}
\left(
\delta_\e {n^2 \over \rho_\star} + \D_\e 
\right)^{1\over 2}\\
&\lesssim {n^{3\over 2} \over \rho_\star^{1\over 2}} \delta_\e^{1\over 2} + n^{1\over 2} \D^{1\over 2}_\e \lesssim \logep^{-{9 \over 80}}.
\end{align*}
The rest of the terms on the right-hand side of
\eqref{fdiff.est1}
are smaller using the bounds on $n, \rho_\star$.  Therefore,  we find that
\begin{equation}
\| \nabla( f_1(s) - f_1(s') )\|_{L^{4\over3} } 
\lesssim
 \logep^{-{41\over 400}} | \log \logep |^{1\over 2}.
\label{fdifferenceest}\end{equation}


5b. We estimate  $\| \zeta(s) - \zeta(s') \|_{L^4 }$.  
Assume that $0 \le t-\delta_\e \le s, s' \le t \le \tau_1$.
In order to find a time-Lipschitz bound on $\zeta$,
we have from the definition \eqref{zeta.def}  that
\begin{align*}
\zeta_k(s) - \zeta_k(s')
& = \sum_j  \p_{x_k x_m}(\frac{\eta_j}{|\eta_j|}\cdot  \Phi_j) (s) \ 
j_m(u_\star )(s)  \\
& \quad -  \sum_j \p_{x_k x_m}(\frac{\eta_j}{|\eta_j|}\cdot  \Phi_j) (s') \ 
j_m(u_\star )(s') \\
& = \sum_j \p_{x_k x_m} \LB \frac{\eta_j}{|\eta_j|}\cdot ( \Phi_j (s) - \Phi_j (s') )\RB  \ 
j_m(u_\star )(s) 
\\
& \quad  + \sum_j \p_{x_k x_m}(\frac{\eta_j}{|\eta_j|}\cdot  \Phi_j) (s')  \ 
\LB j_m(u_\star )(s) - j_m(u_\star )(s') \RB 
\\
& = Z_1 + Z_2.
\end{align*}
First consider $Z_1$.  From the definitions,
\begin{align*}
\LN 
\p_{x_k x_m} \LB \frac{\eta_j}{|\eta_j|}\cdot ( \Phi_j (s) - \Phi_j (s') )\RB  \ 
\RN_{L^\infty  }
& \leq  \LN \p_{x_k x_m}[\varphi (x - a_j(s) ) - \varphi ( x - a_j(s'))] \RN_{L^\infty  } \\
& \leq  C \LN \p_{x_k x_m x_n} \varphi \RN_{L^\infty} \LV a_j(s) - a_j(s') \RV 
 \lesssim {n \over \rho_\star^3} \delta_\e 
\end{align*}
using \eqref{a.change}.

As in \cite{JSp2} we claim that
\begin{equation}
\mbox{supp}\, \nabla^2 \Phi_j(s) \cup \mbox{supp}\, \nabla^2\Phi_j(s')\subset 
B_{3 \rho_\star}(\xi_j(s)) \setminus B_{\frac 12 \rho_\star}(\xi_j(s))
\label{supports.bound}\end{equation}
for all $\e$ sufficiently small.
This follows from \eqref{a.change}, \eqref{xi.change}, and \eqref{eta.rocks1}.  In particular
 the distances separating $a_i(s), a_i(s'), \xi_i(s), \xi_i(s')$ 
are significantly smaller than $ \rho_\star$.

The support condition \eqref{supports.bound} implies that 
$| j(u_\star)(\xi(s))| 
\leq 
{C n \over \rho_\star}$ on the support of $ Z_1$.  Since the support of $Z_1$ has measure bounded by 
$C n \rho_\star^2$,
we conclude that
\begin{align}  \label{t2biiestI}
\LN Z_1 \RN_{L^4} \leq  {C n^{2} \over \rho_\star^4}\left( C n \rho_\star^2\right)^{1\over 4}
\delta_\e \lesssim 
 \frac{n^{9\over 4}}{\rho_\star^{7\over 2}} \delta_\e \lesssim \logep^{-{163\over 800}}.
\end{align}
Finally we consider $Z_2$. 
Since $\LN \sum_j \p_{x_l x_m}  \Phi_j  \RN_{ L^\infty}  \lesssim {1 \over \rho_\star}$, and
using that $\mbox{supp}\, Z_2$ has measure at most $Cn\rho_\star^2$,
we use
H\"older's inequality to estimate
\[
\LN Z_2 \RN_{L^4} \le
\frac C{\rho_\star}
\LN j (u_\star )(s) - j(u_\star )(s') \RN_{L^\infty( \cup_j \mbox{\scriptsize{supp}} \nabla^2 \Phi_j(s')   )}
(C n \rho_\star^2)^{1\over 4}.
\]
It then follows that supp$\,\cup_j\nabla^2\Phi_j(s') \subset \Omega_{\rho_\star/2}(\xi(s)) \cap
\Omega_{\rho_\star/2}(\xi(s'))$. We therefore  use 
\eqref{chmdifferencebound1} to find that
\[
\LN j (u_\star )(s) - j(u_\star )(s') \RN_{L^\infty( \cup_j \mbox{\scriptsize{supp}} 
 \nabla^2 \Phi_j(s') )} 
\leq { C \over \rho^2_\star} \sum_{j=1}^n \LV \xi_j (s) - \xi_j (s') \RV.
\]
Consequently,  \eqref{xi.change}
and \eqref{avg.est1} imply that
\begin{align}  
\LN Z_2 \RN_{L^4} 
&
\leq  C({ n \over \rho_\star^2})^{5\over 4} 
( \frac{n^2}{\rho_\star} \delta_\e + \eta(s) + \eta(s') + C t_\e) \nonumber  \\
&\lesssim
 {n^{5\over 4} \over \rho^{5\over 2}_\star} 
\sup_{\delta \leq s \leq t}  \LA \eta(t)\RA_{\delta_\e} +
{n^{5\over 4} \over \rho_\star^{5\over 2}} \LB {n^2 \over \rho_\star} \delta_\e + \delta_\e {n^{3\over 2} \over \rho_\star} \sqrt{\A_\e ( \D_\e +\logep^{-{1\over2}})} + t_\e \RB 
\nonumber \\
& \lesssim {n^3 T \over \rho_\star^4} \sup_{\delta \leq s \leq t}  \LA \eta(t)\RA_{\delta_\e} + \logep^{-{159 \over 800}} | \log \logep |^{1\over 2}.
\label{t2biiestII}
\end{align}
Combining  \eqref{t2biiestI} and \eqref{t2biiestII}  yields
\begin{align}
\LN \zeta(s) - \zeta(s') \RN_{L^4 } 
\leq  
C \frac{n^3 T}{\rho_\star^4}
\sup \LA \eta(t)\RA + {C\logep^{-{159 \over 800}}| \log \logep |^{1\over 2}.}
\label{zetadifferenceest}\end{align}

6. Finally we combine the above with 
\eqref{52b.split1}, \eqref{fdifferenceest}, and \eqref{zetadifferenceest}
to deduce that
\begin{equation} \label{A4.est}
|A_4| \lesssim {n^3 T \over \rho_\star^4} \sup \LA \eta(t)\RA +  \logep^{-{3\over 10}}.
\end{equation}
Combining \eqref{A1.est}, \eqref{A2.est}, \eqref{A3.est}, and \eqref{A4.est} yields the bounds on $\LA T_5 \RA_{\delta_\e}$ and $\LA T_6 \RA_{\delta_\e}$, and this finishes the proof of Proposition~\ref{averagedGronwallProp}.
%
%
%
%
%
%
%
\end{proof}
\subsection{Continuity arguments}
\label{subsec:continuity}
We now complete the proof of our Theorem~\ref{thm:quantdynamics}. 
\begin{proof}[Proof of Theorem~\ref{thm:quantdynamics}]

Recall $\tau_{max} = \min \{ \tau_0, C  \sqrt{ |\log \logep| {\rho_\star^4 \over n^3}} \}$ denotes the claimed longest possible time interval for which we can pin the vortices to the $a_j(t)$'s.  The main point of the proof will be to show that
 all relevant estimates hold up to time
 $\tau_{max}$ by a combination of continuity arguments for the Jacobian and a Gronwall estimate on $\LA \eta \RA_{\delta_\e}$.
  If $\tau_2 = \tau_1 = \tau_{max}$ then Theorem~\ref{thm:quantdynamics} follows
directly. We assume this statement does not hold and the following is 
a proof by contradiction in several parts.

1.  We first claim for any $T > 0$ that the solution operator to \eqref{tdgl} is continuous  from 
$[0,T] \to \dot{H^1}$, in particular
\begin{equation}  \label{modcontinuityest}
\| \nabla u(t) - \nabla u(s) \|_{L^2} \leq C_\e \ o(|t-s|)  
\end{equation}
for $0\leq s \leq t \leq T$, where $C_\e$ depends on $\e$ and $T$ but is independent of $t,s$.  
It is standard theory (see for instance \cite{Evans} Section 5.9, Theorem 4) that if 
$u\in L^2(0,T;H^2(\Omega;\C))$ and 
$\partial_t u\in L^2(0,T;L^2(\Omega;\C))$ then
$u\in C^0([0,T];H^1(\Omega;\C))$, which implies \eqref{modcontinuityest}. 
These conditions are true for solutions of \eqref{tdgl} since by the gradient flow 
property,
\[
\int_0^T \int_\Omega |\p_t u|^2 \le CE_\e(u(0)) \leq C_\e
\]
and from \eqref{conservationmass} {and the initial conditions one finds $\| u (t)\|_{L^\infty} \leq 1$} due to the maximum principle; hence, 
\begin{align*}
\int_0^T \int_\Omega |\Delta u|^2 
& \le C\int_0^T \int_\Omega |\p_t u|^2  + 
C\int_0^T \int_\Omega {1\over \e^2} |u|^2 {(1 - |u|^2)^2\over \e^2}  \\
& \le C E_\e(u(0)) + C_\e T\|u(t)\|^2_{L^\infty } E_\e(u_0)
 \le C_\e,
\end{align*}
where $C_\e$ depends on $T$ and $\e$.  

2.  Since $J(u) =  \det \nabla u = -\nabla^\perp u_1 \cdot \nabla u_2$, where $u = u_1 + i u_2$,   then for any $0 \leq s \leq t \leq T$
\begin{align}
\LN J(u)(t) - J(u)(s) \RN_{\dot{W}^{-1,1} }  
& = \sup_{ \| \phi \|_{W^{1,\infty}_0 } \leq 1} \LV \int \phi \LC  J(u)(t) - J(u)(s) \RC \RV \nonumber  \\
& \leq C \LN \nabla u(t) - \nabla u(s) \RN_{L^2} \LN \nabla u(t) + \nabla u (s) \RN_{L^2} \nonumber \\
& \leq C_\e \ o(|t-s|)  \label{modcontinuityestJac}
\end{align}
 from \eqref{modcontinuityest}, where $C_\e$ depends on $\e$ and $T$.

3.  We claim that $0 \leq \tau_2 < \tau_1\leq \tau_{max}$.  Suppose this claim fails, then by the definitions of $\tau_1, \tau_2$ and our assumption,  $0\leq \tau_2 = \tau_1 < \tau_{max}$.  
By maximality of $\tau$, we have
\begin{equation} \label{badDestimate}
D(a(\tau_1)) \leq 1
\end{equation}
and
\begin{equation} \label{badJacobianestimate}
\LN J(u)(\tau_1)  - \sum_{j=1}^n  \delta_{a_j(\tau_1)} \RN_{\dot{W}^{-1,1} }
\leq \D_\e .
\end{equation} 

Consider first \eqref{badDestimate}.  Since $\tau_2 = \tau_1$ then $\eta(\tau_1) \leq  {1\over2} \D_\e ={1 \over 2} \logep^{-{1\over 4}}$, then \eqref{excessenergyestimateviaetacontrol} implies $D(a(\tau_1)) \leq 
\mathcal{A}_\e \LB \sup_{s\in[0,\tau_2]} \eta(s) + \logep^{-{1\over2}} \RB \leq {1\over2}$.  We now claim that there exists a $\mu_0$ such that for all
$\tau_1 \leq \tilde{t} \leq \tau_1 + \mu_0$ then $D(a(\tilde{t})) \leq 1$.  In particular, by \eqref{ODE} and \eqref{Wderivatives} 
\begin{align*}
D(a(\tilde{t})) &
= D(a(\tau_1)) + \int_{\tau_1}^{\tilde{t}} \LV \dot{a} \RV^2  - \int_{\tau_1}^{\tilde{t}} \int_\Omega {\LV \p_t u \RV^2 \over \pi \logep}   
 \leq {1\over2} +  C \mu_0 {n^3 \over \rho_\star^2} \leq 1
\end{align*}
for $\mu_0$ small enough.  

Next consider \eqref{badJacobianestimate}.   
Again $\eta(\tau_1) \leq  {1\over 2} \D_\e$ and so by Lemma~\ref{lemmameasuringmu} we have 
\[
\| J(u)(\tau_1) - \pi\sum  \delta_{a_j(\tau_1)} \|_{\dot{W}^{-1,1}} 
\leq {1\over 2} \logep^{-{1\over4}} + C t_\e \leq {5\over 8} \logep^{-{1\over 4}} 
\]
for $\e$ small enough.
By \eqref{modcontinuityestJac} there exists a $\mu_1 >0$ such that for all $\tau_1 \leq \tilde{t} \leq \tau_1 + \mu_1$
\[
\| J(u)(\tilde{t}) - J(u)(\tau_1) \|_{\dot{W}^{-1,1}} \leq C(E_\e(u_0)) \ o(\mu_1) \leq {1\over 8} \logep^{-{1\over 4}}
\]
for $\mu_1$ small enough.  Furthermore,  there exists $\mu_2$ such that for $\tau_1 \leq \tilde{t} \leq \tau_1 + \mu_2$
\begin{align*}
\| \pi \sum \delta_{a_j(\tilde{t})} -  \pi \sum \delta_{a_j(\tau_1)} \|_{\dot{W}^{-1,1}} & \leq C \sum \LV a_j(\tilde{t} ) - a_j(\tau_1) \RV 
\leq C \mu_2 \sum \LV \nabla_{a_j} W \RV \\
& \leq C {n^2 \over \rho_\star} \mu_2 \leq  {1 \over 8} \logep^{-{1\over 4}}
\end{align*}
for $\mu_2$ small enough, where we used \eqref{Wderivatives} in the third inequality.  Therefore, for $\mu = \min \{ \mu_0, \mu_1 , \mu_2 \} > 0$ and all $\tau_1 \leq \tilde{t} \leq \tau_1 + \mu$ we have 
\begin{align*}
 \LN J(u)(\tilde{t}) - \sum \pi \delta_{a_j(\tilde{t})} \RN_{\dot{W}^{-1,1}} & \leq \| J(u)(\tilde{t}) - J(u)(\tau_1) \|_{\dot{W}^{-1,1}}  \\
 & \quad + \| \pi \sum \delta_{a_j(\tilde{t})} -  \pi \sum \delta_{a_j(\tau_1)} \|_{\dot{W}^{-1,1}} \\
 & \quad + \| J(u)(\tau_1) - \sum  \delta_{a_j(\tau_1)} \|_{\dot{W}^{-1,1}} \\
& \leq {7\over 8} \logep^{-{1\over 4}} \leq \D_\e,
\end{align*}
and $D(a(\tilde{t})) \leq 1$.  
As $\mu>0$, this contradicts the maximality of $\tau_1$.

4.  We claim that if $\tau_2 \leq  \logep^{-{1\over 4}}$ then
$\tau_2$ cannot be maximal.
First,
using \eqref{doteta.bound1} and $\eta(0) \leq {1\over 8} \mathcal{D}_\e$, we have for all $0 \leq t \leq  \tau_2$
\begin{equation*} \label{etalineargrowthestimates}
\begin{split}
\eta(t)  & \leq {1\over8 }\mathcal{D}_\e +  t C \logep^{-{7\over80}} |\log \logep|^{1\over2} \\
& \leq  {1\over8 }\mathcal{D}_\e +  \tau_2 C \logep^{-{7\over80}} |\log \logep|^{1\over2} 
 \leq  {1\over8 }\mathcal{D}_\e + {1\over 8} \mathcal{D}_\e, 
\end{split}
\end{equation*}
and so $\eta(\tau_2) \leq {1\over 4} \mathcal{D}_\e$.  
Next, we use \eqref{dotetainitialestimates}, Young's Inequality, and Lemma~\ref{gronwalllemma} below with
$x(t) = \eta(t)$, $A = \logep^{3 \over 10}  $, $B = \logep^{-{2\over5}}$, and $G(t) = \int_\Omega {|\p_t u |^2 \over \logep^2}$ to get for all $\tau_2\leq t \leq \tau_2 + \tilde \mu$
\begin{equation} \label{badgronwall}
\eta(t) \leq \exp\LC \tilde \mu \logep^{3 \over 10}  \RC \LB \eta(\tau_2) + \logep^{-{2\over 5}} \RB 
\leq {1\over 4} \mathcal{D}_\e
\end{equation} 
for $0<\tilde\mu < \tau_1 - \tau_2$ small enough.  
This contradicts the maximality of $\tau_2$.

{5.  Using  Step 4, we have
 $\tau_1 > \tau_2\geq \delta_\e$.} 
 
6.  We now show that the assumption $\tau_2<\tau_{max}$ 
leads to a contradiction.
By Step 4 and Step 5  we see that $\tau_2 > \delta_\e$ with $\sup_{0\leq s \leq \delta_\e} \eta(s) \leq   {1\over 4} \mathcal{D}_\e$; therefore, $\LA \eta \RA_{\delta_\e}(\delta_\e)\leq {1\over 4} \mathcal{D}_\e$. 
From \eqref{avg.finalest} in Proposition~\ref{averagedGronwallProp}
we have the differential inequality for the averaged $\LA \eta \RA_{\delta_\e}$,
\[
{d \over d t} \LA \eta \RA_{\delta_\e} \leq 
|  \LA T_j \RA_{\delta_\e}(t)| \lesssim {n^3 \over \rho^4_\star} \sup_{s \in [\delta_\e,t]} \LA \eta \RA_{\delta_\e} (s)  + \logep^{-{3\over 10}}.
\]
for all $\delta_\e \leq t \leq \tau_2$.  
Using the Gronwall argument from Lemma~\ref{gronwalllemma} below with $x(t) = \LA \eta \RA_{\delta_\e}$, $A = {n^3 \over \rho^4_\star}$, $B = \logep^{-{3\over 10}}$, and $G(t) = 0$, we find 
\[
\LA \eta (t) \RA_{\delta_\e} \leq \LC \LA \eta (\delta_\e) \RA_{\delta_\e} + \logep^{-{3 \over 10}} {\rho_\star^4 \over n^3 \tau_{max}}  \RC  \exp \LB { n^3 \tau_{max} (t - \delta_\e) \over \rho_\star^4}   \RB \leq {3\over8} \D_\e
\]
 for all $\delta_\e \leq t \leq \tau_2$. In particular, $\eta(\tau_2) \leq \LA \eta \RA_{\delta_\e} (\tau_2) + |\eta (\tau_2) - \LA \eta \RA_{\delta_\e} (\tau_2) | \leq {4 \over 9} \mathcal{D}_\e$. 
Repeating the argument in Step 4 and using \eqref{badgronwall},
we see that the estimate necessary for $\tau_2$ also holds at
$ \tau_2 + \widehat{\mu}$ for some $\widehat{\mu} < \tau_1 - \tau_2$,
contradicting the maximality of $\tau_2$.

7. From Step 3 and Step 6 we see that $\tau_2=\tau_1=\tau_{max}$, which proves 
\eqref{energydensityconvergencethm1} and \eqref{Jacobianconvergencethm1}.

8. Finally, we prove \eqref{honeconvergence} and \eqref{lpconvergence}.  Since $\Omega_{\rho_\star}(a(t)) \subset \Omega_{\sigma_\star}(\xi(t))$ then $\int_{\Omega_{\rho_\star}(a(t))}
e_\e(|u(t)|) + {1\over 4} | {j(u(t)) \over |u(t)|} - j(u_\star(a(t))|^2 
\leq \int_{\Omega_{\rho_\star}(a(t))}
e_\e(|u(t)|) + {1\over 2} | {j(u(t)) \over |u(t)|} - j(u_\star(\xi(t))|^2 +\int_{\Omega_{\rho_\star}(a(t))}
 {1\over 2} |  j(u_\star(\xi(t)) - j(u_\star(a(t))|^2$.  From \eqref{chmdifferencebound1} we find 
 $\int_{\Omega_{\rho_\star}(a(t))} |  j(u_\star(\xi(t)) - j(u_\star(a(t))|^2 \lesssim \logep^{-{1\over 5}}$.  Bound \eqref{lpconvergence} follows from a similar estimate, using  \eqref{chmdifferencebound2} instead.  
\end{proof}

We conclude with the following Gronwall estimate used at the end of the proof of Theorem~\ref{thm:quantdynamics}.
\begin{lemma} \label{gronwalllemma}
Suppose $A, B$ are positive constants and $G(t)\geq 0$ is integrable, and suppose
\[
{d \over d t} x(t) \leq  A \, \sup_{s\in [0,t]} x(s) +  B +  G(t)
\]
then for any $0 \leq r \leq t$
\begin{equation*}
x(t) \leq e^{A( t - r)} \LB  x(r)  + { B  \over A} + \int_r^t G(s) \RB   .
\end{equation*}
\end{lemma}

\begin{proof}
Let $m(t) = \sup_{s \in [0,t]} x(s)$ then $\dot{m}(t) \le \max \{ \dot{x}(t), 0\}$ since the maximum can increase only if $x$ increases.  On the one hand, if $\dot{m}(t) \le \dot{x}(t)$ then $\dot{m}(t) \le \dot{x}(t) \leq A m(t) + B + G(t)$.  On the other hand, if $\dot{m} (t) \le 0$ then 
$\dot{m}(t) \leq A m(t) + B + G(t)$.  The estimate follows.

\end{proof}
\section{Hydrodynamic limit}
\label{sec:hydrodynamic}
In this section we will prove Theorem \ref{thm:hydrodynamic}
in two steps. First, we show that under good assumptions on the
initial data, the ODE vortex cloud converges to a solution
of the mean field equation. 

Then we show that these assumptions on the initial data and
 those of Theorem~\ref{thm:quantdynamics}
 can be simultaneously fulfilled for a suitably
 chosen sequence $n\to\infty$, and then we can 
 relate the rescaled energy densities
 ${e_{\e_n}(u_{\e_n}(t)) \over n \pi |\log \e_n| }$
 to the mean field equation.

\begin{proposition}[Convergence of ODE to mean field PDE]
 \label{propODELimit}
Consider a sequence of initial data $\{ a_j(0) \}_{j=1}^n$, and assume $$- {1\over n^2} \sum_{j\neq k} N_n(a_j(0), a_k(0)) \lesssim 1$$ for every $n$.   
Let $a_j(t)$ solve \eqref{ODE},
 with $W(a)$ in the Dirichlet case.  
Setting $\omega_n(t) = {1\over n} \sum_{j=1}^n\delta_{a_j(t)}$ then 
in the rescaled time $\overline{t} = n t$, we find 
$\omega_n(\overline{t}) \to \omega(\overline{t})$ in $ \mathcal{M}$ for all $\overline{t}$ and $\omega$ is a generalized interior weak solution 
(as defined in \eqref{generalizedweaksolution})
to
\[
\p_{\overline{t}} \omega + \operatorname{div}\LC 4\pi\nabla \LC \Delta_{\mathcal{N}}^{-1} \omega \RC \omega \RC = 0
\]
with $\omega_0 = \lim \omega_n(0)$.   Finally, we have 
$v \in  L^2_{loc}(\Omega)$ where $v = 4\pi\nabla \Delta^{-1}_\mathcal{N}(\omega)$.  
\end{proposition}

We first show that the vortex density function $\omega_n(\overline{
t})$ satisfies an equation very close to \eqref{generalizedweaksolution}.  
Recall from \cite[Theorem VIII.3]{BBH}
that
\begin{align*}
- \nabla_{a_j} W(a)  = 2\pi\nabla S_n^j (a_j) 
\end{align*}
where
\(
S_n^j(x)  = \sum_{k=1}^n N_n(x,a_k) - \log |x - a_j| 
\)
so
\begin{align*}
\nabla S_n^j(x) & =  \sum_{k \neq j}^n \nabla N_n(x,a_k) + \LB \nabla N_n(x,a_j) - {x - a_j \over |x - a_j|^2} \RB  \\
& = \sum_{k \neq j}^n \nabla N_n(x,a_k) +  \nabla H_n( x, a_j) .
\end{align*}
For any test function $\chi \in C_0^\infty (\Omega)$ and with $\omega_n = {1\over n } \sum_{j=1}^n \delta_{a_j(t)}$ we have  
\begin{align*}
{1\over n} \p_t  \int \chi \omega_n(t) & = {1\over n^2} \p_t  \sum_{j=1}^n \chi (a_j)  = {1\over n^2} \sum_{j=1}^n \p_\ell \chi (a_j) \LB \dot{a}_j \RB_\ell \\
& = -  \frac{1}{\pi n^2} \sum_{j=1}^n \p_\ell \chi (a_j) \LB \nabla_{a_j} W(a) \RB_\ell  =  \frac{2}{n^2} \sum_{j=1}^n \p_\ell \chi (a_j) { \p_\ell S_n^j} (a_j) \\ 
& =  {2\over n^2} \sum_{j=1}^n  \int \p_\ell \chi (x)  \p_\ell S_n^j (x) \delta_{a_j}(x)   .
\end{align*}
Using $2\pi\delta_{a_k} = \Delta N_n(\cdot,a_k)$  and
 the above identity for $\nabla S_n^j$  yields 
\begin{align*}
 {1\over n} \p_t  \int \chi  \omega_n(t) & =   
{1\over \pi} {1\over n^2} \sum_{j=1}^n  \int \p_\ell \chi (x)
\LB \sum_{k \neq j} \p_\ell N_n(x,a_k) + \p_\ell H_n(x,a_j) \RB \p_m \p_m N_n(x,a_j) dx 
 \\
  & =  {1\over  \pi}  {\int \int \int_{y \neq z}}  \p_\ell  \chi(x) \p_\ell N_n(x,y) \p_m \p_m N_n(x,z) \omega_n(y) \omega_n(z) dy dz dx
 \\
 & \quad + {1\over \pi} {1\over n} \int \int \p_\ell \chi(x) \p_\ell H_n(x,y) \p_m \p_m N_n(x,y) \omega_n (y) dy dx \\
& = A_n + B_n .\\ 
\end{align*}

Following \cite{LiuXin} we define the matrix-valued function $\mathcal{K}(n, y,z; \eta)$ 
\begin{equation} \label{matrixkernal}
\mathcal{K}_{jk}(n, y, z, \eta) = \int_\Omega \eta(x) \p_{x_j} N_n(x,y) \p_{x_k} N_n(x,z) dx,
\end{equation}
and after a short calculation
using symmetry and \eqref{eq:ibpdiv},
 one can rewrite $A_n$ and $B_n$ as 
\begin{equation*}
\label{FormAnBn}
\begin{split}
A_n & = - \frac1\pi \int \int  \int_{y\neq z} \LC \mathcal{K}_{11} - \mathcal{K}_{22} \RC\LC n, y, z, \LC \p_{x_1}^2  - \p_{x_2}^2  \RC\chi\RC  \omega_n(y) \omega_n(z) dy dz dx \\
& \quad - \frac4\pi\int \int \int_{y\neq z}  \mathcal{K}_{12} \LC n, y, z,  \p_{x_1} \p_{x_2}  \chi\RC  \omega_n(y) \omega_n(z) dy dz dx 
=:A_n^1 + A_n^2\\
B_n & =  {1\over n \pi} \int \int \p_m \p_m \LC \p_\ell \chi  \ \p_\ell H_n(x,y) \RC N_n (x,y) \omega_n(y) dy dx.
\end{split}
\end{equation*}
We will show that as $n \to \infty$,  $B_n$ converges to zero and $A^j_n$'s converge to the form of the generalized weak solution.  However, in order to complete the proof, we prove two technical lemmas on the  $\mathcal{K}_{jk}$ and the vorticity maximal function (defined below).  

%
%
%
\begin{lemma} \label{EvansMullerEstimates}
The matrix functions $\mathcal{K}_{jk}(n, y, z, \eta)$ defined in \eqref{matrixkernal} satisfy the following estimates 
for $y, z \in \Omega$ and $\eta \in C^\infty_0(\Omega)$:
\begin{align}
\label{N1N1N2N2}
&\LV \LC \mathcal{K}_{11} -  \mathcal{K}_{22} \RC\LC n, y, z, \eta \RC \RV  \leq C \\
\label{N1N2}
&\LV \mathcal{K}_{12} \LC n, y, z, \eta \RC   \RV  \leq C \\
\label{N1N1N2N2Log}
&\LV \mathcal{K}_{11} \LC n, y, z, \eta \RC \RV  + \LV \mathcal{K}_{22} \LC n, y, z, \eta \RC \RV \leq 2 \log |y - z| + C
\end{align}
where $C$ depends only on $\eta$, $\varphi_\star$, and $\Omega$.   Finally, we have the bound
\begin{equation} \label{Hboundcompact}
\LV \nabla^k_x H_n(x, y) \RV \leq {C \over \dist ( y, \p \Omega )^k }
\end{equation}
where $C$ depends on $k$, $\varphi_\star$, and $\Omega$.
\end{lemma}
\begin{proof}
These estimates are similar to ones found in Delort \cite{Delort} and Evans-M\"uller \cite{EvansMuller} for the associated Green's function on $\R^2$;
therefore, we only sketch the proof of \eqref{N1N2} following the argument of \cite{EvansMuller}.  
The proofs of \eqref{N1N1N2N2} and \eqref{N1N1N2N2Log} can be established by similar adjustments of arguments in \cite{EvansMuller}.

To prove \eqref{N1N2} one needs to examine  the behavior of the gradient of $H_n(x,p) = N_n(x,p) - \log |x-p|$ defined via \eqref{formTrivialNeumannFunction} and \eqref{formNeumannFunction}.   Since the test function $\eta$ has compact support away from the boundary, it follows that $\p_{x_j} H_n(x,\cdot)$  is bounded for all $x$ on the support of $\eta$ (as are higher derivatives of $H_n(x,\cdot)$), as in the proof of Lemma~\ref{SandierSoretEstimates}.   
We can now write 
\begin{align*}
\LV \mathcal{K}_{12} \RV 
&  = \LV \int \eta(x) \LB { (x - y)_1 \over |x - y|^2 } + \p_{x_1} H_n(x,y) \RB \LB { (x - z)_2 \over |x - z|^2 } + \p_{x_2} H_n(x,z) \RB dx \RV \\
& \leq  \LV \int \eta(x) \LB { (x - y)_1 \over |x - y|^2 } { (x - z)_2 \over |x - z|^2 } \RB \RV 
+ \LV \int \eta(x) \LB { (x - y)_1 \over |x - y|^2 } \p_{x_2} H_n(x,z) \RB \RV\\
& \quad + \LV \int \eta(x) \LB \p_{x_1} H_n(x,z) { (x - y)_2 \over |x - y|^2 } \RB \RV
+ \LV \int \eta(x) \LB \p_{x_1} H_n(x,z) \p_{x_2} H_n(x,z) \RB \RV \\
& = I_1 + I_2 + I_3 + I_4
\end{align*}
Using the support of $\eta$ and the explicit estimates in the proof of Theorem~1.1 of \cite{EvansMuller}, it follows that $I_1 \leq C$.  $I_4 \leq C$ due to the uniform bounds on $\nabla_x H_n(x,\cdot)$ for $x$ having compact support away from the boundary, where $C$ depends on the distance of the support to the boundary.  Finally, we consider the the bound on $I_2$ and $I_3$, which can be handled by identical bounds. 
Due to the uniform bound on $\nabla_x H_n(x,\cdot)$ away from the boundary, we have
\begin{align*}
I_2 & = \LV \int \eta(x) \LB { (x - y)_1 \over |x - y|^2 } \p_{x_2} H_n(x,z) \RB \RV  \lesssim \int_{\operatorname{supp}(\eta)} {1\over |x - y| } dx  \\
& \lesssim \int_0^{\operatorname{diam}(\Omega)} dr \lesssim 1.
\end{align*}
Combining the estimates  yields \eqref{N1N2}. 
\end{proof}

Define for any Radon measure $\mu$ the maximal vorticity function $M_r(\mu)$ of DiPerna-Majda
\cite{DiPernaMajda}  
\[
M_r(\mu) =\sup _{x \in \Omega, 0 < t \leq T} \int_{B_r(x) \cap \Omega} | \mu (y,t)| d y
\]
for $0  < r \leq {1\over2}$.  As in \cite{LiuXinCPAM, LiuXin, LinZhang} we prove a decay estimate on $M_r(\omega_n)$ below in order to pass to the limit in the main term $A_n$.  

\begin{lemma} \label{maximalvorticityestimate} Suppose $\{a_j(t)\}_{j=1}^n$ arise from the hypotheses of Proposition~\ref{propODELimit}, then we can  bound
$$
M_r(\omega_n(t))  \lesssim { 1\over \sqrt{|\log r|}} + {1 \over \sqrt{n}} 
$$
for all $n$ and all $r \leq 1$.  Furthermore,
\[
M_r(\omega) \lesssim{1 \over \sqrt{| \log r |}}.
\]
\end{lemma}
\begin{proof}
Following the structure of the argument in \cite{LiuXin} we 
have for some positive integer $k_x \leq n$,
\begin{align*}
|\log r| M^2_r(\omega_n(t)) & = | \log r| \LB {1\over n} \# \{ a_j (t) \in B_r(x) \cap \Omega \} \RB^2  \\
& =  | \log r| {k_x(k_x  - 1) \over n^2 } + | \log r| {k_x \over n^2} \\
& \lesssim \LB  {1\over n^2} \sum_{|a_j - a_k| \leq r} \LB -N_n(a_j,a_k) + C \RB \RB + { |\log r| \over n} M_r(\omega_n)  \\
& \lesssim 1 + { | \log r | \over n } M_r(\omega_n)
\end{align*}
where we used Lemma~\ref{SandierSoretEstimates}
and $\sum_{j \neq k} 1 \leq n^2$.
 Since $M_r(\omega_n) \leq 1$  the bound follows.

For the bound on $M_r(\omega)$ we have for $\chi \in C^\infty$ where $\chi = 1$ on $B_r(x)$ and $\chi = 0$ on $\R^2 \backslash B_{2r}(x)$, $x$ is chosen where $\int_{B_r(x) \cap \Omega} \omega(t)$  is maximal, then
\[
M_r(\omega) \leq \int \chi \omega  = \lim_{n \to \infty} \int \chi \omega_n \leq \lim_{n\to \infty} M_{2r}(\omega_n) \lesssim {1\over \sqrt{ |\log r|}}.
\]
\end{proof}

\begin{proof}[Proof of Proposition~\ref{propODELimit}]

We now examine the convergence behavior of $A^j_n$ and $B_n$.  From Lemma~\ref{EvansMullerEstimates} one can follow the arguments of \cite{Schochet, LiuXin} to establish the convergence of $A^j_n$.  Looking at $A_n^1$ and taking $\chi \in C_0^\infty(\Omega)$ and setting $\eta = \LC \p_{x_1}^2  - \p_{x_2}^2  \RC\chi$, we have
\begin{align*}
A^1_n & = - \frac1\pi\int \int  \int_{\{|y -  z| \geq r\} \cap \Omega} \LC \mathcal{K}_{11} - \mathcal{K}_{22} \RC\LC n, y, z, \eta \RC  \omega_n(y) \omega_n(z) dy dz dx  \\
& \quad - \frac1\pi\int \int \int_{\{0< |y -  z| < r\} \cap \Omega} \LC \mathcal{K}_{11} - \mathcal{K}_{22} \RC\LC n, y, z, \eta \RC  \omega_n(y) \omega_n(z) dy dz dx  .
\end{align*}
Since $\LC \mathcal{K}_{11} - \mathcal{K}_{22} \RC\LC n, y, z, \eta \RC$ is continuous in each variable and bounded in the first
 region then that term converges to 
\[
 - \frac1\pi\int \int \int_{\{|y -  z| \geq r\} \cap \Omega} \LC \mathcal{K}_{11} - \mathcal{K}_{22} \RC\LC \infty, y, z, \eta \RC  \omega(y) \omega(z) dy dz dx.
\]
On the other hand in the second region we have 
\begin{align*}
& \LV \frac1\pi\int \int \int_{\{0<|y -  z| < r\} \cap \Omega} \LC \mathcal{K}_{11} - \mathcal{K}_{22} \RC\LC n, y, z, \eta \RC  \omega(y) \omega(z) dy dz dx \RV \\
 & \leq C\int \int_{\{0< |y -  z| < r\} \cap \Omega} \omega_n(y) \omega_n(z) dy dz  \\
 & \lesssim \LN \omega_n \RN_{\mathcal{M}(\Omega)} \int_{\{|z| < r\} \cap \Omega} \omega_n (z) dz  \\
 & \lesssim M_r(\omega_n)
\end{align*}

and by Lemma~\ref{maximalvorticityestimate} the term goes to zero as $n\to \infty$ and $r\to 0$.  This implies $A_n^1 \to A^1$.  The convergence of $A_n^2$ is much easier since the kernal is continuous on the entire domain.  

Next, we show that $B_n \to 0$, and here we crucially use the compact support of the our test function $\chi$.   $B_n$ consists of three terms, depending on where the derivatives hit.  We consider the worst case in which all derivatives hit $H_n$.  Using \eqref{Hboundcompact} we get 
\begin{align*}
& {1\over n \pi} \LV \int \int \p_\ell \chi \p_m \p_m \p_\ell H_n(x,y) N_n(x,y) \omega_n(y) dy dx \RV \\
& \lesssim {1\over n} \LN \omega_n \RN_{\mathcal{M}}  \to 0
\end{align*}
as $n \to \infty$.  The rest of the terms of $B_n$ are estimated in a similar fashion.  

Finally, we can prove the estimate on the kinetic energy in the fashion of Liu-Xin \cite{LiuXin}.   As in \cite{LiuXin} one can
 use the decay of $M_r(\omega) \to 0$ to prove that 
\begin{equation}\label{limitingomegakineticenergyterm}
-\int\int_{ \{ |y - z| \leq r \} \cap \Omega}  \log|y -z| \omega(y) \omega(z) dy dz \lesssim 1.
\end{equation}
Then for $K$ a compact set in $\Omega$ take  a  nonnegative test function $\chi \in C_0^\infty(\Omega)$ with $\chi = 1$ on $K$.  Then 
\begin{align*}
\int_K v^2 & \leq \int \chi v^2  = 4\int \int \LC \mathcal{K}_{11} +  \mathcal{K}_{22} \RC \omega(y) \omega(z) dy dz \\
& = 4\int\int _{ \{ |y-z| < r \} \cap \Omega} \LC \mathcal{K}_{11} +  \mathcal{K}_{22} \RC \omega(y) \omega(z) dy dz \\
& \quad + 4\int\int _{ \{ |y-z| \geq r \} \cap \Omega} \LC \mathcal{K}_{11} +  \mathcal{K}_{22} \RC \omega(y) \omega(z) dy dz \\
& = A + B.
\end{align*}
Since $B$ is away from the singularity, then we see immediately that $B$ is bounded.  The bound on $A$ follows  from \eqref{N1N1N2N2Log} and 
\eqref{limitingomegakineticenergyterm}.


\end{proof}

%

We are now in position to establish the hydrodynamic limit.  The primary task is to approximate the initial data in a suitable way by quantized vortices that satisfy a good energy bound.  Then we can use Proposition~\ref{propODELimit}.

\begin{proof}[Proof of Theorem~\ref{thm:hydrodynamic}]
We first approximate initial data for $0 \leq \omega_0 \in \mathcal{M}\cap \dot{H}^{-1}(\Omega)$ in a suitable way so that we can use both Theorem~\ref{thm:quantdynamics} and Proposition~\ref{propODELimit}.   

1. Assume 
$\operatorname{supp} \omega_0 \subset \widetilde{\Omega}$ 
with $\dist(\widetilde{\Omega}, \p \Omega)\geq C  > 0$.  We then cover our set $\Omega$ with nonoverlapping squares $\{Q_j\}$, where
\[
{Q_j} \equiv j'\hbox{th square of side-length } h ,
\]
so there exist $O(h^{-2})$ squares $Q_j$ that cover $\Omega$.   We then set
\begin{equation} \label{hrate}
h = n^{-{1\over4}}
\end{equation}
so $h^{-2} \ll n$.
We now define
\[
\omega_{0}^{n} = \sum_{Q_j} {\omega_{0,j}^{n}}
\]
where  the $\omega_{0,j}^{n}$ are set below.  
Next, set
\[
\widetilde{n}^h_j = { n \over 2 \pi}   \int_{Q_j} \omega_0 
\]
and  ${n_{j}} = \lfloor \widetilde{n}^h_j \rfloor $, then $| {n_{j}} - \widetilde{n}^h_j | < 1$ and 
\begin{equation} \label{errorbounddegreesquares}
\LV \sum_j {n_{j}} - n \RV \lesssim h^{-2} = n^{1 \over2} . 
\end{equation}
Since $\omega_0$ has compact support, then for all $h \leq h_0 = h_0(\Omega)$ small enough, if $Q_k \cap \p \Omega \neq \varnothing$ then  $n_j= 0$.  If we set $\widehat{n} = \sum_j n_j$  then 
\[
n - C n^{1\over 2} \leq \widehat{n} \leq n
\]
so $\widehat{n} \to \infty$ in the same rate as $n \to \infty$.  We can then use $\widehat{n}$ instead of $n$ in the discussion below; however, we relabel $\widehat{n}$ as $n$ for simplicity.

Next, we slice $Q_j$ into $n_j$ thin rectangles of equal width.  They will be aligned vertically and horizontally in alternating sequence, see Figure~\ref{alignmentofrectangles}. 
\begin{figure}[htbp] 
   \centering
   \includegraphics[width=3in]{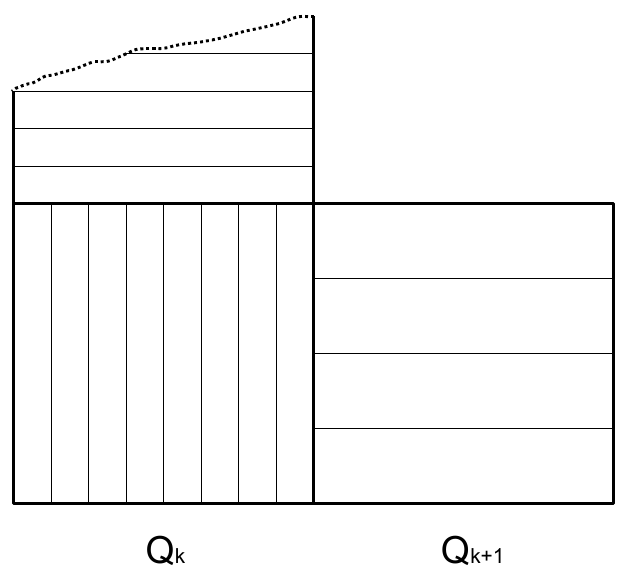} 
   \caption{Construction of the rectangles in the $Q_k$'s}
   \label{alignmentofrectangles}
\end{figure}
 In the center of each of these subrectangles
we label points $\{a_{0,j}^1, \ldots, a_{0,j}^{ {n_{j}}} \}$, so the distance between neighboring points is ${h \over n_j}$.   Finally, we let 
\begin{equation*} \label{subsquareomegan}
\omega_{0,j}^{n} =  {1\over n }\sum_{k=1}^{{n_{j}}} \delta_{a_{0,j}^k}
\end{equation*}
where the $a_{0,j}^k$ are defined above.  In the worst-case scenario all vortices are located in a single cell with intervortex distance $O({h \over n} ) \approx n^{-{5 \over 4 }}$, and we will need to check that this conforms to the correct bound on  $\rho_{a(0)}$.

We claim that $\omega_0^n \to \omega_0$ in $\mathcal{M}(\Omega)$.  Let  $f_U$ denote the average of $f$ on $U$.   Then for $\chi \in C_0^0(\Omega)$, $\LV \int_{\Omega} \chi \LC \omega_0^n - \omega_0 \RC \RV \leq \sum_{Q_j} \LV \int_{Q_j } (\chi - \chi_{Q_j}) (\omega^n_0 - \omega_0) \RV + \sum_{Q_j} \LV   \chi_{Q_j} \RV \LV \int_{Q_j }  \omega_0^n -\omega_0 \RV \to 0$ as $n \to \infty$ from \eqref{hrate}, \eqref{errorbounddegreesquares}, and the continuity of $\chi$.  Therefore, $\omega_n^0 \to \omega_0$ in $\mathcal{M}(\Omega)$.  

2. Finally, we claim that
\begin{equation}  \label{approxenergysequencebound}
-{1\over n^2} \sum_{a_{0,i}^j \neq a_{0,k}^\ell} N_n (a_{0,i}^j,a_{0,k}^\ell) \lesssim 1.
\end{equation}
Since the support of $\omega_0$ lies in a compact set away from the boundary, then $$\min\{ \operatorname{dist}(a_{0,j}^k, \p \Omega) \} \geq C > 0$$ uniformly in $n$.  Hence, we have $| H_n(a_{0,j}^k, a_{0,i}^\ell)| \leq C$ uniformly in $n$.  In particular, 
to establish \eqref{approxenergysequencebound} it is sufficient to prove
\begin{equation*}  \label{approxenergysequenceboundlog}
-{1\over n^2} \sum_{a_{0,i}^j \neq a_{0,k}^\ell} \log |a_{0,i}^j - a_{0,k}^\ell| \lesssim 1.
\end{equation*}
We subdivide the sum into those vortex interactions arising from the same $Q_j$'s and those that arise from differing $Q_k$'s,
\begin{align*}
-{1\over n^2} \sum_{a_{0,i}^j \neq a_{0,k}^\ell} \log |a_{0,i}^j - a_{0,k}^\ell|
& = -{1\over n^2} \sum_j \sum_{k \neq \ell} \log |a_{0,j}^k - a_{0,j}^\ell |
 -{1\over n^2} \sum_{j \neq k} \sum_{i, \ell} \log |a_{0,j}^i - a_{0,k}^\ell | \\
 & = A + B
\end{align*}
We now consider the sum $A$.  Concentrating on a single $Q_j$, assume without loss of generality that the subrectangles are vertical and $a_{0,j}^1$ is located at the origin.  Then the vortices in this square are located along the $x$-axis with $x$ values at $\{0, \Delta, 2\Delta, \ldots, (n_j-1) \Delta \}$, where $\Delta = {h \over n_j}$.   Summing over the log interactions yields
\begin{align*}
- \sum_{k \neq \ell} \log |a_{0,j}^k - a_{0,j}^\ell | & 
= - \LB  (n_j-1) \log | \Delta| + (n_j-2) \log| 2 \Delta| + \cdots +  \log | (n_j-1) \Delta| \RB \\
& \leq {n_j (n_j-1) \over 2} \log \Delta^{-1} \leq {n^2_j \over 2} \log h^{-1} + {n_j^2 \over 2} \log n \\
& \leq 3 n_j^2 \log h^{-1},
\end{align*}
since $\log h^{-1} = {1\over 4} \log n$.  Now summing over the $j$'s yields and using that ${n_j\over n} \leq {1\over 2\pi}\int_{Q_j} \omega_0$, we get 
\begin{align*}
A & = - {1\over n^2} \sum_j \sum_{k \neq \ell} \log |a_{0,j}^k - a_{0,j}^\ell |  \lesssim {1\over n^2} \sum_j   n_j^2 \log h^{-1} \\
& \lesssim \sum_j \log h^{-1} \int_{Q_j} \omega_0(y) \int_{Q_j} \omega_0(z) \\
& \lesssim - \sum_j\int_{Q_j} \int_{Q_j} \omega_0(y) \log |y - z| \omega_0(z) dy dz
\end{align*}

Next we bound $B$.  Let $p_j$ denote the center of the square $Q_j$.   Due to the alternating alignment of the subrectangles in Figure~\ref{alignmentofrectangles}, we see that $ \LV \log | a_{0,j}^i - a_{0,k}^\ell | -  \log | p_j - p_k| \RV \leq C$, even for neighboring squares.  Therefore,
\begin{align*}
 B& = -{1\over n^2} \sum_{j \neq k} \sum_{i, \ell} \log |a_{0,j}^i - a_{0,k}^\ell |\\
 & =  -{1\over n^2} \sum_{j \neq k}  \sum_{i = 1}^{n_j} \sum_{\ell = 1}^{n_k} \log |a_{0,j}^i - a_{0,k}^\ell | \\
 & \leq  C - {C\over n^2} \sum_{j \neq k}  \sum_{i = 1}^{n_j} \sum_{\ell = 1}^{n_k} \log |p_i - p_k|  
  = C  - {C\over n^2} \sum_{j \neq k}  {n_j} {n_k} \log |p_i - p_k|    \\
 & \lesssim 1 - \sum_{j\neq k} \log |p_i - p_k|   \int_{Q_j} \omega_0(y)\int_{Q_k} \omega_0(z) \\
 & \lesssim 1 - \sum_{j\neq k} \int_{Q_j} \int_{Q_k} \omega_0(y) \log |y - z| \omega_0(z) dy dz
\end{align*}
Combining $A$ and $B$ together we find
\begin{align*}
A + B & \lesssim 1 -  \int_{\Omega} \int_\Omega \omega_0(y) \log |y - z| \omega_0(z) dy dz \\
& \lesssim 1 +   \LN \omega_0 \RN_{\dot{H}^{-1}(\Omega)} \LN \mu_{\operatorname{supp}(\omega_0)}(y)  \int \log|y-z|  \omega_0 (z)\RN_{H^1(\Omega)} \\
& \lesssim 1 + \LN \omega_0 \RN^2_{\dot{H}^{-1}} \lesssim 1,
\end{align*}
where $\mu_Q$ is the characteristic function on $Q$.


3.  Now we complete  the proof of the hydrodynamic limit.  
Set $\e_n$ such that $n =  |\log | \log | \log \e_n| | |^{1 \over 4} $ and 
\[
\tilde\omega_n (t) = { 1\over n} {e_{\e_n}(u_{\e_n}(t)) \over \pi | \log \e_n|}.
\]
Given the initial measure $\omega_0$, we build our initial data $u_{\e_n}(0)$ with vortices at $\{a_{0,j}^k\}$ as generated above, and 
satisfying the hypotheses of Theorem~\ref{thm:quantdynamics}.  Such data can be constructed following Lemma 14 of \cite{JSp2}.
Then since the energy is decreasing in time and using 
\eqref{Energytotalbound}, we obtain for a subsequence that
 $\tilde\omega_n\to\tilde\omega$
in $\mathcal{M}(\Omega\times[0,\infty))$.
Furthermore, the intervortex distance is no worse than
$$\rho_{a(0)} \geq C {h \over n} \geq C n^{- {3 \over2}} \geq C \LV \log \LV \log \LV \log \e_n \RV \RV \RV^{-{3 \over 8}}
\geq \LV \log \LV \log \LV \log \e_n \RV \RV \RV^{- {1\over 3}};
$$
therefore, both $n$ and $\rho_{a(0)}$  satisfy the requirements of Proposition~\ref{lowerboundrhostardirichlet}.  

From Proposition~\ref{propODELimit} we obtain that 
$\omega_n=\frac1n\sum \delta_{a_j(t)}$ converges 
to some $\omega$ that is an interior 
weak solution of \eqref{meanfieldequation}.
By Theorem~\ref{thm:quantdynamics}, we see that
$\omega_n-\tilde \omega_n\to 0$ in distribution,
and so $\tilde\omega=\omega$ also solves \eqref{meanfieldequation}.


\end{proof}

\begin{proof}[Proof of Theorem~\ref{lowerboundrhostardirichletconverse}] 

The proof of Theorem~\ref{lowerboundrhostardirichletconverse} follows along the same lines as the proof of 
Theorem~\ref{thm:hydrodynamic}.  In particular we use assumptions \eqref{hyp1}-\eqref{hyp4} in order to satisfy the hypotheses 
of Theorem~\ref{thm:quantdynamics}.  Next, assumptions \eqref{hyp3}-\eqref{hyp4} ensure the long-time existence of the vortex dynamics via Proposition~\ref{lowerboundrhostardirichlet}.  Finally, assumption \eqref{hyp5} allows us to use the ODE to PDE result, Proposition~\ref{propODELimit}.  The proof follows.

%
%

\end{proof}

 \end{document}